\documentclass[11pt]{amsart}
\usepackage{amscd} 
\usepackage[all,cmtip]{xy}               

\numberwithin{equation}{subsection}
\usepackage{xcolor}

\usepackage{cancel}

\usepackage{graphicx}
\usepackage{amssymb}
\usepackage{MnSymbol}

\DeclareMathAlphabet{\mathpzc}{OT1}{pzc}{m}{it}

\newtheorem{proposition}{Proposition}[section]
\newtheorem{definition}[proposition]{Definition}
\newtheorem{lemma}[proposition]{Lemma}
\newtheorem{theorem}[proposition]{Theorem}
\newtheorem{corollary}[proposition]{Corollary}
\newtheorem{conjecture}[proposition]{Conjecture}

\newcommand{\C}{\mathbb{C}}
\newcommand{\CC}{\mathbf{C}}

\DeclareMathOperator{\reach}{reach}
\DeclareMathOperator{\glob}{glob}
\DeclareMathOperator{\Glob}{Glob}
\DeclareMathOperator{\Int}{integ}
\DeclareMathOperator{\image}{image}
\DeclareMathOperator{\sgn}{sgn}
\DeclareMathOperator{\kl}{Kl}
\DeclareMathOperator{\Val}{Val}
\DeclareMathOperator{\V}{\mathcal V}
\DeclareMathOperator{\Gr}{Gr}
\DeclareMathOperator{\AGr}{\overline{ Gr}}

\DeclareMathOperator{\CP}{\C P}
\newcommand{\CPn}{\CP^n}

\DeclareMathOperator{\CH}{\C H}
\DeclareMathOperator{\ch}{ch}

\DeclareMathOperator{\mass}{\operatorname {mass}}
\DeclareMathOperator{\Tan}{\operatorname {Tan}}
\DeclareMathOperator{\Sym}{\operatorname {Sym}}

\DeclareMathOperator{\nor}{\operatorname{nor}}

\DeclareMathOperator{\ang}{\operatorname{Ang}}
\DeclareMathOperator{\nul}{\operatorname{Null}}
\DeclareMathOperator{\sn}{\operatorname{sn}}
\DeclareMathOperator{\cs}{\operatorname{cs}}
\DeclareMathOperator{\ct}{\operatorname{ct}}
\DeclareMathOperator{\tn}{\operatorname{tn}}
\DeclareMathOperator{\pd}{pd}
\DeclareMathOperator{\Ch}{Ch}
\DeclareMathOperator{\id}{id}
\newcommand\vol{\operatorname{vol}}

\newcommand \valun{\Val^{U(n)}}

\newcommand\Hom{\mathbf{Hom}}

\newcommand{\curvuinf}{\mathrm{Curv }^{U(\infty)}}
\newcommand{\curvinfty}{\mathcal {C}}
\newcommand{\Cn}{\mathbb{C}^n}

\newcommand{\curv}{\mathrm{Curv }}
\newcommand{\curvun}{{\mathrm{Curv}}^{U(n)}}
\newcommand{\val}{\mathrm{Val}}

\newcommand{\restrict}[2]{\left. #1 \right|_{#2}}
\newcommand{\R}{\mathbb{R}}
\newcommand{\Rn}{\mathbb{R}^n}

\newcommand{\K}{\mathcal{K}}
\newcommand{\Ksm}{\mathcal{K}^{sm}}

\newcommand{\barun}{\overline{U(n)}}

\newcommand{\baron}{\overline{O(n)}}
\newcommand{\eps}{\epsilon}

\newcommand{\area}{\operatorname{area}}
\newcommand\pf{\operatorname{Pf}}
\newcommand\spann{\operatorname{span}}

\newcommand{\en}{\mathfrak{n}}
\newcommand{\el}{\mathfrak{l}}

\newcommand\piinv{\pi^{-1}}
 
\newcommand \bargo{\overline {G_o}}

\newcommand \cmeas{\operatorname{\C\mathbf {Meas}}}

\newcommand\manyderivs{\left.\frac{\partial^k}{\partial \lambda_1 \cdots \partial \lambda_k}\right|_{\lambda=0}}
\title{ Integral geometry of complex space forms}
\author{Andreas Bernig, Joseph H.G. Fu, and Gil Solanes}
\thanks{ AB supported by DFG grants BE 2484/3-1 and BE 2484/5-1. JHGF supported by NSF grant DMS-1007580. GS supported by FEDER/MEC grant MTM2009-07594.}

\begin{document}
\begin{abstract}
We show how Alesker's theory of  valuations on manifolds gives rise to an algebraic picture of the integral geometry of any Riemannian isotropic space.
We then apply this method to give a thorough account of the integral geometry of the complex space forms, i.e. complex projective space, complex hyperbolic space and complex euclidean space. In particular, we compute the family of kinematic formulas for invariant valuations and invariant curvature measures in these spaces. In addition to new and more efficient  framings of the tube formulas of  Gray and the kinematic formulas of  Shifrin, this approach yields a new formula expressing the volumes of the tubes about a totally real submanifold in terms of its intrinsic Riemannian structure. We also show by direct calculation that the Lipschitz-Killing valuations stabilize the subspace of  invariant angular curvature measures, suggesting the possibility that a similar phenomenon holds for all Riemannian manifolds. We conclude with a number of open questions and conjectures. 
\end{abstract}

\email{bernig@math.uni-frankfurt.de}
\email{fu@math.uga.edu}
\email{solanes@mat.uab.cat}

\address{Institut f\"ur Mathematik, Goethe-Universit\"at Frankfurt,
Robert-Mayer-Str. 10, 60054 Frankfurt, Germany}
\address{ Department of Mathematics, 
University of Georgia, 
Athens, GA 30602, USA}
\address{Departament de Matem\`atiques, Universitat Aut\`onoma de Barcelona, 08193 Bellaterra, Spain}

\maketitle
\tableofcontents
\setcounter{tocdepth}{200}

\section{Introduction} 
\subsection{Classical integral geometry} As originally proposed by Blaschke \cite{blaschke}, the subject of integral
geometry is essentially the study of  {\it kinematic formulas}, the most famous being the euclidean {\it principal
kinematic formula} 
\begin{equation}\label{eq:pkf1}
\int_{\baron} \chi(A\cap gB) \, dg = \omega^{-1}_n\sum_{i+j = n} \binom n i^{-1}\omega_i \omega_j \mu_i(A) \mu_j(B),
\end{equation}
which holds for all sufficiently nice compact subspaces $A,B \subset \Rn$, where $\baron :=O(n) \ltimes \Rn$ is the
euclidean group. Here $\chi$ is the Euler characteristic and $\omega_k$ is the volume of the unit ball of dimension $k$.
The $\mu_k$ are the {\it intrinsic volumes}, which may be expressed as certain integrals of curvature of $A,B$ if these
bodies have smooth boundary. In particular, $\mu_0=\chi$ and $\mu_n$ is the volume. Hadwiger proved that the $\mu_k$
span the space of  continuous $\baron$-invariant {\it convex valuations} (i.e. finitely additive functionals on convex sets). For each $k$ there is a formula of this type
with $\chi$ replaced by $\mu_k$ in the integrand.  Similar formulas exist also in the real space forms of nonzero
curvature, i.e. in the sphere and in hyperbolic space \cite{santalo}.  

One special case of obvious interest occurs when $B$ is taken to be a geodesic ball of small radius $r$. In this case
the left hand side of \eqref{eq:pkf1} may be interpreted as the volume of the tubular neighborhood of radius $r$ about
$A$. This yields H. Weyl's famous tube formula \cite{weyl}.

Subsequently Federer \cite{cm} showed that each $\mu_k$ admits a localization:  for $A\subset \Rn$ as above there exists
a signed {\it curvature measure} $\Phi_k(A,\cdot)$ such that
$$
\Phi_k(A,A) = \mu_k(A), \quad k = 0,\dots,n.
$$
Again, when $A$ is smooth the quantity $\Phi_k(A,U)$ is an integral of curvature over  $\partial A\cap U$.
Furthermore, there are local versions of the global kinematic formulas described above, e.g.
\begin{equation}\label{eq:local pkf1}
\int_{\baron} \Phi_0(A\cap gB,U \cap gV) \, dg = \omega^{-1}_n\sum_{i+j = n} \binom n i^{-1}\omega_i \omega_j
\Phi_i(A,U) \Phi_j(B,V)
\end{equation}
for any open sets $U,V$.

Several authors have given formulas of this type for subspaces of
 the complex space forms, i.e. complex euclidean space $\C^n$,  complex projective space and complex hyperbolic space:
Abbena, Gray and Vanhecke \cite{abbena et
al}; Gray and Vanhecke \cite{grayvan}; Gray \cite{gray85,gray}; Shifrin \cite{shifrin81,shifrin84}; and Tasaki
\cite{tasaki00,tasaki03}. The papers \cite{abbena et al} and \cite{grayvan} gave a local formula for the volumes of
tubes about smooth submanifolds $M$ of a complex space form, 
 and \cite{gray85, gray} expressed the local tube formula in terms of familiar geometric quantities when $M$ is a smooth
complex analytic submanifold.
 
  Meanwhile,  \cite{shifrin81, shifrin84} studied kinematic formulas in complex projective spaces at both the local and
the global level, but only in the case where both subspaces $A,B$ are complex analytic. Just as in the case of the real
space forms, the local and the global versions of the formulas that emerge are again formally identical. The same happens if $A,B$ are real submanifolds of complementary dimensions. This case was treated in \cite{befu11, tasaki00, tasaki03}.

\subsection{Integral geometry and valuations} More recently, Alesker introduced a range of new algebraic structures on
the space $\Val(V)$ of continuous translation-invariant convex valuations on a general finite-dimensional real vector
space $V$ \cite{ale01, ale03, ale04}. In particular, if $V$ is euclidean and $H$ is a compact subgroup of the orthogonal
group $O(V)$ acting transitively on the unit sphere of $V$, then the space $\Val^H(V)$ of $H$-invariant elements of
$\Val(V)$ is finite-dimensional. Such groups $H$ have been classified \cite{borel49, mont-sam}. Furthermore for each
$\mu \in \Val^H(V)$ there exists a kinematic formula for $\mu$, analogous to \eqref{eq:pkf1}. Using this approach, the
integral geometry of the spaces $(V,H)= (\Cn,U(n)),(\Cn,{SU(n)}), (\R^7,{G_2}), (\R^8, {Spin(7)})$ has been successfully
analyzed in
\cite{be09,be11, be12, befu11}.

Subsequently, Alesker introduced a theory of {\it valuations on manifolds} \cite{ale04, ale05a, ale05b, ale05d, ale10,
alefu05}. The basic theme of the present paper is to exploit this theory to give a natural and penetrating approach to integral geometry in curved manifolds.  An element $\phi
$ belonging to the space $\V(M)$ of valuations on a smooth manifold $M$ is a real-valued functional $\phi =
[[\alpha,\beta]]$ on sufficiently smooth subsets $P\subset M$, given as 
$$
\phi(P) = \int_P \alpha + \int_{N(P)} \beta.
$$
Here $\alpha, \beta $ are smooth differential forms living on $M$ and the sphere bundle $SM$ respectively, and
$N(P)\subset SM$ is the space of outward normal vectors to $P$. 

Thus  the second term may be viewed as a total curvature integral.
 It is therefore natural to consider also the corresponding localized functional $\Phi = [\alpha,\beta]$ given by
$$
\Phi(P,U) = \int_{P\cap U} \alpha + \int_{N(P)\cap \pi^{-1}U} \beta
$$
for Borel subsets $U \subset P$, where $\pi:SM\to M$ is the projection. Such a localized functional is called a {\it
curvature measure}, and the vector space that they comprise is denoted by $\curvinfty(M)$. In general, the natural
projection $\curvinfty(M) \to \V(M)$ is surjective but not injective.

 Let now  $M$ be Riemannian and $G$ a Lie group of isometries. If $\Phi$ as above is invariant under the action of $G$,
then $\alpha,\beta$ may also be taken to be invariant.  It follows that the spaces $\curvinfty^G(M),\V^G(M)$ of
invariant curvature measures and  valuations are finite-dimensional precisely when $(M,G)$  is an {\it isotropic space},
i.e.  when $G$ acts transitively
on the unit tangent bundle $SM$ (\cite{ ale03, ale-be09}).
Such spaces have been classified
(\cite{borel49, helgason, mont-sam,  tits}): the list consists of all euclidean spaces $M=V$ under the action of $ G:= H
\ltimes T$, for groups $H$ as above, where $T$ is the group of  translations; together with the real, complex,
quaternionic and octonian space forms under their full isometry groups. The members of these families with nonzero
curvature coincide with the Riemannian symmetric spaces of rank one.

 In this setting there are kinematic formulas for $(M,G)$ analogous to \eqref{eq:pkf1} and \eqref{eq:local pkf1}
(cf.\cite{fu90}). They are encoded by the {\it kinematic operators}
\begin{align*}
K_G&: \curvinfty^G(M) \to \curvinfty^G(M) \otimes \curvinfty^G(M) \\
k_G&:\V^G(M) \to \V^G(M) \otimes \V^G(M).
\end{align*}
We call these the {\it local} and {\it global} kinematic operators respectively.
The latter is the image of the former under the natural projection $\curvinfty^G(M) \to \V^G(M)$.

From this perspective, the classical cases described above arise by taking $M$ to be a general real space form and $G$
to be the full group of isometries of $M$. It is natural to consider all of these spaces together, as a one-parameter
family indexed by the curvature $\lambda \in \R$--- in fact, they are all isomorphic in the sense of Theorem
\ref{thm_transfer} below. The subject is further simplified by the fact that the projection from invariant curvature
measures to invariant valuations is bijective. It follows that
the kinematic formulas at the local and global levels are formally identical.

The case of the complex space forms is much more complicated. The paper \cite{befu11} applies the Alesker theory of
valuations to study the curvature zero case, i.e. $\Cn$ under the action of the hermitian isometry group $\barun:=U(n)
\ltimes \Cn$, and computes the global kinematic operator $ k_{\overline{U(n)}}$. However, the map from invariant
curvature measures to invariant valuations is {\it not} injective, so this gives only partial information about the
local kinematic operator $ K_{\overline{U(n)}}$. Due to the same difficulty, in the complex space forms of non-zero
curvature even the  global kinematic formulas were unknown in general. In complex dimensions $n=2,3$ the principal
kinematic formula $k_G(\chi)$ was found in \cite{pa02}. In higher dimensions, a very special case of this formula was
calculated in \cite{ags}.

\subsection{Results of the present paper}
Here we build on the approaches of \cite {ags, befu11} to compute both the global and the local kinematic operators for
all complex space forms. We give also several geometric consequences, recovering and extending all of the classically
known results.

\subsubsection{General theory} In Section \ref{sect:basics} we give an account of the integral geometry of
general isotropic spaces from the valuations perspective. Our approach is new, and several of the facts and definitions
presented here (in particular, the Lipschitz-Killing algebra and the concept of angularity) have not appeared
previously in the literature. 

We define (Definition \ref{def:curv & val} and the Remarks following) the space $\curvinfty(M)$ of smooth curvature measures on a smooth manifold $M$,
and the globalization map $\curvinfty(M)\to \V(M)$ {to} the space of smooth valuations on $M$. We also consider the
spaces $\curv(V)$ and $\Val(V)$ of translation-invariant curvature  measures and valuations on a finite-dimensional real
vector space $V$; if $V=T_xM$, the tangent space to $M$ at a point $x$, these serve as infinitesimal versions of general
elements of $\curvinfty(M), \V(M)$. We take note of the Alesker product on $\V(M)$, and show (Proposition \ref{prop_def_module}) that
$\curvinfty(M)$ is  a module over $\V(M)$. This new algebraic structure turns out to be crucial in
the study of local kinematic formulas, and also in some of the global results.

If $M$ is Riemannian, then there is a natural identification (Proposition \ref{prop_isom_curv_general}) of the curvature measures on $M$ with the sections of the vector bundle of translation-invariant curvature measures on the tangent spaces $T_xM$. We expound (Section \ref{sect:LK algebra}) on Alesker's observation that any Riemannian manifold admits a canonical {\it Lipschitz-Killing subalgebra} 
$LK(M) \subset \V(M)$, consisting of valuations given by integrating the classical Lipschitz-Killing curvatures. They
correspond bijectively (Section \ref{sect:LK cms}) to a subspace of {\it Lipschitz-Killing curvature measures}, distinguished by the natural
geometric property of {\it angularity} (Section \ref{sect:angular cm}). This property is key to the initial calculations of Section \ref{sect:t & tau}.

 If $M$ is isotropic under the action of the group $G$, we define (Theorem \ref{thm:basic K}, Section \ref{sect:global and semilocal}) the local and global kinematic operators $K_G,k_G$,
for the spaces $\curvinfty^G(M),\V^G(M)$. These operators may be regarded as endowing each  space with the structure of
a cocommutative and coassociative coalgebra. We give a precise version (Theorem \ref{thm_transfer}) of Howard's Transfer Principle \cite{howard},
showing that the coalgebra $\curvinfty^G(M)$ is canonically isomorphic to the corresponding coalgebra of invariant
elements of $\curv(T_xM)$ for a representative point $x$. We also give a precise account (Section \ref{sect:an cont 1}) of how  integral geometric
relations in positively curved isotropic spaces $M$  may be analytically continued to yield a family of relations valid
also for negatively curved spaces in the same family (or vice versa).
 We prove a new version (Theorem \ref{thm_ftaig}), applicable (in addition to the known compact and euclidean cases) to hyperbolic spaces, of what
may be regarded as the {\it fundamental theorem of algebraic integral geometry}, which states that $k_G$ is adjoint to
the restricted Alesker multiplication map. We also define (Section \ref{sect:global and semilocal}) the {\it semi-local} kinematic operator $\bar k_G$ which is
similarly related to the structure of $\curvinfty^G(M)$ as a module over $\V^G(M)$.

\subsubsection{Complex space forms} In Section \ref{sect:global} we begin our investigation of the complex space forms.
It is natural to view these spaces as comprising a family $\{\CP^n_\lambda\}_{\lambda \in \R,n \in \mathbb N} $, where
$n$ is the complex dimension and $\lambda$ is one-fourth of the holomorphic sectional curvature. The case $\lambda =0$
reduces to the case $(M,G) = (\Cn, \barun)$ studied in \cite{befu11}.

 We first recall (Section \ref{subsect:invariant}) the construction of the invariant curvature measures from \cite{befu11, pa02}. By the
Transfer Principle, we may canonically identify the coalgebra of invariant curvature measures on any $\CP^n_\lambda$
with the special case $\lambda = 0$, i.e. the space $\curvun$ of $\barun$-invariant curvature measures on $\CP^n_0 =
\Cn$. Following \cite{ags, pa02} we describe (Proposition \ref{prop_kernels_glob}) the kernels of the globalization maps to the corresponding algebras
$\V^n_\lambda:= \V^G(\CPn_\lambda)$ of invariant valuations. We show (Proposition \ref{prop:generators}) that  $\V^n_\lambda$ is generated by the
Lipschitz-Killing subalgebra $LK(\CP^n_\lambda) \simeq\R[t]/(t^{2n+1})$ together with the valuation $s$ given by the
integral of the Euler characteristic of  intersections with generic totally geodesic  complex hypersurfaces, and show (Proposition \ref{prop_mu2st})
how to express polynomials in $s,t$ as globalizations of elements of $\curvun$.  
Our main result in this section is an explicit algebra isomorphism $I_\lambda:\V^n_0 \to \V^n_\lambda$ (Theorem \ref{thm:1st iso}),
which yields an expression (Theorem \ref{thm:pkf lambda}) for the principal kinematic formula that is formally independent of the curvature $\lambda$. This determines the global kinematic operators $k_\lambda$ completely.
We give also a family of additional isomorphisms (Theorem \ref{thm_other_isom}), including two particularly striking examples, although we do not use
them elsewhere in this paper.

Section \ref{section: tube formulas} is devoted mainly to the tube formula, which results from the evaluation of the
principal kinematic formula on geodesic balls. Remarkably, the resulting global tube formula (Theorem \ref{thm:general tube}) turns out to be formally
parallel to the corresponding formula in $\Cn$ in a way that is not apparent in the formulations of Gray. In the
context of a general Riemannian manifold $M$ we show (Theorem \ref{thm:local tube}) how to recover  local tube formulas from the global ones using the
first variation operator introduced in \cite{befu11} (note that these local tube formulas are {\it not} given directly
by the local or semi-local kinematic formulas); applying this to the case $M= \CP^n_\lambda$ we recover (Corollary \ref{cor:gray local}) the local tube
formulas of Gray and Vanhecke. We give a brief discussion of kinematic formulas as applied to complex analytic
submanifolds, recovering Shifrin's results and placing them in context (Theorem \ref{thm:kc gamma}, Corollary \ref{thm:kc gamma}).
We also give new global tube formulas for subspaces of totally real submanifolds $N \subset \CP^n_\lambda$ (i.e.
submanifolds that are isotropic with respect to the K\"ahler form):  these have a particularly simple form, expressible
solely in terms of the Lipschitz-Killing valuations of $N$ as a Riemannian submanifold (Theorem \ref{thm:real tube}).

In Section \ref{sect:module_cn}
we study $\curvun$ as a module over the $\V^n_0$. Our calculations turn on two key points: first, that the action of the
valuation $s \in \V^n_\lambda$ on $\curvun$ is formally independent of $\lambda$ (Proposition \ref{prop:s indep}); and second, that a certain natural
subspace $\mathrm{Beta} \subset \curvun$ is stabilized by this action (Proposition \ref{prop:sb sub b}). As a consequence of these computations, we show
that the $\V^n_0$-module $\curvun$ is spanned by two elements (Theorem \ref{thm_free_module}).

Finally, in Section \ref{sect:module cpn} we give in explicit form the local kinematic formulas for complex space
forms, which is one of the main results of this paper. These formulas contain more information than the global kinematic
formulas $k_\lambda$. 
To do so, we use the isomorphism $I_\lambda^{-1}:\V^n_\lambda \to \V^n_0$ of Section
\ref{sect:global} to embed the family of global kinematic operators $k_\lambda$ into a common model. Differentiating
with respect to $\lambda$, we extract enough information to determine the curvature-independent local kinematic operator
$K:\curvun \to \curvun\otimes \curvun$ (Theorem \ref{thm:local kin}). To
complete the picture, we also compute the structure of $\curvun$ as a module over $\V^n_\lambda$ (Theorem
\ref{thm_module_cpn}). As an application, we show by direct computation that the Lipschitz-Killing subalgebra in
$\V^n_\lambda$ stabilizes the
subspace of angular curvature measures in the Riemannian manifold $\CPn_\lambda$ (Theorem \ref{thm_angularity_thm2}).

Section \ref{sect:questions} is devoted to open problems. In addition there is an Appendix, in which we establish a
fundamental multiplication formula for smooth valuations in an isotropic space.

\subsection*{Acknowledgements} It is a pleasure to thank  J. Abardia, S. Alesker,  E. Gallego, F. Schuster and T.
Wannerer for their many helpful comments.
\section{Basics}\label{sect:basics}

\subsection{Notation}\label{notation}
\begin{align*}
\omega_n&:= \text{ the volume of the unit ball in } \Rn=\frac{\pi^\frac{n}{2}}{\Gamma\left(\frac{n}{2}+1\right)}\\
\alpha_n&:= (n+1) \omega_{n+1} = \text{ the volume of the unit sphere } S^n
\end{align*}

Given a subset $S$ of a vector space we denote by $\spann S$ its span.

The cosphere bundle of a smooth manifold $M$ is denoted by $S^*M$. The sphere bundle of a Riemannian
manifold
$M$ is denoted by $SM$. The projections $S^*M \to M, SM \to M$ are denoted by $\pi$, and we write $S_x^*M,
S_xM$ for the fibers. The canonical $1$-form on  {$S^*M$} is denoted by $\alpha$. The {\it Reeb vector field} is the
unique vector field $T$ with $i_T\alpha=1, \mathfrak{L}_T \alpha=0$, where $\mathfrak{L}$ denotes the Lie derivative.

The complex projective space of dimension $n$ and holomorphic sectional curvature $4\lambda>0$ is denoted by
$\CP^n_\lambda$. The K\"ahler form $\kappa$ and the volume then satisfy 
that 
\begin{displaymath}
 \int_{\CP^n_\lambda} \frac{\kappa^n}{n!}=\vol \CP^n_\lambda=\frac{\pi^n}{\lambda^n n!}.
\end{displaymath}
 
We will also use the notation $\CP^n_\lambda$ as a shorthand for $\C^n$ if $\lambda=0$ and for complex hyperbolic space $\CH^n$ (with
holomorphic sectional curvature $4\lambda$) if $\lambda<0$. 

For $\lambda\neq 0$ we let $G_\lambda$ be the isometry group of $\CP_\lambda^n$. Thus $G_\lambda \simeq$ the projective unitary group $PU(n+1)$
for
$\lambda >0$ and $G_\lambda \simeq PU(n,1)$ for $\lambda <0$. For $\lambda=0$ we set
$G_0=\overline{U(n)}=\C^n\ltimes U(n)$. The Haar measure on $G_\lambda$ is
normalized in such a way that the measure of the set of group elements mapping a fixed point to a compact
subset of $\CP^n_\lambda$
equals the volume of this subset. In particular, if $\lambda>0$, the total measure of the isometry group
equals the volume of $\CP^n_\lambda$. 

We will use the standard notation from differential geometry:
\begin{displaymath}
 \sn_\lambda(r):=\left\{\begin{array}{c c} \frac{1}{\sqrt{\lambda}} \sin(\sqrt{\lambda}r) & \lambda>0\\
r & \lambda=0\\ 
\frac{1}{\sqrt{|\lambda|}} \sinh(\sqrt{|\lambda|}r) & \lambda<0,
\end{array} \right.
\end{displaymath}
which is an analytic function in $(\lambda,r) $. Moreover, we will use 
\begin{displaymath}
 \cs_\lambda:= \sn_\lambda', 
\tn_\lambda:=\frac{\sn_\lambda}{\cs_\lambda}, \ct_\lambda:=\frac{\cs_\lambda}{\sn_\lambda}.
\end{displaymath}

We also use the standard notation
\[
 (2k+1)!!=(2k+1)\cdot(2k-1)\cdot(2k-3)\cdots 1
\]
and set formally $(-1)!!=1$.

\subsection{Curvature measures and valuations on manifolds}

Let $M$ be a smooth oriented manifold of dimension $m$.  Following \cite{ale05b}, we let $\mathcal{P}(M)$ denote the
space of  simple differentiable polyhedra of $M$ (which is the same as the space of compact smooth submanifolds with
corners). Given $A \in \mathcal{P}(M)$, the conormal cycle $N(A)$ is a Legendrian cycle in the cosphere bundle $S^*M$
(compare Def.~2.4.2 in \cite{ale05b} or \cite{fu94b}). 

\begin{definition}\label{def:curv & val}
Given differential forms $\omega \in \Omega^{m-1}(S^*M), \phi \in \Omega^m(M)$, the functional $\Phi=:\Int(\omega,\phi)$
that assigns to each differentiable polyhedron $A \subset M$ the signed {Borel} measure
\begin{displaymath}
{ M\supset\ } U\mapsto \Phi(A,U):=\int_{N(A) \cap \pi^{-1}U} \omega + \int_{A \cap U} \phi
\end{displaymath}
is called
 a {\bf smooth curvature measure} on $M$. The space of smooth curvature measures is denoted by $\curvinfty(M)$.
If $M$ is an affine space, we denote the subspace of translation-invariant curvature measures on $M$ by $\curv(M)$.

 Given differential forms $\omega \in \Omega^{m-1}(S^*M), \phi \in \Omega^m(M)$, the functional $\mu$ defined by 
\begin{displaymath}
 \mu(A):=\int_{N(A)} \omega + \int_A \phi, \quad A \in \mathcal{P}(M),
\end{displaymath}
is called a {\bf smooth valuation} on $M$. The space of smooth valuations is denoted by $\mathcal{V}(M)$.
If $M$ is an affine space, we denote the subspace of translation-invariant smooth valuations on $M$ by $\Val(M)$.
\end{definition}

{\bf Remarks.} 1. Thus a curvature measure {$\Phi$} is not actually a measure, but a function that assigns a signed
measure to
any sufficiently smooth subspace. {Note that we can evaluate {$\Phi(A,U)$ for}  non-compact subsets $A$
provided that the Borel set $U$ is relatively compact in $A$.}

2. A smooth curvature measure $\Phi\in \curvinfty(M)$ may be applied to any subset $X\subset M$ that is ``smooth enough"
to admit a normal cycle in the sense of \cite{fu94}. The class of such subsets is only very dimly understood, but
includes semiconvex sets ({also known as} sets with positive reach) and, if $M$ possesses a real analytic
structure, compact subanalytic sets, together with all generic intersections of such sets. We observe also that every
compact smooth polyhedron is semiconvex.

3. Our notation $\Val$ is a slight departure from the usual sense of $\Val$ as the space of continuous
translation-invariant convex valuations. However, the Irreducibility Theorem \cite{ale01} of Alesker implies that the
space we have defined is naturally identified with a dense subspace of this latter space.
For the general theory of smooth valuations we refer to \cite{ale05a,ale05b,ale05d,ale06,ale-be09,alefu05,bebr07}.

For $\Phi \in \curvinfty(M)$ we put
$$
[\Phi]:= \glob(\Phi):= [A \mapsto \Phi(A,A)] \in \V(M).
$$
Thus we have maps
$$
\begin{CD}
 \Omega^{m-1}(S^*M) \times\Omega^m(M) @>^{\Int}>> \curvinfty(M) @>^{\glob}>>\V(M).
\end{CD}
$$
{We put also
$$
 \Glob := \glob\circ \Int.
$$}

Recall that $S^*M$ is naturally a contact manifold, and admits a global contact form $\alpha \in \Omega^1(S^*M)$.

\begin{theorem}[\cite{bebr07}] \label{prop:kernel thm} Let $p \in M$, and put $D: \Omega^{n-1}(S^*M) \to \Omega^n(S^*M)$
for the Rumin differential (\cite{rumin94}). Then
\begin{align*}
\ker \Int &= \{(\omega,0): \omega = \alpha \wedge \psi + d\alpha \wedge \theta \text{ for some } \psi\in
\Omega^{n-2}(S^*M), \theta \in \Omega^{n-3}(S^*M)\}, \\
{\ker  \Glob } &= \{(\omega,\phi): D\omega + \pi^* \phi = 0, \ \int_{S_p^*M} \omega = 0\},
\end{align*}
where $\pi:S^*M\to M$ is the projection.
\end{theorem}

In the affine case  $M=V$, the space of translation invariant curvature measures $\curv(V)$ admits a natural
decomposition (analogous to McMullen's decomposition of $\val(V)$)
\begin{equation}\label{eq:curv decomp}
  \curv(V)=\bigoplus_{k=0}^n \curv_k(V)
\end{equation}
where $\curv_k(V)$ contains the $k$-homogeneous curvature measures $\Phi$; i.e. those $\Phi$ for which $\Phi(t A,t
U)=t^k\Phi(A,U)$ for $t>0$. This follows from Proposition \ref{prop:invariant curv} below: any translation-invariant
curvature measure may be expressed as $\Int(\omega,\phi)$ for some translation-invariant forms $\omega \in
\Omega^{n-1}(S^*V)^V$ and $\phi \in \Omega^n(V)^V$. The latter clearly induce the invariant curvature measures of degree
$n$.  The subspace of  translation-invariant  elements of the former space may be decomposed as 
\begin{equation}\label{eq:forms decomp}
\Omega^{n-1}(S^*V)^V=\bigoplus_{k=0}^{n-1} \Omega^{k}(V)^V \otimes \Omega^{n-k-1}(S(V))
\end{equation}
which clearly induces the remaining parts of the grading \eqref{eq:curv decomp}.

\subsubsection{Module structure} One of the main features of $\mathcal{V}(M)$ is the existence of
a canonical product structure, see \cite{alefu05} for its construction and \cite{ale-be09} for an alternative 
approach. Another salient feature is the existence of a natural decreasing filtration $\V(M) =W_0\supset
\dots\supset W_n \supset W_{n+1}=\{0\}$, where $n= \dim M$. The product respects the filtration, i.e. $W_i
\cdot W_j \subset W_{i+j}$. 

If $\Phi=\Int(\omega,\phi)$ is a smooth curvature measure on $M$ and $f \in C^\infty(M)$, then $\Phi_f(A):=\int_M f
d\Phi(A,-)$ is a smooth valuation, which is represented by the pair $(\pi^*f \wedge \omega,f\phi)$. Thus if $f \equiv 1$
then $\Phi_f = [\Phi]$.

\begin{proposition} \label{prop_def_module}
 Let $M$ be an oriented manifold, $\mu \in \mathcal{V}(M), \Phi \in \curvinfty(M)$. Then there is a unique smooth
curvature measure $\mu \cdot \Phi \in \curvinfty(M)$ such that for all $f \in C^\infty(M)$ we have 
\begin{equation} \label{eq_def_module}
 (\mu \cdot \Phi)_f=\mu \cdot \Phi_f.
\end{equation}
(The product on the right hand side is the product of smooth valuations on $M$). With respect to this multiplication,
$\curvinfty(M)$ is a module over ${\V}(M)$.
\end{proposition}

\proof
The proof is based on an expression for the product of two smooth valuations which was obtained in \cite[Theorem 2]{ale-be09}.

Let $\mu_i = {\Glob}(\omega_i,\phi_i), i = 1,2$. Then $\mu_1 \cdot \mu_2 = { \Glob}(\omega,\phi)$, with
\begin{equation} \label{eq_albe_formula}
 \omega = Q_1(\omega_1,\omega_2),\quad \phi = Q_2(\omega_1,\phi_1,\omega_2,\phi_2),
\end{equation}
where $Q_1$ and $Q_2$ are certain (explicitly given) operations on differential forms. It is easy to check that they
satisfy 
\begin{align} \label{eq_linearity_albe_formula1}
 Q_1(\pi^*f \wedge \omega_1,\omega_2) & =\pi^*f \wedge Q_1(\omega_1,\omega_2)\\
Q_2(\pi^*f \wedge \omega_1,f \phi_1,\omega_2,\phi_2) & = fQ_2(\omega_1,\phi_1,\omega_2,\phi_2), \quad f \in C^\infty(M).
\label{eq_linearity_albe_formula2}
\end{align}

Now if $\Phi = \Int(\omega_1,\phi_1)\in \curvinfty(M)$ then $\mu_2\cdot \Phi := \Int(\omega,\phi)$.  Thus
\eqref{eq_linearity_albe_formula1} and \eqref{eq_linearity_albe_formula2} imply \eqref{eq_def_module}. 

\endproof

If $\iota:N \to M$ is an immersion of smooth submanifolds, there exists a canonical restriction map
$\iota^*:\mathcal{V}(M) \to \mathcal{V}(N)$. The restriction is compatible with the product and the filtration. We refer
to \cite{ale10} for more details on this construction. 

\begin{proposition}
 Let $M$ be a smooth manifold and $\iota:N \to M$ a properly embedded smooth submanifold. Given $\Phi \in
\curvinfty(M)$, there is a unique curvature measure  $\iota^* \Phi \in \curvinfty(N)$ such that 
\begin{equation} \label{eq_restriction}
 (\iota^*\Phi)_{\iota^*f}=\iota^*(\Phi_f), \quad f \in C^\infty(M).
\end{equation}
$\iota^* \Phi$ is called the {\bf restriction} of $\Phi$, and is also denoted by  $\restrict\Phi N$.
\end{proposition}
\proof
Since every smooth function on $N$ can be extended to a smooth function on $M$, the uniqueness is clear. 

Let us prove  existence. If a smooth valuation $\mu$ on $M$ is given by a pair of forms $(\omega,\phi)$, then the
pull-back $\iota^*\mu$ is given by a pair of forms $\omega'=Q_3(\omega,\phi), \phi'=Q_4(\omega,\phi)$, where $Q_3,Q_4$
are certain Gelfand transforms with the property that   
\begin{align*}
 Q_3( \pi_M^* f\ \omega,f\phi) & =\pi_N^* \iota^*f Q_3(\omega,\phi),\\
 Q_4( \pi_M^* f\ \omega,f\phi) & =\iota^*f Q_4(\omega,\phi),  
\end{align*}
see \cite[Proposition 3.1.2]{ale10}. If the smooth curvature measure is represented by the pair $(\omega,\phi)$, then we define $\Psi$ as
the measure represented by $(Q_3(\omega,\phi),Q_4(\omega,\phi))$. Thus 
\eqref{eq_restriction} is satisfied. 
\endproof

\subsubsection{Curvature measures in a Riemannian manifold}
Let $M$ be a Riemannian manifold. We show that the space $\curvinfty(M)$ of smooth curvature measures on $M$ is
naturally identified with the space of smooth sections of the bundle $\curv(TM)$ of translation-invariant curvature
measures on the tangent spaces of $M$. 

Let $\xi  \in S M$, with $\pi \xi = x\in M$, where $\pi:SM\to M$ is the projection. 
The Riemannian connection induces a decomposition 
\begin{equation}\label{eq:decomp 1}
T_{\xi} SM = H \oplus V \simeq T_xM \oplus  \xi^\perp 
\end{equation}
into horizontal and vertical subspaces, where the derivative $D\pi$ induces an isomorphism $H \to T_xM$ and $V$ is naturally
isomorphic to $\xi^\perp \subset T_xM$. Identifying $H$ with $T_xM$, this decomposition may be refined further as 
\begin{equation}\label{eq:split line}
T_{\xi} SM \simeq \spann(\xi) \oplus \xi^\perp \oplus \xi^\perp.
\end{equation}

On the other hand, we may think of $\xi$ also as a point of  the sphere bundle $ST_xM$ of the tangent space $T_xM$,
projecting to $0 \in T_xM$. Since $ST_xM=T_xM \times S(T_xM)$ (where
$S(T_xM)$ is the unit sphere in $T_xM$), the tangent space at $\xi$ again splits as
\begin{equation}\label{eq:decomp 2}
T_{\xi} ST_xM = T_xM \oplus  \xi^\perp  \simeq \spann(\xi) \oplus \xi^\perp \oplus \xi^\perp.
\end{equation}
Thus we have a canonical identification $ T_{\xi} SM \simeq T_{\xi} ST_xM$, inducing
an isomorphism 
\begin{equation}\label{eq:lambda lambda}
\Lambda^{*}{ T^*_{\xi}} SM \simeq \Lambda^{*}{ T_\xi^*} S(T_xM).
\end{equation}

Fixing  $x\in M$ and letting $\xi \in S_xM$ vary, these identifications now yield a map 
$$
\restrict{\Omega^*(SM)}{S_xM} \to \Gamma(\restrict{\Lambda^*{ T^*}ST_xM}{S_xM})\simeq \Omega^*(ST_xM)^{T_xM}
$$
to the space of translation-invariant differential forms on the sphere bundle of $T_xM$; here the last isomorphism is
obtained by pulling back via the translation map to the origin of $T_xM$. Let us denote the vector bundle over $M$ with
fibers $\Omega^*(ST_xM)^{T_xM}$ by $\Omega^*(STM)$. Now if we let $x$ vary as well we obtain an isomorphism
\begin{equation}\label{eq:identification_forms}
   \Omega^*(SM) \simeq \Gamma(\Omega^*(STM)),
\end{equation}
where $\Gamma$ denotes the space of smooth sections.

\begin{proposition} \label{prop_isom_curv_general}
The isomorphism \eqref{eq:identification_forms} induces an isomorphism
\begin{displaymath}
\tau:\curvinfty(M) \simeq \Gamma(\curv(TM)).
\end{displaymath}
\end{proposition}
\begin{proof}
By Proposition \ref{prop:kernel thm}, if we put $\bar \alpha$ for the canonical 1-form of $S(T_xM)$ then it is enough to
show that $\alpha, \bar \alpha$ and $d\alpha, d\bar \alpha$ correspond to each other under the identification
\eqref{eq:lambda lambda}.

That this is true of $\alpha$ and $\bar \alpha$ is trivial: in both cases the differential form in question is
given by projection to the line $\spann(\xi) \simeq \R$ in the decompositions \eqref{eq:split
line}, \eqref{eq:decomp 2}. Note that the correspondence $d\alpha \sim d\bar \alpha$ is {\it not} a formal
consequence, since the isomorphism \eqref{eq:lambda lambda} only respects the linear structure but does not
intertwine the exterior derivative.

Representing elements of $T_{\xi}S(T_xM)$ as indicated by the decomposition \eqref{eq:decomp 2}, we have
\begin{equation}\label{eq:symplect}
d\bar\alpha ((v_1,w_1),(v_2,w_2)) =  \langle w_1,v_2\rangle - \langle v_1,w_2\rangle.
\end{equation}
In order to compare this with $d\alpha$, let $\bar X, \bar Y $ be smooth vector fields on $SM$ defined in a neighborhood
of ${\xi}$. Assume that $\pi_*\bar X, \pi_*\bar Y$ are everywhere linearly independent. Then there exists a smooth unit
vector field $Z$ on $M$, defined in a neighborhood of $x$, such that
$Z(x) = \xi$ and $\bar X({\xi}) , \bar Y({\xi})$ are tangent to the graph of $Z$ at $\xi$.
Define vector fields $X,Y$ on a neighborhood of  $x \in M$ by  $X(p):= \pi_*\bar X(Z(p)), Y(p):= \pi_*\bar Y(Z(p))$.
With respect to the decomposition \eqref{eq:decomp 1}, the values of $\bar X,\bar Y$ along the graph of $Z$ are given by
$$
\bar X(Z) = ( X, \nabla_X Z), \quad \bar Y(Z) = ( Y, \nabla_Y Z).
$$
Since, for $W \in T_ZSM$,
$$
\alpha_Z(W) =  \langle Z,\pi_* W\rangle,
$$
we may compute

\begin{align*}
(d\alpha)_Z(\bar X,\bar Y) &=  D_{\bar X} \langle Z,\pi_*\bar Y\rangle -D_{\bar Y}\langle Z,\pi_*\bar X\rangle -
\langle Z,\pi_*[\bar X,\bar Y]\rangle \\
&= D_X \langle Z,Y\rangle -D_Y \langle Z,X\rangle - \langle Z,[X,Y]\rangle \\
&= \langle \nabla_X Z,Y\rangle - \langle \nabla_Y Z,X\rangle.
\end{align*}
Comparing with \eqref{eq:symplect}, this concludes the proof under the linear independence assumption above. The
dependent case follows by continuity.
\end{proof}

Fixing $x \in M$, another way to describe the map $\tau_x:\curvinfty(M) \to \curv(T_xM)$ is as follows. Observe that (by
continuity) an element $\Phi \in \curv(T_xM)$ is determined by its values $\Phi(P,\cdot)$ on convex polytopes $P \subset
T_xM$. These values have the form
\begin{equation}\label{eq:curv meas polytope} 
\Phi(P,\cdot) = \sum_{k=0}^n\sum_{F\in \mathfrak{F}_k(P)}  c(\Phi, \Tan(P,F))\restrict{\vol_k} F
\end{equation}
where $\mathfrak{F}_k(P)$ denotes the $k$-skeleton of $P$ and $\Tan(P,F)$ is the tangent cone of $P$ at a
generic point of $F$. This cone is the Minkowski sum of the $k$-plane $\vec F\in \Gr_k(T_xM)$ generated by $F$
and a proper convex cone lying in the orthogonal complement of $\vec F$. Now if $\Psi \in \curvinfty(M)$ then
for a cone $C$ of this type we have for small $\eps >0$
$$
c(\tau_x(\Psi), C) = \Theta_k(\restrict{\Psi(\exp_x(C \cap B(0,\eps)))} F, x),
$$
where $\Theta_k$ denotes the density with respect to $k$-dimensional volume.

\subsection{Kinematic formulas} 

Now suppose in addition that the Riemannian manifold $M^n$ admits  a Lie group $G$ of isometries such that the
induced action on the tangent sphere bundle $SM$ is transitive. We then say that $(M,G)$ is an {\bf isotropic
space}.

Given such $(M,G)$, we denote by $\curvinfty^G= \curvinfty^G(M)\subset \curvinfty(M)$ and $\V^G= \V^G(M)\subset \V(M)$
the respective spaces of all smooth curvature measures $\Phi\in \curvinfty(M)$ and all smooth valuations $\phi \in
\V(M)$ that are invariant under the action of $G$, i.e. 
\begin{equation}\label{eq:invt cm}
\Phi (gX,g U) = \Phi(X,U), \quad \phi(gX) = \phi(X)
\end{equation}
for all $g \in G$, every sufficiently nice $X \subset M$, and every open set $U \subset M$. 

 \begin{proposition}\label{prop:invariant curv} Let $M$ be a Riemannian manifold, and $G$ a group of isometries $M\to
M$. Then
 $\Phi \in \curvinfty^G(M)$ if and only if $\Phi = \Int(\omega,\phi)$ for some $G$-invariant forms $\omega ,\phi$.
In particular, if $G$ acts isotropically on $M$ then $\dim \curvinfty^G(M), \dim \V^G(M) <\infty$.
\end{proposition}

\begin{proof} 
Let $\Phi=\Int(\omega,\phi)$ be $G$-invariant. Since $\Phi(M,U)=\int_U \phi$, the form
$\phi$ must be $G$-invariant. 
 
By the Lefschetz decomposition (compare \cite{huybrechts05}, Prop. 1.2.30), the restriction of the horizontal part of
$\omega$ to the contact distribution $Q$
can be uniquely written as
the sum of a primitive form $\omega_p$ and a multiple of the symplectic form $d\alpha$. By Proposition
\ref{prop:kernel thm}, any two forms $\omega, \omega' $ define the same curvature measure iff
$\omega_p=\omega'_p$.

Since $\Phi$ is $G$-invariant, it follows that $\omega_p$ is a $G$-invariant form on the contact
distribution. Let $\tilde \omega$ be the unique horizontal form whose restriction to $Q$ equals $\omega_p$.
Then $\Phi=\Int(\tilde \omega,\phi)$ is a representation of $\Phi$ by invariant forms.  

The last conclusion now follows from the  fact that the spaces of invariant differential forms on $SM$ and on $M$ are
finite dimensional.
\end{proof}

\subsubsection{Analytic continuation}\label{sect:an cont 1}

 {In our discussion of the complex space forms, many formulas will depend on the curvature parameter $\lambda$.
Typically these formulas are much easier to derive in the compact cases (i.e. where $\lambda >0$). Fortunately, they
turn out to depend analytically on $\lambda$, and therefore it makes sense to argue that their geometric content remains
valid also for $\lambda \le 0$. 
 
We justify this contention in this subsection. We consider products of invariant valuations in general families
of isotropic spaces $M_\lambda, \lambda \in \R$. Such valuations may be expressed in terms of invariant differential
forms on the sphere bundle $SM_\lambda$, and (as in  the transfer principle Theorem \ref{thm_transfer} below) the
algebras of such forms for different values of $\lambda$ may be canonically identified. 

We will need to express the Alesker product of valuations in these terms, using the formulation of \cite{ale-be09}, the
key ingredient of which is the action of the {\it Rumin differential} $D_\lambda$ on differential forms of degree $n-1$
on the sphere bundle $SM_\lambda$. Although the various $D_\lambda$ do {\it not} intertwine the identifications above,
our main point here is that they do vary analytically with $\lambda$, and therefore the same is true of the Alesker
product for invariant valuations when expressed in terms of a common model for the various spaces of invariant
differential forms.
  }

Let $M_+,M_-$ be $n$-dimensional rank one symmetric spaces with $M_\pm=G_\pm/H$, and canonical decompositions of the
Lie algebras $\mathfrak g_\pm= \mathfrak p_\pm\oplus \mathfrak h$. One says that $M_+,M_-$ form a dual pair if
$\mathfrak p_+$ can be identified with $\mathfrak p_-$ in such a way that the Lie brackets become
\[
 [X ,Y ]_-=-[X , Y]_+,\qquad [X ,U ]_-=[X , U]_+,\qquad [U ,V ]_-=[U , V]_+
\]
for $X,Y\in \mathfrak p:=\mathfrak p_{\pm}$, $U,V\in \mathfrak h$. It is known that every rank one symmetric space
has a dual in
this sense, and that each dual pair contains exactly one compact space. 

Let us fix invariant  metrics $g_{\pm}$ on $M_\pm,$  such that the identification $\mathfrak
p_+\equiv\mathfrak p_-$ becomes an isometry. For positive $\lambda$, let us denote $M_{\pm \lambda}$ the
symmetric space
$M_{\pm}$ endowed with the metric $g_{\pm \lambda}=\lambda^{-1}g_\pm$. We also put $M_0=\mathfrak p$ and
$G_0=H\ltimes \mathfrak p$.
This yields a family of Riemannian symmetric spaces $M_\lambda=G_\lambda/H, \lambda \in \R$ with a common
canonical
decomposition $\mathfrak g=\mathfrak p \oplus \mathfrak h$. Without loss of generality we assume that
$M_\lambda$ is compact for $\lambda>0$. 

The Lie bracket satisfies
\begin{equation}
 \label{eq_bracket_decomposition}
[\mathfrak p,\mathfrak p]_\lambda \subset\mathfrak h,\qquad [\mathfrak p,\mathfrak h]_\lambda\subset\mathfrak
p,\qquad [\mathfrak h,\mathfrak h]_\lambda\subset\mathfrak h.
\end{equation}

If we prescribe that $\mathfrak p$ is
isometrically identified with each $T_oM_\lambda$, the Lie brackets are given by
\begin{equation} \label{eq_lie_brackets}
 [X ,Y ]_\lambda=\lambda[X , Y]_1,\qquad [X ,U ]_\lambda=[X , U]_1,\qquad [U ,V ]_\lambda=[U , V]_1
\end{equation}
for $X,Y\in \mathfrak p$, $U,V\in \mathfrak h$, $\lambda\in\R$. 

It follows that the family of differential maps
 \[
  d_\lambda:\bigwedge{}^*\mathfrak g^*\rightarrow \bigwedge{}^{*+1}\mathfrak g^*
 \]
is polynomial in $\lambda$. Indeed, let $\theta_1,\ldots, \theta_n$ be a basis of $\mathfrak p^*$, and let
$\varphi_{n+1},\ldots, \varphi_N$ be a basis of $\mathfrak h^*$. By \eqref{eq_bracket_decomposition} and the
Maurer-Cartan equation (cf. e.g. \cite{ko-no}, Chapter I, page 41), there are constants $C_i^{jk}$ independent of $\lambda$ such that 
\begin{align}
 d_\lambda\theta_i&=\theta_i\circ [\ ,\ ]_\lambda=\sum_{j,a} C_i^{ja} \theta_j\wedge \varphi_a
\label{eq_dtheta}\\
d_\lambda\varphi_a&=\varphi_i\circ [\ ,\ ]_\lambda=\lambda\sum_{j,k} C_a^{jk} \theta_j\wedge \theta_k +
\sum_{b,c} C_a^{bc} \varphi_b\wedge \varphi_c \label{eq_dvarphi}
\end{align}
where $1\leq i,j,k\leq n$ and  $n+1\leq a,b,c\leq N$.

\bigskip  Now we consider $\Omega^{k}(SM_\lambda)^{G_\lambda}$. Let $\mathfrak k$ be the Lie algebra of the subgroup
$K\subset H$ fixing an element of $SM_\lambda$. Then $\Omega^{k}(SM_\lambda)^{G_\lambda}$ is identified with 
\[
 E_\lambda^k=\left\{ \omega\in {\bigwedge}^{k}\mathfrak g^*\colon \omega|_{\mathfrak k}=d_\lambda\omega|_{\mathfrak
k}=0\right\}.
\]
It follows from \eqref{eq_dtheta} and \eqref{eq_dvarphi} that $E^k:=E^k_\lambda$ is independent of
$\lambda$, and that $d_\lambda(E^k)\subset E^{k+1}$. In particular $\Omega^*(SM_\lambda)^{G_\lambda}\simeq
\Omega^*(SM_0)^{G_0}$. This identification coincides with the one induced by \eqref{eq:identification_forms}.

\begin{lemma}\label{lemma_rumin_analytic}
 The one-parameter family of Rumin differentials $D_\lambda: E^{n-1}\rightarrow E^n$ is analytic in $\lambda$.
\end{lemma}
\proof
Let $\omega_1,\ldots,\omega_k$ be a basis of the space of vertical $(n-1)$-forms. Let $\tau_1,\ldots,\tau_k$ be
a basis of the space of horizontal $n$-forms. Let
$\rho_1,\ldots,\rho_l$ be a basis for the space of vertical $n$-forms.

Then we obtain
\begin{displaymath}
d_\lambda \omega_i = \sum_{j=1}^k c_{i,j}(\lambda) \tau_j + \sum_{j=1}^l c'_{i,j}(\lambda) \rho_j
\end{displaymath}
where the coefficients $c,c'$ are polynomials in $\lambda$.

To find $D_\lambda \omega$, we must find $c_1(\lambda),\ldots,c_k(\lambda)$ such that
$D_\lambda \omega= d_\lambda(\omega+\sum_s c_s(\lambda) \omega_s)$ is vertical. Expanding this yields a system
of linear equations for the $c_s(\lambda)$'s where the coefficients are polynomials in $\lambda$. By Rumin's
theorem \cite{rumin94}, there exists a unique solution. By Cramer's rule this solution will be
rational in $\lambda$, hence analytic. From this it follows that $D_\lambda \omega=d_\lambda \omega+\sum_s
c_s(\lambda) d_\lambda \omega_s$ is analytic in $\lambda$.
\endproof

\begin{proposition}\label{prop:module_analytic}
 There exists a family of linear maps $\psi_\lambda$ depending analytically on $\lambda\in \R$
\[
 \psi_\lambda: (E^{n-1}\oplus {\bigwedge}^{n}\mathfrak p^*) \oplus(E^{n-1}\oplus {\bigwedge}^{n}\mathfrak
p^*)\longrightarrow (E^{n-1}\oplus {\bigwedge}^{n}\mathfrak p^*)
\]
such that, for $\omega_1,\omega_2\in E^{n-1}\simeq\Omega^{n-1}(SM_\lambda)^{G_\lambda}$ and
$\phi_1,\phi_2\in\bigwedge^{n}\mathfrak p^*\simeq {\Omega^n(M_\lambda)^{G_\lambda}}$,
\[
 \Int \psi_\lambda(\omega_1,\phi_1,\omega_2,\phi_2)={\Glob}(\omega_1,\phi_1)\cdot \Int(\omega_2,\phi_2),
\]
where the product in the last expression is the module product of Proposition \ref{prop_def_module}.
\end{proposition}

\proof
From the proof of Proposition \ref{prop_def_module}, we recall that 
\[
{\Glob}(\omega_1,\phi_1)\cdot \Int(\omega_2,\phi_2)=
\Int(Q_1(\omega_1,\omega_2),Q_2(\omega_1,\phi_1,\omega_2,\phi_2))
\]
where $Q_1,Q_2$ are certain operations on differential forms explicitly given in \cite{ale-be09}. These operations
involve the Rumin differential $D_\lambda$, and are independent of $\lambda$ besides that. Therefore the result follows
from the lemma above. 
\endproof

For $\lambda \in \R,\lambda'\in\mathbb R\setminus\{0\}$ let us consider
$f_\lambda^{\lambda'}:E_\lambda\rightarrow E_{\lambda'}$ given as follows. If 
\begin{displaymath}
\omega=\omega_1\otimes\omega_2\in \bigwedge{}^*\mathfrak p^*\otimes\bigwedge{}^*\mathfrak h^*=\bigwedge{}^*\mathfrak
g^* 
\end{displaymath}
then 
\begin{displaymath}
 f_\lambda^{\lambda'}(\omega):=\left(\frac \lambda{\lambda'}\right)^{\lfloor\frac{\deg\omega_1}2\rfloor}\omega,
\end{displaymath}
and we extend linearly. We say that $\omega\in E$ is \textbf{even} (resp. \textbf{odd}) if only even (resp.
odd) degrees of $\omega_1$ occurred in the previous definition.  Note that $d_\lambda$ preserves this parity
notion  by \eqref{eq_dtheta} and \eqref{eq_dvarphi}.

If $\omega$ is even, then $\omega\in E_\lambda$ and $f_\lambda^{\lambda'}(\omega)\in  E_{\lambda'}$ define the same form
in $SM_{\lambda}=SM_{\lambda'}$, whenever $\lambda\lambda'>0$.  If $\omega$ is odd, then $\omega$ and
$(\lambda/\lambda')^{1/2}f_{\lambda,\lambda'}(\omega)$ correspond to the same form in $SM_{\lambda}=SM_{\lambda'}$,
provided $\lambda\lambda'>0$.   

Since the Rumin operator $D$ is independent of the Riemannian metric and since $D_\lambda$ respects the
parity, we have for all $\lambda$ with $\lambda \lambda'>0$ 
\begin{equation}
 \label{eq:DffD}D_{\lambda'} f_{\lambda}^{\lambda'} \omega= f_{\lambda}^{\lambda'} D_{\lambda}\omega, \qquad
\forall\omega\in E_{\lambda}.
\end{equation}
Since both sides are analytic in $\lambda$, this equation is true for all $\lambda$. 

\begin{proposition}\label{prop:no vol}
The volume valuation cannot be represented by an invariant form $\omega\in
\Omega^{n-1}(SM_\lambda)^G$.
\end{proposition}

\proof
Suppose that $\omega\in
\Omega^{n-1}(SM_\lambda)^G$ represents the volume. By \cite{bebr07} this is equivalent to $D_{\lambda}
\omega=\pi^*\vol$. 

For $\lambda>0$,  evaluation on the whole space $M_\lambda$ yields $0=\vol(M_\lambda)$,  a contradiction.

Next we consider the case $\lambda=0$. From \eqref{eq:forms decomp} it follows that the degree of each homogeneous
component of the valuation represented by
$(\omega,0)$ is less than $n$. Since the degree of $\vol$ is $n$, this shows the statement in this case. 

Now suppose that $D_{\lambda} \omega=\pi^*\vol$ for some $\lambda<0$.  By \eqref{eq:DffD}
\begin{displaymath}
 D_{1}f_{\lambda}^{1}(\omega)={\lambda}^{\lfloor\frac n2\rfloor}\pi^*d\vol,
\end{displaymath}
in contradiction to what we have shown above. 
\endproof

\begin{proposition}
A $G$-invariant valuation can be represented by a pair $(\omega,\phi)$ of $G$-invariant forms. 
\end{proposition}

\proof  Let $\mu\in\V^{G_\lambda}(M_\lambda)$ be represented by $\omega\in \Omega^{n-1}(SM_\lambda)$ and
$\phi\in \Omega^n(M_\lambda)$. If $\lambda>0$, then $G_\lambda$ is compact, and we may take $\omega,\phi$
invariant by averaging over $G_\lambda$.

The statement for $\lambda=0$ follows from Theorem 3.3. in \cite{bebr07}. 

Suppose now $\lambda<0$. By Proposition 1.7 in \cite{bebr07}, the form $D_\lambda\omega+\pi^*\phi=:\eta$ is
invariant. Let us consider $f_ \lambda^1\eta$ which defines an invariant vertical closed form in
$\Omega^n(SM_1)$. The Gysin long exact sequence \cite{gysin} of the sphere bundle $\pi\colon SM_1\rightarrow M_1$ yields
$H^n(SM_1)=0$ or $H^n(SM_1)\cong H^n(M_1)$. Hence, \[
f_\lambda^1\eta= D_{1}\rho+\pi^*\varphi
\]
for some differential forms $\rho\in \Omega^{n-1}(SM_{1})$, $\varphi\in\Omega^n(M_1)$. By averaging over ${G_1}$ we can
make $\rho,\varphi$ invariant. Put $\rho_{\lambda}=f_ 1^{\lambda}\rho$ and $\pi^*\varphi_{\lambda}=f_
1^{\lambda}\pi^*\varphi$. By \eqref{eq:DffD},
\[
 D_\lambda \rho_\lambda+  \pi^*\varphi_\lambda=\eta.
\]
It follows from Theorem \ref{prop:kernel thm} that $\mu\equiv {\Glob}(\rho_\lambda,\varphi_\lambda)$ up to
a multiple of the Euler characteristic $\chi$. Since $\chi$ can be represented by an invariant pair, the statement
follows.
\endproof

\subsubsection{Local kinematic formulas} Any isotropic space admits an array of kinematic formulas. These come in three
types: 
{\bf local}, 
{\bf global}, and
{\bf semi-local}. The local formulas apply to invariant curvature measures, while the global formulas apply to invariant
valuations. The latter are obtained from the former by intertwining with the globalization map, and the semi-local
formulas are obtained by applying globalization only partially. Thus the local formulas are in a certain sense the most
fundamental, especially in view of the transfer principle Theorem \ref{thm_transfer} below. On the other hand the
global formulas may be viewed as adjoint to the multiplicative structure on $\V^G$ (Theorem \ref{thm_ftaig} below), and
thus enjoy special formal properties. The semi-local formulas are similarly related to the structure of $\curvinfty^G$
as a module over $\V^G$.

The construction of the kinematic formulas for curvature measures was given in a slightly different formulation in
\cite{fu90,fu94}. We sketch the proof, emphasizing certain points that will be important in the present paper.
Referring to the second part of the Remark following Definition \ref{def:curv & val}, the precise class of subspaces
$X,Y \subset M$ for which it is valid remains poorly understood, although by \cite{fu94} this class includes
the  subanalytic and semiconvex cases, and therefore also the case of compact differentiable polyhedra.

Let $dg$ denote the bi-invariant measure on $G$, normalized so that 
\begin{equation}\label{eq:normal haar}
dg(\{g: go \in E\}) = \vol(E)
\end{equation}
whenever $o \in M$ and $E \subset M$ is measurable.

Recall that a {\it coproduct} on a vector space $V$ is a linear map $c:V \to V\otimes V$. The coproduct $c$ is {\it
cocommutative} if $c(v)$ is symmetric for all $v \in V$, and {\it coassociative} if
\begin{equation}\label{eq:def coassoc}
(\id\otimes c) \circ c = (c \otimes \id ) \circ c.
\end{equation}

\begin{theorem}[\cite{fu90}] \label{thm:basic K} There exists a cocommutative, coassociative coproduct 
$$
K= K_G: \curvinfty^G \to  \curvinfty^G \otimes  \curvinfty^G 
$$
such that for all sufficiently nice compact subsets $P,Q \subset M$ and open sets $U,V \subset M$
\begin{equation}\label{eq:curv K}
K(\Phi) (P,U;Q,V) = \int_{G} \Phi(gP \cap Q,g U \cap V) \, dg.
\end{equation}
\end{theorem}
\begin{proof}
We use some notions from Geometric Measure Theory and refer to \cite{gmt}
for details. The main point is that the normal cycles of generic intersections $P\cap gQ$ arise as slices of a fixed
current that is constructed in a canonical way from $N(P),N(Q)$.

Let $\pi:SM\to M$ be the sphere bundle over $M$. Pick reference points $o \in M, \bar o \in S_oM$. Let
\begin{equation}\label{eq:C}
\CC:= \{(g, v )\in G_o \times S_oM: g\bar o \ne -\bar o, \ v\in \overline {\bar o, g\bar o}\},
\end{equation}
where for non-antipodal points $v,w \in S_oM$ we put $\overline {v,w}$ for the open minimizing geodesic arc joining $v$
to $w$. Then $\CC$ has finite volume and admits a natural orientation, and the action
$$
(\ell_1,\ell_2)\cdot (g,v):= (\ell_2g \ell_1^{-1}, \ell_2v)
$$
of $G_{\bar o}\times G_{\bar o}$ on $G_o \times S_oM$ induces an orientation- and volume-preserving action on $\CC$. Let
$E$ be the fiber bundle over $SM \times SM$ obtained by pulling back the bundle
$$
G\times SM \to M \times M, \quad (g, \xi) \mapsto (\pi  g\xi, \pi  \xi)
$$
to $SM \times SM$ via the projection $SM \times SM \to M\times M$. Thus we have the commutative diagram (cartesian
square)
\begin{equation}\label{basic diagram}
\xymatrix{E \ar[d] \ar[r]^p & G\times SM \ar[d]\\
SM \times SM \ar[r] & M \times M
}
\end{equation}
where the vertical maps have fibers isomorphic to $G_o \times S_oM$. The natural actions of the group $G\times G$ are
intertwined by the maps of \eqref{basic diagram}. 

 While the group of transition functions for the bundle on the right is $G_o \times G_o$, on the left the group reduces
to $G_{\bar o}\times G_{\bar o}$. Therefore there is a natural copy of $\CC$ inside the fiber over each $(\xi,\eta) \in
SM \times SM$, namely
$$
\CC_{\xi,\eta}= \{(g,\zeta) \in G \times SM: \pi  g \xi = \pi  \eta = \pi  \zeta=: x, \ g\xi \ne -\eta, \ \zeta \in
\overline{g\xi,\eta} \subset S_{x}M\}.
$$
 Let $\pi_{\CC}:F \to SM \times SM$ the corresponding subbundle of $E$ and $\iota:F \to E$ the inclusion.

For $P,Q $ as above we may construct the fiber product 
$$
 T_0(P,Q) :=\iota_*(N(P) \times N(Q) \times_E \CC)
$$
 as a well-defined integral current living in $E$.
 Put 
 \begin{align}
\label{eq:def gamma bundles} \Gamma_1&:=\{(g,x,\eta)\in G \times M\times SM:  gx = \pi_M\eta\}, \\
 \Gamma_2 &:=\{(g,\xi,y) \in G \times SM \times M:  g\pi \xi = y\}.
 \end{align}
 We may regard $\Gamma_1, \Gamma_2$ as smooth fiber bundles over $M\times SM, SM \times M$ respectively, with fibers diffeomorphic to $G_o$. Thus we construct also two currents living in $G \times M \times SM$ and $G \times SM \times M$, respectively, by
\begin{align*}
T_1(P,Q)&:= \{(g,x,\eta)\in P\times N(Q):  gx = \pi_M\eta\},\\
 T_2(P,Q)&:= \{(g,\xi,y) \in N(P) \times Q:  g\pi \xi = y\}.
\end{align*}

Consider the natural projections
$$
 p_1:G \times M \times SM\to G\times SM,\  p_2:G \times SM \times M \to G\times SM
$$
and put
\begin{equation}\label{eq:decomp T}
T(P,Q):= p_* T_0(P,Q) + p_{1*}T_1(P,Q) + p_{2*}T_2(P,Q)  \subset G \times SM.
\end{equation}
Clearly 
\begin{equation}\label{eq:mass T}
\mass T(P,Q) \le c ( \mass N(P)  \mass N(Q) + \vol P \mass N(Q) + \mass N(P) \vol Q).
\end{equation}

For a.e. $g\in G$ the normal cycle of $gP\cap Q$ is given in terms of  slices of $T(P,Q)$ by
\begin{equation}\label{eq:slicing procedure}
N(gP\cap Q) = \pi_{SM*}\langle T(P,Q), \pi_G, g\rangle
\end{equation} 
where $\pi_{SM}, \pi_G$ are the natural projections of $G\times SM$ to the two factors (on the right hand side  we use the notation $\langle C, f, t\rangle$ of \cite{gmt}, section 4.3, to denote the slice of the current $C$ by the map $f$ at the value $t$). Moreover, if $U,V\subset M$ are
open then 
$$
W=W(U,V):= \{(g,\zeta) \in G\times SM: \pi\zeta \in gU \cap V\}
$$
is open, and for a.e. $g\in G$
\begin{equation}\label{eq:slice}
N(gP\cap Q) \lefthalfcup \pi^{-1}(gU \cap V) = \pi_{SM*}\langle T(P,Q) \lefthalfcup W, p,g \rangle.
\end{equation}

Now consider the operator  
$$
H:\Omega^*(SM)^G \to \Omega^*(SM\times SM)^{G\times G} \simeq \Omega^*(SM)^G\otimes  \Omega^*(SM)^G 
$$ 
given by
\begin{equation}\label{eq:def H}
 H\phi:= \pi_{\CC*} (\iota \circ p)^*(\phi \wedge dg).
\end{equation}
Thus, for $\phi \in \Omega^*(SM)^G $,
$$
\int_{p_*T_0(P,Q)} \phi \wedge dg = \int_{N(P) \times N(Q) } H\phi.
$$

Given $P,Q,U,V$ as above and $\omega \in \Omega^{n-1}(SM)^G$ we may now compute, using  Theorem 4.3.2 of \cite{gmt}, 
\begin{align*}
\int_G \Int(\omega,0)(gP\cap Q,gU\cap V) \,dg &= \int_G\left( \int_{N(gP\cap Q) \cap \pi^{-1}(gU \cap
V)}\omega\right)\,dg\\
&= \int_G\left(\int_{ \pi_{SM*}\langle T(P,Q) \lefthalfcup W, \pi_G,g \rangle}\omega\right) \, dg\\
&= \int_G\left(\int_{ \langle T(P,Q) \lefthalfcup W, \pi_G,g \rangle}\pi_{SM}^*\omega\right) \, dg\\
&= \int_{  T(P,Q) \lefthalfcup W}\pi_{SM}^*\omega\wedge \pi_G^* dg\\
&= \int_{(N(P)\lefthalfcup U)\times (N(Q)\lefthalfcup V)} H\omega\\
&+ \int_G \left(\int_{N(Q)\lefthalfcup\pi^{-1}(  gP \cap gU\cap V))} \omega \right) \, dg\\
&+ \int_G \left(\int_{N(gP) \lefthalfcup \pi^{-1}(gU \cap Q\cap V)} \omega\right) \, dg
\end{align*}
where in the last equality we decompose the integral according to \eqref{eq:decomp T} and use the definition of the
fiber product of currents in \cite{fu90}. In view of the normalization \eqref{eq:normal haar}, the last two terms yield
$\vol(P\cap U) \int_{N(Q) \lefthalfcup\pi^{-1} V} \omega + \vol(Q\cap V) \int_{N(P) \lefthalfcup\pi^{-1} U} \omega$.
Therefore the kinematic operator $K_G$ is  given by
\begin{align}\label{eq:KG integ}
K_G(\Int(\omega, 0)) &=( \Int \otimes \Int) (H\omega,0) + \vol \otimes \Int(\omega,0) + \Int (\omega ,0)\otimes \vol  \\
\notag K_G(\Int(0, d\vol_M))&= K_G(\vol) = \vol \otimes \vol.
\end{align}

That the coproduct $K_G$ is cocommutative follows from the fact that the map $g \mapsto g^{-1}$ preserves the Haar
measure $dg$. Coassociativity follows from Fubini's theorem.
\end{proof}

\subsubsection{Global and semi-local formulas, and the {fundamental theorem}} \label{sect:global and semilocal}

We define
\begin{align*}
k &= k_G: \V^G\to \V^G\otimes \V^G ,\\
\bar k & = \bar k_G: \curvinfty^G \to \curvinfty^G\otimes \V^G
\end{align*}
by
\begin{align}
\label{eq:def k}k_G \circ \glob&:= (\glob\otimes \glob) \circ K_G ,\\
\label{eq:def bark}\bar k_G&:= (\id \otimes \glob)\circ K_G.
\end{align}
Alternatively, we may regard $k$ as taking values in $\Hom(\V^G(M)^*,\V^G(M))$. 

Although
$\glob$ is not in general injective, it is trivial to check that the first of these relations defines $k_G$
uniquely. Thus if $A,B \subset M$ are nice sets and $U\subset M$ is open then
\begin{align*}
k_G(\phi)(A,B) &= \int_G \phi(A\cap gB) \, dg\\
\bar k_G(\Phi)(A,U;B) &= \int_G \Phi(A\cap gB, U) \, dg.
\end{align*}

By Proposition \ref{prop:no vol}, $\vol$ does not belong to the space $\{{\Glob} (\omega,0) :\omega
\in
\Omega^{n-1}(SM)^G\}$. Hence the following definition makes sense. 

\begin{definition} $\vol^* $ is the  element of  $\V^{G*}$ characterized uniquely by
$$
\langle \vol^*, \vol\rangle = 1, \quad \langle\vol^*, {\Glob} (\omega,0) \rangle = 0 \text{ for all
} \omega \in \Omega^{n-1}(SM)^G.
$$
\end{definition}

\begin{lemma}\label{lem:kg phi vol}
If $\phi \in \V^G $ then
$$
k(\phi) (\vol^*, \cdot) = \phi.
$$
\end{lemma}
\begin{proof} 
This follows at once from \eqref{eq:KG integ} and the definition of
$k$.
\end{proof}

\begin{definition}
 The map $\pd:\V^G(M)\rightarrow (\V^G(M))^*$ given by
\[
 \langle \pd(\phi),\mu\rangle =\langle \vol^*, \phi\cdot\mu\rangle,
\]
is called (normalized) Poincar\'e duality.
\end{definition}
If $M$ is euclidean then this coincides with the Poincar\'e duality defined in \cite{ale04}. If $M$ is
compact, this is $\vol(M)^{-1}$ times the Poincar\'e duality defined in \cite{ale05d}. Clearly $\pd$ is self-adjoint.

Given $A \in \mathcal{P}(M)$, we put 
\begin{displaymath}
 \mu_A^G:=\int_G \chi(\cdot \cap gA)dg = k(\chi) (A,\cdot) \in \V^G.
\end{displaymath}
\begin{proposition} \label{general_Hadwiger_thm}
 Given any smooth valuation $\phi$ on an isotropic space $M$, we have 
\begin{equation} \label{eq_prod_special}
\mu_A^G \cdot \phi=\int_G \phi(\cdot \cap gA)dg.
\end{equation}
\end{proposition}

We defer the proof of this proposition to the Appendix. 

\begin{corollary} \label{cor_product_special}
 Given $A \in \mathcal{P}(M)$, $\Phi \in \curvinfty(M)$ and $U \subset M$ open we have 
\begin{displaymath}
 \mu_A^G \cdot \Phi(B,U)=\int_G \Phi(B \cap gA,U)dg.
\end{displaymath}
\end{corollary}

\begin{proposition}\label{prop:pd phi} If $\phi \in \V^G$ then
 \begin{equation}\label{eq:pd_vs_mu}
  \langle \pd(\mu_A^G),\phi\rangle=\phi(A).
 \end{equation}
\end{proposition}
\proof
By Proposition \ref{general_Hadwiger_thm} and Lemma \ref{lem:kg phi vol},
\begin{displaymath}
  \langle \pd(\mu_A^G),\phi\rangle =\langle \vol^*, \phi \cdot \mu_A^G\rangle
= \langle \vol^*, k_G(\phi)( \cdot,A)\rangle = k_G(\phi)( \vol^*,A)= \phi(A).
\end{displaymath}
\endproof

Let $\mathrm{ev}:\mathcal{P}(M) \to \V^G(M)^*$ be
the evaluation map. Then $\mu_A^G=k(\chi)(\mathrm{ev}_A)$. The proposition may be written as
\begin{equation} \label{eq_pd_ev}
 \pd \circ k(\chi) \circ \mathrm{ev}=\mathrm{ev} \quad \text{or} \quad \pd (\mu^G_A)=\mathrm{ev}_A.
\end{equation}

\begin{corollary}\label{coro:span} \mbox{}
 \begin{enumerate}
     \item[(i)] $\pd:\V^G(M) \to \V^G(M)^*$ and $k(\chi):\V^G(M)^* \to \V^G(M)$ are mutually inverse isomorphisms.
   \item[(ii)] $\V^G= \spann\{\mu_A^G:A \in \mathcal P(M)\}$.
 \end{enumerate}
\end{corollary}
\proof Let $P$ be the vector space with basis $\mathcal P(M)$ and $\mathrm{ev}\colon P\rightarrow (\V^G)^*$ the linear
extension of the evaluation map. Its adjoint $\mathrm{ev}^*\colon \V^G \rightarrow P^*$ is clearly injective. Therefore,
$\mathrm{ev}\colon P\rightarrow (\V^G)^*$ is onto. From \eqref{eq_pd_ev} we deduce that $\pd$ is left inverse to
$k(\chi)$. Since $\V^G(M)$ and $\V^G(M)^*$ have the same finite dimension, $\pd$ is also right inverse to $k(\chi)$.
This shows (i).                  

To see (ii), it is enough to notice that $k(\chi)\circ \mathrm{ev}$ is onto, and that
$\mu_A^G=k(\chi)(\mathrm{ev}_A)$.
\endproof

\begin{theorem} \label{thm_ftaig_curv}  $K,k, \bar k$ are all compatible with the multiplication by elements of $\V^G$,
i.e. given $\phi ,\psi\in \V^G, \Phi \in \curvinfty^G$
\begin{align}
 \label{eq:K mult} K(\phi \cdot \Phi)&=(\phi \otimes \chi) \cdot K(\Phi)=(\chi \otimes \phi)\cdot K(\Phi),\\
 \label{eq:bark mult}  \bar k(\phi\cdot\Phi)&  =  (\chi\otimes \phi) \cdot \bar k(\Phi)=(\phi\otimes \chi) \cdot \bar
k(\Phi),\\
 \label{eq:k mult} k(\phi\cdot\psi)&= (\chi \otimes \phi) \cdot k(\psi) =  (\phi \otimes \chi) \cdot k(\psi),\\
\label{eq:bark mult bis} \bar k(\phi\cdot\Phi)&  = k(\phi)\cdot (\chi \otimes \Phi).
\end{align}
\end{theorem}
\proof  Using Corollary \ref{coro:span} (ii), the relation \eqref{eq:K mult} follows easily from Corollary
\ref{cor_product_special} and Fubini's theorem. Since $\glob(\phi\cdot \Phi) = \phi\cdot \glob\Phi$ for all $\phi\in \V,
\Phi \in \curvinfty$, the relations \eqref{eq:bark mult}, \eqref{eq:k mult} follow at once from \eqref{eq:def k},
\eqref{eq:def bark}.  Equation \eqref{eq:bark mult bis} follows from Corollary \ref{cor_product_special} for
$\phi=\chi$, and from \eqref{eq:bark mult}, \eqref{eq:k mult} in the general case.
\endproof

\begin{corollary}\label{coro_module_versus_kin}
Let $\bar m \in \Hom(\V^G \otimes \curvinfty^G,\curvinfty^G)=\Hom(\curvinfty^G,\curvinfty^G \otimes \V^{G*})$ be the
module structure. Then  
\begin{equation} \label{eq_module_versus_kin}
 \bar m=(id \otimes \pd) \circ \bar k
\end{equation}
or equivalently 
\begin{displaymath}
 \bar k = (id \otimes \pd^{-1}) \circ \bar m.
\end{displaymath}
In particular, $\bar k$ is injective. 
\end{corollary}

\proof
Let $\Phi \in \curvinfty^G$. By \eqref{eq_pd_ev}, $\langle \pd(\psi),\mu_C^G\rangle=\psi(C)$ for every $\psi \in
{\V}^G$. 
By Corollary \ref{cor_product_special}, 
\[
 \mu_C^G \cdot \Phi =\bar k(\Phi)(\cdot,\cdot,C)=\langle (\id\otimes\pd)\bar k(\Phi),\mu_C^G\rangle,
\]
 and the stated equation follows. If $\bar k(\Phi)=0$, then it follows that $\bar m(\Phi)=0$, which means $\phi \cdot
\Phi=0$ for all valuations $\phi \in \mathcal{V}^G$. Taking $\phi:=\chi$ implies that $\Phi=0$. 
\endproof

The next theorem was shown in \cite{befu06} and \cite{ale-be09}, assuming $M$ to be flat or compact respectively. Now we
extend it to the general case. 

\begin{theorem}\label{thm_ftaig}
 Let $m:\mathcal{V}^G(M) \otimes \mathcal{V}^G(M) \to \mathcal{V}^G(M)$ be the restricted multiplication
map, $\pd:\mathcal{V}^G(M) \to \mathcal{V}^G(M)^*$ the normalized Poincar\'e duality and $k:\mathcal{V}^G(M)
\to \mathcal{V}^G(M) \otimes \mathcal{V}^G(M)$ the kinematic coproduct. Then the following diagram commutes:
\begin{displaymath}
\xymatrixcolsep{3pc}
\xymatrix{\mathcal{V}^G(M) \ar[d]^{\pd} \ar[r]^-{k}& \mathcal{V}^G(M) \otimes \mathcal{V}^G(M)   \ar[d]^{\pd
\otimes \pd}\\
\mathcal{V}^G(M)^* \ar[r]^-{m^*} & \mathcal{V}^G(M)^* \otimes \mathcal{V}^G(M)^*}
\end{displaymath}
\end{theorem}

\proof
Since $\pd$ is self-adjoint \eqref{eq:pd_vs_mu} yields
\[
 \langle (\pd\otimes\pd)\circ k(\phi),\mu_A^G\otimes\mu_B^G\rangle=\int_G \phi(A\cap gB)dg.
\]
By Proposition \ref{general_Hadwiger_thm}
\begin{align*}
 \langle m^*\circ \pd(\phi),\mu_A^G\otimes\mu_B^G\rangle&=\langle \pd\phi,\mu_A^G\cdot \mu_B^G\rangle\\
&=\langle \vol^*,\phi\cdot \mu_A^G\cdot \mu_B^G\rangle \\
&=\left\langle \vol^*,\int_{G\times G}\phi(\cdot\cap gA\cap hB)dgdh\right\rangle.
\end{align*}
By \eqref{eq:KG integ}, if $\langle\vol^*,\phi\rangle=0$, then
\begin{align*}
 \int_{G}\int_G\phi(\cdot\cap gA\cap hB)dg\,dh &=  \int_{G}\int_G\phi(\cdot\cap g(A\cap hB))dg\,dh \\
&=\int_{G}(\phi(A\cap hB)\vol +\vol(A\cap hB)\phi+\sum_{i,j}\varphi_i(A\cap hB)\varphi_j )dh
\end{align*}
for some $\varphi_i\in\V^G$ with $\langle \vol^*,\varphi_i\rangle=0$. Hence,
\[
\left\langle \vol^*,\int_{G}\int_G\phi(\cdot\cap gA\cap hB)\, dg\, dh\right\rangle=\int_{G}\phi(A\cap gB)dg
\]
as desired. The case $\phi=\vol$ follows similarly. 
\endproof

\subsection{Transfer principle} \label{subsec_transfer}
Picking points $o\in M, \bar o \in SM$ with $\pi_M \bar o = o$, let $G_{\bar o} \subset G_o \subset  G$ denote the
stabilizers of $o, \bar o $ respectively. This yields identifications
$$
M \simeq G/G_o, \quad SM \simeq G/G_{\bar o} 
$$ 
corresponding to the maps $g \mapsto go, \ g\mapsto g \bar o$ respectively, such that the natural projections $G/G_{\bar
o} \to G/G_o$ and $SM \to M$ are compatible.

The decompositions \eqref{eq:split line} and \eqref{eq:decomp 2} are invariant under the actions of $G_{\bar o}$.
Furthermore the action of $G$ induces canonical isomorphisms $\curv^{\bargo}(T_oM) \simeq \curv^{\overline{
G_x}}(T_xM)$ for every $x \in M$. If we identify these spaces in this way then the isomorphism of Proposition
\ref{prop_isom_curv_general} yields
\begin{proposition} \label{prop_isom_curv}
The map $\tau $ of Proposition \ref{prop_isom_curv_general} induces an isomorphism 
\begin{equation}\label{eq:curv sim 1}
\curvinfty^G(M) \simeq \curv^{\bargo}(T_oM).
\end{equation}
\end{proposition}

As noted in the previous section, each of the 
spaces $\curvinfty^G(M),\curv^{\bargo}(T_oM) $ is a coalgebra, with the kinematic operator as coproduct. 

\begin{theorem}\label{thm_transfer} 
The linear isomorphism \eqref{eq:curv sim 1} is an isomorphism of coalgebras.
\end{theorem}
\begin{proof} 

We actually prove the following stronger statement, at the level of invariant differential forms rather than curvature
measures. The apparatus of the proof of Theorem \ref{thm:basic K} has an analogue for the pair $(T_oM, \bargo)$, leading
to an operator
$$
 H':\Omega^*(ST_oM)^{\bargo} \to  \Omega^*(ST_oM)^{\bargo}\otimes  \Omega^*(ST_oM)^{\bargo}.
$$ 
Since the isomorphism \eqref{eq:curv sim 1} is induced by the natural isomorphism
\begin{equation}\label{eq:sim sim}
\Omega^*(SM)^G\simeq \Lambda^*({ T^*_{\bar o}}SM)^{G_{\bar o}} \simeq \Lambda^*({ T^*_{\bar o}}ST_oM)^{G_{\bar o}}
\simeq \Omega^*(ST_oM)^{\bargo},
\end{equation}
the maps $H,H'$ determine and are determined by maps
\begin{align}
\label{eq:ptwise kf 1}\Lambda^*({ T^*_{\bar o}}SM)^{G_{\bar o}}  &\to \Lambda^*({ T^*_{\bar o}}SM)^{G_{\bar o}} \otimes
\Lambda^*({ T^*_{\bar o}}SM)^{G_{\bar o}},\\
\label{eq:ptwise kf 2}\Lambda^*({ T^*_{\bar o}}ST_oM)^{G_{\bar o}}  &\to \Lambda^*({ T^*_{\bar o}}ST_oM)^{G_{\bar o}}
\otimes \Lambda^*({ T^*_{\bar o}}ST_oM)^{G_{\bar o}}.
\end{align}
We must show that these maps are intertwined by the middle isomorphism of \eqref{eq:sim sim}.

To this end we consider the derived diagram {corresponding to \eqref{basic diagram}}
\begin{equation}\label{derived diagram}
\xymatrix{
\left.TE\right|_{ F} \ar[r]\ar[d] & \left.T\left(G\times SM \right)\right|_{F}=\left. TG \right|_{G_o} \times
\left.TSM\right|_{S_oM} \ar[d]_p\\
T_{\bar o}SM \oplus T_{\bar o}SM \ar[r]^\pi & T_oM \oplus T_oM 
}
\end{equation}
where $  F = G_o\times S_oM $ is the fiber over $(\bar o,\bar o)$, which includes the submanifold $\CC$ defined in \eqref{eq:C}.
The map \eqref{eq:ptwise kf 1} is then constructed as follows. An element of $\Lambda^*(T_{\bar o}SM)^{G_{\bar o}}$
induces a $G_o$-invariant section of $\Lambda^*\restrict{TSM}{S_oM}$. We take the wedge product with the invariant
volume form of $G$, pull back via the top map, and finally push down using the fiber integral over $\CC$.
Likewise, \eqref{eq:ptwise kf 2} arises in the same way from the diagram
\begin{equation}\label{derived diagram 2}
\xymatrix{
\left.T\tilde E\right|_{ F} \ar[r] \ar[d] & \left.T\left(\bargo\times SM \right)\right|_{F}=\left. T\bargo \right|_{G_o}
\times \left.TST_oM\right|_{S_oM} \ar[d]_{\tilde p}\\
T_{\bar o}ST_oM \oplus T_{\bar o}ST_oM \ar[r]^{\tilde \pi} & T_oM \oplus T_oM 
}
\end{equation}

The identification $T_{\bar o} SM \simeq T_{\bar o} ST_oM$ that arises from \eqref{eq:decomp 1} and \eqref{eq:decomp 2}
induces an isomorphism between the lower left spaces, intertwining the projections $\pi, \tilde \pi$.
In the same way we may identify the second factors of the top right spaces, intertwining the restrictions of  $p,\tilde
p$. 

Finally we wish to identify the first factors of the top right spaces. The identification we give will not be
canonical, but  it is enough for our purposes. To this end we realize $\bargo$ as a manifold as the cartesian product
$G_o \times T_oM$, acting on $T_oM$ by
$$
(g,v)\cdot w:= g_*(w+v).
$$
Thus $\restrict{T\bargo}{G_o} \simeq TG_o { \times} T_oM$. Recalling that $\tilde p$ is the derivative of the map
$(g,\zeta)\mapsto  { (g\pi\zeta,\pi\zeta)}$ on the right of the euclidean analogue of \eqref{basic diagram}, we see that
for $(g, \zeta) \in G_o\times S_oM$, and $(\gamma,v) \in T_{ g}G_o \times T_oM, w \in T_\zeta ST_oM$,
\begin{equation}\label{eq:p action}
\tilde p (\gamma,v;w)= (g_* \pi_*w + g_* v, \pi_*w)
\end{equation}
where $\pi:ST_oM \to T_oM$ is the projection.

We may represent $\restrict {TG} {\bargo}$ in a similar fashion. Let $\mathfrak{m} \subset \mathfrak{g}$ be a
linear complement to $\mathfrak{g}_o$ in the Lie algebra $\mathfrak{g} = T_eG$. Then $\pi_*$ induces an
isomorphism $\mathfrak{m} \to T_oM$ where $\pi:G\to M= G/G_o$ is the quotient map. Translating $\mathfrak{m}$ 
by the left multiplication map $L_g, g \in  \bargo$, we obtain an injection 
$$
\begin{CD}T_oM @>{\pi_*^{-1}}>> \mathfrak{m} @>L_{g*}>> g_*\mathfrak{m} \subset T_gG,
\end{CD}
$$
which gives an identification $\restrict {TG}{\bargo} \simeq TG_o { \times} T_oM$. With respect to this
decomposition, the map $p$ acts formally exactly as in \eqref{eq:p action}.

Thus the diagrams \eqref{derived diagram}, \eqref{derived diagram 2}, except for the top left corners, may be
identified, intertwining all of the maps that appear. 
It follows that the top left corners may also be identified:
\begin{lemma}\label{lem:fiber product lemma} Let $M, E'$ be smooth fiber bundles over $N$, and $E$ their fiber product:
\begin{equation}
\xymatrix{
E \ar[r] \ar[d] & E' \ar[d]_p\\
M \ar[r]^f & N 
}
\end{equation}
Let $m \in M$ and $F$ the fiber of $E'$ over $f(m)$, which is the same as the fiber of $E$ over $m$. Then the restricted
tangent bundle $\restrict{TE} F$ is the fiber product over $T_{f(m)} N$ of $T_mM$ and $\restrict{TE'} F$:
\begin{equation}
\xymatrix{
\restrict{TE} F \ar[r] \ar[d] & \restrict{TE'} F \ar[d]\\
T_mM \ar[r] & T_{f(m)} N 
}
\end{equation}
\end{lemma}
\begin{proof} Since $f,p$ are submersions, the implicit function theorem implies that the original fiber product 
$$
E = \{(m, e) \in M\times E': f(m )=p(e) \}
$$
is a smooth submanifold, with
$$
T_{(m,e)} E= \{(v,w) \in T_m M \oplus T_eE':  f_*v= p_*w\}.
$$
Thus 
$$
\restrict {TE} F = \{(m,v;e,w) \in T_mM \oplus \restrict{TE'}F:  f_*v =p_*w \}
$$
as claimed.
\end{proof}

To conclude the proof of Theorem \ref{thm_transfer}, the lemma implies that the maps \eqref{eq:ptwise kf 1},
\eqref{eq:ptwise kf 2} correspond, provided the selected volume forms $dg$ correspond when restricted to $\bargo$. But
this is ensured by the convention \eqref{eq:normal haar}: each corresponding volume form is then the product of the
probability measure on $\bargo$ with the pullback of the volume form of $M$ under the projection $G\to M$.
\end{proof}

{\it Remark.} Theorem \ref{thm_transfer} is a special case of the general transfer principle of Howard \cite{howard}.

By Proposition \ref{prop_def_module}, $\curvinfty^G(M)$ is a module over $\V^G(M)$, while $\curv^{G_o}(T_oM)$ is a
module over $\Val^{ G_o}(T_oM)$.  On the other hand, in view of Proposition \ref{prop_isom_curv}, we may identify the
coalgebras $\curvinfty^G(M) \simeq \curv^{\bargo}(T_oM)$. From this perspective the coalgebra of invariant curvature
measures is simultaneously a module over $\V^G(M)$ and over $\Val^{G_o}(T_oM)$.

\begin{proposition}\label{prop:module commute}
The two module structures on $\curvinfty^G(M) \simeq \curv^{G_o}(T_oM)$ commute, i.e. if 
$\rho \in \V^G(M), \psi \in\Val^{ G_o}(T_oM)$, and $\Phi \in  \curvinfty^G(M)$, then 
\begin{displaymath}
 \psi \cdot \tau(\rho \cdot \Phi)=\tau(\rho \cdot \tau^{-1}(\psi \cdot \tau(\Phi))).
\end{displaymath}
\end{proposition}

\proof
By Theorem \ref{thm_transfer} we have 
\begin{displaymath}
 K_{\overline{ G_o}} \circ \tau=\tau \otimes \tau \circ K_G.
\end{displaymath}
Using Theorem \ref{thm_ftaig_curv} we obtain that 
\begin{align*}
 K_{\overline{ G_o}}(\psi \cdot \tau(\rho \cdot \Phi)) & =(\psi \otimes id) \cdot K_{\overline{ G_o}}(\tau(\rho \cdot
\Phi))\\
 & =((\psi \circ \tau)  \otimes \tau) K_G(\rho \cdot \Phi)\\
& =((\psi \circ \tau) \otimes (\tau \circ \rho)) K_G(\Phi).
\end{align*}

On the other hand,
\begin{align*}
 K_{\overline{ G_o}}(\tau(\rho \cdot \tau^{-1}(\psi \cdot \tau(\Phi)))) & = (\tau \otimes \tau) \circ K_G(\rho \cdot
\tau^{-1}(\psi \cdot \tau(\Phi)))\\
& =  (\tau \otimes (\tau \circ \rho)) K_G(\tau^{-1}(\psi \cdot \tau(\Phi)))\\
& =  (id \otimes (\tau \circ \rho \circ \tau^{-1})) K_{\overline{ G_o}}(\psi \cdot \tau(\Phi))\\
& =   (\psi \otimes (\tau \circ \rho \circ \tau^{-1})) K_{\overline{ G_o}}(\tau(\Phi))\\
& = ((\psi \circ \tau) \otimes (\tau \circ \rho)) K_G(\Phi).
\end{align*}  
By Corollary \ref{coro_module_versus_kin}, $\bar k_{\overline{ G_o}}$ and hence $K_{\overline{ G_o}}$ are injective,
which finishes the proof.
\endproof

\subsection{Angular curvature measures}\label{sect:angular cm}

\begin{definition}\label{def:angular cm} 
Let $V$ be a euclidean space of dimension $m$. A translation-invariant curvature measure $\Xi \in \curv(V)$ is called {\bf
angular} if, for any compact convex polytope $P \subset V$,
\begin{equation}\label{eq:def angular}
\Xi(P, \cdot ) = \sum_{k=0}^m \sum_{F \in \mathfrak{F}_k(P)} c_\Xi({\vec F}) \angle (F,P) \restrict{\vol_k} F
\end{equation}
where $c_\Xi(\vec F)$ depends only on the $k$-plane $\vec F \in \Gr_k(V)$ parallel to $F$ and where $\angle(F,P)$
denotes the outer angle of $F$ in $P$ (in other words, for fixed $\vec F$ the quantities $c(\Phi, \Tan(P,F))$ in
\eqref{eq:curv meas polytope} are proportional to $\angle(F,P)$).  The space of such translation-invariant curvature
measures will be denoted by $\ang(V)$.

Given a Riemannian manifold $M$, a smooth curvature measure $\Phi \in \curvinfty(M)$ is called {\bf angular} if
$\tau_x(\Phi)\in \ang(T_xM)$ for every $x \in M$.
The space of angular curvature measures will be denoted by $\mathcal{A}(M)$.
\end{definition}

If $\Xi \in \ang(V)$ then it is clear that the associated valuation $[\Xi]$ is even, and the constant  of \eqref{eq:def
angular} is 
$$c_\Xi(\vec F)=\kl_{[\Xi]}(\vec F),
$$
the value at $\vec F$ of the Klain function (cf. \cite{befu11}) of the associated valuation $[\Xi] \in \Val(V)$.

The elementary symmetric functions of principal curvatures determine a sequence of canonical angular curvature measures
in any Riemannian manifold $M^m$, expressed in terms of differential forms on the sphere bundle $SM$ as follows. Let
$e_1(\xi),\dots,e_m(\xi) = \xi$ be a local frame for $M$ defined for $\xi \in SM$, with associated coframe $\theta_i$
and connection forms $\omega_{ij}$. Put
\begin{equation}\label{eq:def kappa}
\kappa_i:= \frac1{\alpha_{m-i-1}i!(m-i-1)!}\left[\sum_\epsilon \sgn(\epsilon)
\theta_{\epsilon_1}\dots\theta_{\epsilon_{i}}\omega_{\epsilon_{i+1},m}\dots \omega_{\epsilon_{m-1},m}\right]
\end{equation}
where the sum is over permutations $\epsilon$ of $1,\dots,m-1$. This agrees with the definition of \cite{fu90} if $M$
is
euclidean.  

\begin{definition}
$\Delta_i:= \Int(\kappa_i,0)$.
\end{definition}

Thus if $A\subset M$ is a smooth domain then $\Delta_i(A,\cdot)$ is the measure on $\partial A$ whose density with
respect to $\vol_{m-1}$ is $\alpha_{m-i-1}^{-1}$ times the $(m-i-1)$st elementary symmetric function of the principal
curvatures of $\partial A$.

Compare also \cite{bebr03} for this definition. Observe that if $M=\R^n$ with the standard metric then the $\Delta_i$
coincide with the curvature measures $\Phi_i$ of \cite{cm}. In this case 
\begin{equation}\label{eq:glob delta}
\glob \Delta_i =   \mu_i 
\end{equation}
where $\mu_i$ is the $i$th {\bf intrinsic volume} \cite{klro97}. 

In euclidean spaces the space of translation-invariant angular curvature measures is stabilized by multiplication by
$\mu_1$:

\begin{theorem}[Angularity theorem, first version] \label{thm_angular_flat}
$$\mu_1\cdot \ang(\R^n)\subset \ang(\R^n).$$
\end{theorem}

\begin{proof}
It is enough to show that  $\mu_1\cdot\Phi\in \ang(\R^n)$ assuming that $\Phi\in\ang(\R^n)$ is homogeneous, say of
degree $k-1$. Hence, we must find a function $g$ on $\Gr_{k}$  such that, given a convex cone  $C$ containing a maximal
subspace  $F\in\Gr_{k}$
\[
 (\mu_1\cdot\Phi)(C,U)=g(F)\vol_{k}(U)\angle_{n-k-1}(C,F)
\]
where $U\subset F$ is any Borel set. By Corollary \ref{cor_product_special}  and a limit argument 
\begin{align*}
 (\mu_1\cdot\Phi)(C,U)&=\int_{\AGr_{n-1}}\Phi(H\cap C,U)d H=\int_{\AGr_{n-1}}\vol(U\cap H) f(F\cap \vec H)\angle(C\cap
\vec H,F\cap \vec H) dH \\
&=\frac 12\int_{S^{n-1}}\int_\R \vol(U\cap (x v+v^\bot)) f(F\cap v^\bot)\angle(C\cap v^\bot,F\cap v^\bot) dx\
d_v\vol_{S^{n-1}}\\
&=\frac 12\vol(U)\int_{S^{n-1}}\cos\theta f(F\cap v^\bot)\angle(C\cap v^\bot,F\cap v^\bot) \ d_v\vol_{S^{n-1}}
\end{align*}
where $\theta\in [0,\pi/2]$ is the angle between $v$ and $F$, and $\vec H$ denotes the linear space parallel to $H$ . The exterior angle of intersections is 
\begin{align*}
 \angle(C\cap v^\bot,F\cap v^\bot)&=\frac{\alpha_{n-k-1}}{\alpha_{n-k}}\vol_{n-k}(S^{n-1}\cap\{a w+b v\colon
a,b>0,w\in\nor(C,F)\})\\
&=\vol_{n-k-1}(\pi_v \nor(C,F))
\end{align*}
where $\pi_v$ denotes the polar projection of $S^{n-1}\setminus\{\pm v\}$ onto $S^{n-1}\cap v^\bot$.  

A generic $v\in S^{n-1}$ is uniquely described as            
\[
 v=\cos\theta \, u+\sin\theta\, w,\qquad u\in F\cap S^{n-1}=:S^{k-1},\quad w\in F^\bot\cap S^{n-1}=:S^{n-k-1}.
\]
In these terms, 
\begin{equation}\label{tPhiuw}
{(\mu_1 \cdot \Phi)(C,U)} = {\vol(U)}\int_{S^{k-1}}f(F\cap u^\bot)\int_0^{\pi/2} h(\theta)\int_{S^{n-k-1}}\vol(\pi_v
\nor(C,F))d_w\vol_{S^{n-k-1}} d\theta d_u\vol_{S^{k-1}}
\end{equation}
for some function $h(\theta)$. By the coarea formula, denoting by $\mathrm{jac}(x,w)$ the jacobian at $x$ of the
projection $\pi_v$ restricted to $\nor(C,F)\subset S^{n-k-1}$,
 \begin{align*}
 \int_{S^{n-k-1}} \vol(\pi_v(\nor(C,F)) 
d_w\vol_{S^{n-k-1}}&=\int_{\nor(C,F)}\left(\int_{S^{n-k-1}}\mathrm{jac}(x,w)d_w\vol_{S^{n-k-1}}\right)
d_x\vol_{\nor(C,F)}\\
&=j(\theta)\vol(\nor(C,F))
\end{align*}
for some function $j(\theta)$. Indeed, by $SO(n-k)-$invariance the integral between brackets is independent of $x$.
Together with \eqref{tPhiuw}, this gives the result.
 \end{proof}

 \begin{definition}
Let $V$ be a $d$-dimensional euclidean vector space. A curvature measure $\Phi\in \curv(V)$ is  a {\bf constant
coefficient curvature measure} if there exists a constant coefficient form $\omega \in \Lambda^d(V^* \oplus V^*) \subset
\Omega^d(V \oplus V)$ such that for all compact convex bodies $K \subset V$ and all Borel sets $U \subset V$ we have 
\begin{displaymath}
 \Phi(K,U)=\int_{N_1(K) \cap \pi^{-1}U} \omega.
\end{displaymath}
where $N_1(K)$ is the disk bundle defined in \emph{(41)} of \cite{befu06}.
\end{definition}

\begin{lemma} \label{lemma_cc_implies_angular}
Every constant coefficient curvature measure is angular. 
\end{lemma}

\proof
By linearity, it suffices to show that any curvature measure arising from a differential form $\omega=\pi_1^*\omega_1
\wedge \pi_2^* \omega_2$ with $\omega_1 \in \Lambda^kV^*, \omega_2 \in \Lambda^{d-k}V^*$ is angular.

Let $P$ be a polytope and denote by $\mathfrak{F}_k(P)$ the set of $k$-faces of $P$. Let $N_1(P,\sigma) \subset \vec
\sigma^\perp$ be the normal cone of $P$ at a face $\sigma$. Then the degree $k$-part of the current $[N_1(P)]$ is given
by 
\begin{displaymath}
 N^k_1(P)=\sum_{\sigma \in \mathfrak{F}_k(P)} [\sigma] \times [N_1(P,\sigma)].
\end{displaymath}

Since $\omega_2$ is of degree $d-k$ and has constant coefficients, 
\begin{displaymath}
 \int_{N_1(P,\sigma)} \omega_2=\frac{\angle(P,\sigma)}{\alpha_{d-k-1}}\int_{\vec \sigma^\perp \cap B(0,1)} \omega_2,
\end{displaymath}
where $ \angle(P,\sigma)$ denotes the outer angle of $P$ at the face $\sigma$.  

It follows that 
\begin{align*}
 \Phi(P,U) & = \sum_{\sigma \in \mathfrak{F}_k(P)} \int_{\sigma \cap U} \omega_1 \int_{N_1(P,\sigma)} \omega_2\\
& =  \sum_{\sigma \in \mathfrak{F}_k(P)} f(\vec \sigma) \vol_k(\sigma \cap U)  \angle(P,\sigma).
\end{align*}
\endproof

\subsection{Lipschitz-Killing curvature measures}\label{sect:LK cms}
On any Riemannian manifold $M^n$ there is a canonical family of curvature measures defined as follows. Let $\iota\colon
M\rightarrow \R^N$ be an isometric embedding (it is not necessary to invoke the Nash embedding theorem here: the
existence of local isometric embeddings is enough to draw this conclusion). The $k$-th Lipschitz-Killing curvature
measure $\Lambda_k$ is defined by
\begin{equation}\label{eq:lamb=delt}
\Lambda_k=\iota^*\Delta_k.
\end{equation}
That this definition is independent of $\iota$ is the substance of  ``Weyl's tube formula" \cite {weyl}. More precisely,
it is a consequence of the following description of $\Lambda_k$ in terms of a pair of differential forms
$(\Psi_k,\Phi_k)\in \Omega^n(M)\times\Omega^{n-1}(SM)$. Let $e_1(\xi),\ldots, e_n(\xi)=\xi$ be a local moving frame on
$M$ defined for $\xi\in SM$. Let $\theta_i$ be the associated coframe, and $\omega_{i,j}$ the connection forms. Finally,
let $\Omega_{ij}$ denote the curvature forms. Then (cf. \cite{cms})
\begin{equation}
\label{LKinterior} \Psi_k=
\frac{1}{k!(\frac{n-k}2)!(4\pi)^{\frac{n-k}2}} \sum_\epsilon \sgn(\epsilon)
\Omega_{\epsilon_1\epsilon_2} \wedge \cdots \wedge \Omega_{\epsilon_{n-k-1} \epsilon_{n-k}} \wedge
\theta_{\epsilon_{n-k+1}} \wedge \cdots \wedge
\theta_{\epsilon_n} \in \Omega^n(SM)
\end{equation}
when $n-k$ is even, and $\Psi_k=0$ if $n-k$ is odd.  
Note that the sum above may be partitioned into sums of $k!(\frac{n-k}2)!2^{\frac{n-k}2}$ formally identical terms. 

Similarly
\begin{equation}                                                                                                       
\Phi_k=\sum_{2i\le n-k-1} d_{n,k,i} \Phi_{k,i}
\end{equation}
where
\begin{equation}
 \Phi_{k,i}:= \sum_\epsilon \sgn(\epsilon) \Omega_{\epsilon_1\epsilon_2} \wedge \cdots \wedge
\Omega_{\epsilon_{2i-1}\epsilon_{2i}} \wedge
\omega_{n,\epsilon_{2i+1}} \wedge \cdots \wedge \omega_{n,\epsilon_{n-k-1}} \wedge \theta_{\epsilon_{n-k}} \wedge
\cdots \wedge \theta_{\epsilon_{n-1}}\label{phik}
\end{equation}
for some constants $d_{n,k,i}$. 
Both $\Psi_k$ and $\Phi_k$, and in fact all of the $\Phi_{k,i}$, are independent of the moving frame, and hence globally
defined differential forms on $SM$. Moreover $\Psi_k$ is the pullback of a differential form defined on $M$, which we
denote again by $\Psi_k$.

It will be helpful to ascertain in a reliable way the values of the constants in \eqref{LKinterior}. 
To do so we take  $M$ to be a sphere, and use
 \[                          
 \mu_k(S^n) = 2 \mu_k(B^{n+1}) = 2 \binom{n+1} {k} \frac{\omega_{n+1}}{\omega_{n-k+1}}, \quad n-k \text{ even}\\
 \]
which is Theorem 9.2.4 of \cite{klro97}. We will not use the explicit values of the constants $d_{n,k,i}$.

A more general form of the 
restriction property \eqref{eq:lamb=delt} follows easily from the definition.

\begin{lemma}\label{lem:lk invariance} Let $M,N$ be smooth Riemannian manifolds, and denote the associated
Lipschitz-Killing curvature measures by $\Lambda_i(M,\cdot), \Lambda_i(N,\cdot)$ respectively. If $j:N \to M$ is an
isometric embedding, then 
$$
\Lambda_i(N,\cdot) = j^*\Lambda_i(M,\cdot).
$$
\end{lemma}

\begin{proof} If $\iota:M \to \R^N$ is an isometric embedding of $M$ then $\iota \circ j$ is an isometric embedding of
$N$. Therefore $\Lambda_i(N,\cdot) = (\iota \circ j)^* \Delta_i = j^*(\iota^* \Delta_i) = j^* \Lambda_i(M,\cdot)$.
\end{proof}
 
\begin{proposition}\label{LKangular}
The Lipschitz-Killing curvature measures are angular.
\end{proposition}

\begin{proof} In fact we show that each $\Phi_{k,i}$ defines an angular curvature measure.
We observe that  the $\Omega_{ij}$ and the $\theta_i$ are horizontal with respect to the Riemannian connection, and the
$\omega_{n,j}$ are vertical.
Thinking of these as translation-invariant elements of $\Omega^*(ST_xM)$, we find that for any $x \in M$ the associated
curvature measure $\tau_x(\Int(\Phi_{k,i}))\in\curv(T_xM)$  has degree $k+2i $.

Fixing $x \in M$, let $P \subset T_xM$ be a polytope and $F\subset P$ a face of dimension $k+2i$. 
Let us choose for each $(p,v)\in N(P)\cap\pi^{-1}F$ an orthonormal frame  $p;e_1,\ldots,e_n$ with $e_1,\ldots,
e_{k+2i}\in T_pF$ and $e_n=v$.
Then $\omega_{n,1}=\cdots=\omega_{n,k+2i}=0$ when restricted to $N(P)\cap\pi^{-1}F$, while the restriction of
$\omega_{n,k+2i+1}\dots \omega_{n,n-1}$ is the volume form of the sphere of $F^\perp$.
Hence for $U\subset F$
\begin{align*}
&\tau_x(\Int(\Phi_{k,i}))(P,U)=\int_{N(P)\cap \pi^{-1}U} \Phi_{k,i}\\
&=c\angle(P,F) \int_U \sum_\sigma \sgn(\sigma)
\Omega_{\sigma_1\sigma_2}\cdots\Omega_{\sigma_{2i-1}\sigma_{2i}}\theta_{\sigma_{2i+1}}\cdots\theta_{\sigma_{2i+k}}\qquad
\end{align*}
where $\sigma$ runs over all permutations of $\{1,\ldots,k+2i\}$.
\end{proof}

\subsection{The Lipschitz-Killing algebra}\label{sect:LK algebra} Alesker has observed that the globalizations of the Lipschitz-Killing
curvature measures have a special place in the integral geometry of Riemannian manifolds. Recall first that the algebra
of translation-invariant, $SO(n)$-invariant valuations on $\Rn$ is isomorphic to $\R[t]/(t^{n+1})$ \cite{ale03, fu06},
where by convention we identify $t$ with $\frac 2\pi \mu_1$. In general \cite{befu11}
\begin{equation}\label{eq:exp t}
t^i = \frac {i!\omega_i}{\pi^i} \mu_i, {\quad  \mbox{or } \quad\exp(\pi t) = \sum \omega_i \mu_i.}
\end{equation}
Since the restriction map $\V(N) \to \V(M)$ corresponding to an embedding of manifolds $M\to N$ is a homomorphism of
algebras (\cite{ale06}, Thm.2.14), it follows that any isometric embedding $M\to \Rn$ yields a map $l:\R[t]/(t^{n+1})
\to \V(M)$. Each monomial $t^i$ is an element of filtration $i$, and yields a fixed multiple of the $i$-dimensional
volume when applied to a smooth submanifold of dimension $i$. It follows that if $\dim M = k$ then $ l$ factors through
an injection $\R[t]/(t^{k+1}) \to \V(M)$.

 Since any Riemannian manifold $M^k$ admits such an isometric embedding into some $\Rn$, and by \eqref{eq:lamb=delt} the
Lipschitz-Killing curvature measures of $M$ are the restrictions to $M$ of the Lipschitz-Killing curvature measures of
$\Rn$, it follows that $\V(M^k)$ contains a canonically embedded copy of $\R[t]/(t^{k+1})$, the {\bf Lipschitz-Killing
algebra} of $M$.
 
Gray  has shown that the values of the Lipschitz-Killing valuations of K\"ahler manifolds may be expressed in terms of
Chern forms:
\begin{lemma}\label{lemma_t_chern} Let $N$ be a compact K\"ahler manifold of complex dimension $n$. Then 
\begin{equation}
t^{2k}(N) = \pi^{-k} \binom{2k} k \int_{N} \ch_{n-k}(N)\wedge \kappa^{k},
\end{equation}
where $\ch_i$ is the $i$th Chern form and $\kappa$ the K\"ahler form of $N$.
\end{lemma}
\proof Except for the constant factor, the formula is Lemma 7.6 in \cite{gray}. To find out this constant, it is enough
to take $N= (\CP^1)^n$. Here  $\CP^1$ is metrized as usual as the sphere of radius $\frac 1 2$, intrinsic diameter
$\frac \pi 2$ and area $\pi$, and has Chern form $\ch_1(\CP^1)=\frac{2}{\pi} d\area$. By using \eqref{eq:exp t} and the
product formula for intrinsic volumes (\cite{klro97}, Prop. 4.2.3) the result follows after simple computations.
\endproof

\begin{corollary}\label{cor:t^k(cpn)}For $\lambda>0$,
\begin{equation}\label{eq:t2k cpn}
t^{2k}(\CP^n_\lambda) = \binom{2k}{k}\binom{n+1}{k+1}\lambda^{-k}.
\end{equation}
\end{corollary}
\proof
This is a direct consequence of the previous lemma, and the following expression of the Chern forms of $\CP^n_\lambda$
(cf. Corollary 6.25 in \cite{gray})
\[
 \ch_i(\CP_\lambda^n)=\binom{n+1}{i}\left(\frac{\lambda}{\pi}\kappa\right)^i.
\]
\endproof


\section{Global kinematic formulas for complex space forms}\label{sect:global} 
The main result of this section is Theorem \ref{thm:pkf lambda} below, which states that the principal kinematic formula
in $\CPn_\lambda$ is formally independent of $\lambda$ when expressed in terms of certain geometrically natural bases
for $\V^n_\lambda$. Thus the explicit form of this formula for $\lambda =0$, given in \cite {befu11},  holds also for
general $\lambda$. 

The geometrically natural bases 
are the  $\mu_{kq}^\lambda, \tau_{kq}^\lambda$ of Definition \ref{def:tilde mu}. We  give enough information that the
multiplication table for $\V^n_\lambda$ may be determined in these terms--- although we do not give this table
explicitly, it is essentially straightforward to do so in any given dimension $n$. Using the multiplicativity property
\eqref{eq:k mult}, it follows that the entire kinematic operator $k$ may then be given in these terms.

Let $\mathcal{V}^n_\lambda$ denote the algebra of $G_\lambda$-invariant valuations in $\CP^n_\lambda$. We
will show that it is generated by the generator $t$ of the Lipschitz-Killing algebra of
$\CP^n_\lambda$, together with 
\begin{equation}\label{eq:def s}
s := \int_{\AGr^{\C}_{n-1}} \chi(\cdot \cap P ) \, dP
\end{equation}
where $\AGr^{\C}_{n-1}$ denotes the Grassmannian of all totally geodesic complex hyperplanes $P \subset
\CP^n_\lambda$, and $dP$ is the $G_\lambda$-invariant measure, normalized so that if $E\subset \CP^n_\lambda$
is a 2-dimensional differentiable polyhedron contained in some totally geodesic complex line then
$$
s(E) = \frac{\area E}{\pi}.
$$
Under this normalization, for $\lambda >0$
\begin{equation} \label{eq:sk_total}
 s^k(\CP^n_\lambda)=\frac{n-k+1}{\lambda^k}.
\end{equation}

Thus if $\lambda=1$ then $\AGr^{\C}_{n-1} $ is the projective space dual to $\CP^n_1$ and $dP$ is the
invariant probability measure. If $\lambda = 0$ then this definition of $s$ is identical to that in
\cite{fu06, befu11}. 

Considering the sequence of restriction maps $\V^{n+1}_\lambda\rightarrow \V^{n}_\lambda$, we define
$\V^\infty_\lambda:=\varprojlim \V^n_\lambda$.

\subsection{Invariant curvature measures and valuations}\label{subsect:invariant}
We denote by $\curv^{U(n)}$ the space of curvature measures of $\C^n$  invariant under $\overline{U(n)}$. Its
inverse limit under the restriction maps is denoted $\curv^{U(\infty)}$.

Let  $\alpha,\beta,\gamma,\theta_0,\theta_1,\theta_2\in \Omega^*(S\C^n)$ be the {differential} forms defined in \cite{befu11}. It was
shown in \cite{pa02} that the subalgebra $\Omega^*(S\C^n)^{\overline{U(n)}}$ of invariant forms  is generated by these
elements, together with $d\alpha$. Since $\alpha, d\alpha$ vanish identically on normal cycles, it follows that 
$\curv^{U(n)}$ is spanned by integration of the forms  
\begin{align}
\beta_{kq}:&=c_{nkq}\,\beta \wedge \theta_0^{n-k+q} \wedge \theta_1^{k-2q-1} \wedge \theta_2^q,\quad k > 2q,\\
\gamma_{kq}:&=\frac  {c_{nkq}} 2\,\gamma \wedge \theta_0^{n-k+q-1} \wedge \theta_1^{k-2q} \wedge \theta_2^q, \quad
n>k-q,
\end{align}   
where we set for convenience
\begin{displaymath}
 c_{nkq}:=\frac{1}{q!(n-k+q)!(k-2q)! \omega_{2n-k}}.
\end{displaymath}
Put $B_{kq}:=\Int\beta_{kq}, \Gamma_{kq}:=\Int\gamma_{kq}$. Globalizing these curvature measures yields  the
so-called \emph{hermitian intrinsic volumes}  (cf. \cite{befu11})
 \[
 \mu_{kq}:=\glob(B_{kq})=\glob(\Gamma_{kq})\in \valun(\Cn),\qquad 0,k-n\leq q\leq \frac k2\leq n.
\]
They form a basis of $ \valun$. Since $B_{kq}\neq \Gamma_{kq}$ (cf. \eqref{B_on_balls}, \eqref{Gamma_on_balls} below)
we deduce that $\{B_{kq}\}\cup\{\Gamma_{kq}\}$ is a basis of $\curvun$. Next we introduce a more convenient basis.

\begin{definition} \label{def_delta}
 Define curvature measures $\Delta_{kq} \in \curv^{U(n)}, \ \max\{0,k-n\} \leq q \leq \frac{k}{2}<n$ by 
\begin{align*}
  \Delta_{kq} :&= \frac{1}{2n-k}(2(n-k+q)\Gamma_{kq}+(k-2q)B_{kq})\\
\Delta_{2n,n} :&= \vol_{2n}.
\end{align*}
Thus $\Delta_{2q,q} = \Gamma_{2q,q}$ and $\Delta_{k,k-n} = B_{k,k-n}$.
We define also for $k> 2q, q> k-n$
\begin{align*}
N_{kq} :&= \Delta_{kq} - B_{kq}\\
& = \frac{2(n-k+q)}{2n-k}(\Gamma_{kq} - B_{kq}).
\end{align*}
\end{definition}
The family $\{\Delta_{kq}\}\cup\{N_{kq}\}$ is a basis of $\curvun$. 
\begin{proposition}\label{angulardelta}
For $\Phi \in  \curv^{U(n)}$ the following are equivalent:
\begin{enumerate}
\item[(i)] $\Phi\in\mathrm{span}\{\Delta_{kq}\}$, 
\item[(ii)] $\Phi$ is a constant coefficient curvature measure,
\item[(iii)] $\Phi$ is angular.
\end{enumerate}
\end{proposition}

\proof
(i) $ \implies$  (ii): Let 
\begin{displaymath}
 \omega:=\frac{1}{2n-k}(2(n-k+q)\gamma_{kq}+(k-2q)\beta_{kq}).
\end{displaymath}
Then for $f \in C^\infty(\C^n)$ we have 
\begin{displaymath}
 \int_{N(K)} \pi^*f \wedge \omega=\int_{\partial N_1(K)} \pi^*f \wedge \omega=\int_{N_1(K)} (\pi^* df \wedge \omega
+\pi^*f \wedge d\omega).
\end{displaymath}
An easy computation shows that $\pi^* df$ and $\omega$ both vanish on the radial vector field of the disc bundle of
$\C^n$. Since $d\omega=d\beta_{kq}=d\gamma_{kq}$ has constant coefficients, the proof is finished. 

(ii) $ \implies$  (iii) by Lemma \ref{lemma_cc_implies_angular}.  

(iii) $ \implies$  (i): Let $\Phi$ be angular. Since the restricted globalization map $\mathrm{span}\{\Delta_{kq}\} \to
\val^{U(n)}$ is onto, there exists $\Xi \in \mathrm{span}\{\Delta_{kq}\}$ such that $\glob(\Phi-\Xi)=0$. By the two
implications above, the curvature measure $\Psi=\Phi-\Xi$ is angular. Let $\Psi=\sum_i \Psi_i$ be the decomposition of
$\Psi$ into $i$-homogeneous parts.  Each $\Psi_i$ is angular and is hence determined by the Klain function of
$\glob(\Psi_i)$. Therefore, all $\Psi_i$ vanish and $\Phi=\Xi$.
\endproof

We will denote by $\glob_\lambda\colon\curv^{U(n)}\simeq \curvinfty^{G_\lambda}(\CPn_\lambda)\to \mathcal{V}^n_\lambda$
the globalization map given by $\glob_\lambda=\glob\circ \tau_\lambda^{-1}$ where $\tau =\tau_\lambda$ is the transfer
map \eqref{eq:curv sim 1}.  We will also use the notation 
$$
[\Phi]_\lambda:=\glob_\lambda(\Phi).
$$
Clearly $\glob_0 = \restrict\glob{\curvun}$, and
\begin{equation}\label{eq:ker glob0}
\ker \glob_0 = \spann\{N_{kq}\}.
\end{equation}
Thus we put
$$
\ang^{U(n)}:=\spann\{\Delta_{kq}\}, \quad \nul^{U(n)}:=\spann\{N_{kq}\}.
$$
\begin{lemma}\label{lem:gamma beta translation} If $k >2q$ then
\begin{align}
\label{eq:n = b}
[N_{kq}]_\lambda &= -\lambda \frac{q+1}{\pi} [B_{k+2, q+1}]_\lambda\\
\label{eq:BtoDelta}
[B_{kq}]_\lambda&=\sum_{i\geq 0} \frac{\lambda^i(q+i)!}{\pi^iq!}[\Delta_{k+2i,q+i}]_\lambda
\end{align}
for all $\lambda \in \R$.
\end{lemma}
\begin{proof} Equation \eqref{eq:n = b} is Proposition 2.6 of \cite{ags}. Equation \eqref{eq:BtoDelta} follows by
recurrence.
\end{proof}

\begin{corollary}\label{cor:Delta is basis} For every $\lambda \in \R$, the valuations $\glob_\lambda(\Delta_{kq}),
\max\{0,k-n\} \leq q \leq \frac{k}{2}\le n$, constitute a basis of $\V^n_\lambda$.
\end{corollary} 
\proof  For $\lambda=0$ this was proved in \cite{befu11}. For general $\lambda$, equations  \eqref{eq:n = b} and \eqref{eq:BtoDelta}
show that $[\Delta_{kq}]_\lambda$ span $\V^n_\lambda$. By Corollary 3.1.7 of \cite{ale05b}, $\dim \V^n_\lambda=\dim
\V^n_0{ =\dim \ang^{U(n)}}$.
\endproof

\begin{proposition} \label{prop_kernels_glob}  \mbox{}
 \begin{enumerate}
\item[(i)]\label{eq:ker glob}
$\ker \glob_\lambda = \spann\left\{ N_{kq} + \lambda\frac{q+1}{\pi}B_{k+2,q+1} : k>2q, \ q>k-n\right\}.$
\item[(ii)]\label{item:ker intersect} If $\lambda \ne 0$ then
 \begin{equation}\label{eq:ker intersect}
  \ker \glob_0 \cap \ker \glob_\lambda = \{0\}.
 \end{equation}
 \end{enumerate}
\end{proposition}
\begin{proof}  
Conclusion  (i) is immediate from \eqref{eq:n = b} and Corollary \ref{cor:Delta is basis}.

(ii): Any $\Phi\in\ker \glob_\lambda$ has the form $\Phi=\sum_{k,q} c_{kq}\Psi_{kq}$ with $\Psi_{kq}:= N_{kq} +
\lambda\frac{q+1}{\pi}B_{k+2,q+1}$. Thus 
\[
 \glob_0\Phi=\frac{\lambda}{\pi} \sum_{k,q}c_{kq}(q+1)\mu_{k+2,q+1},
\]
for some constants $c_{kq}$
so if this is zero then $c_{kq}=0$ for all $k,q$, by linear independence of the $\mu_{k,q}$.
\end{proof}

Let $\iota_\lambda: \CP^n_\lambda\to \CP^{n+1}_\lambda$  be a totally geodesic embedding, and let us consider the
restriction map $\iota_\lambda^*: \curvinfty^G(\CP^{n+1}_\lambda)\to \curvinfty^G(\CP^n_\lambda)$. Since $\exp\circ
d\iota_\lambda=\iota_\lambda\circ\exp$, the description of the transfer map given after Proposition
\ref{prop_isom_curv_general} implies that the following diagram commutes,
\begin{displaymath}
\xymatrix{\curvinfty^G(\CP^{n+1}_\lambda) \ar[r]^-{\iota_\lambda^*} \ar[d]_{\tau_{\lambda}}& \curvinfty^G(\CPn_\lambda)
\ar[d]_{\tau_\lambda}  \\
\curv^{U(n+1)}  \ar[r]^-{\iota_0^*}& \curv^{U(n)}}
\end{displaymath}
In other words, the map $r=\tau_\lambda\circ\iota_\lambda^*\circ(\tau_\lambda)^{-1}$ is independent of $\lambda$.
\begin{lemma}
\begin{align}\label{rdelta}
 r(\Delta_{kq})&=\Delta_{kq}, &  \max(0,k-n)\leq q\leq \frac k2\leq n
\\
\label{rbeta}
 r(B_{kq})&=B_{kq}, &  \max(0,k-n)\leq q < \frac k2\leq n\\
\label{rdeltavanish}
  r(\Delta_{k,k-n-1})&=r(B_{k,k-n-1})=0,  &  n+1\leq k\leq 2n+2.
\end{align}
 \end{lemma}
\begin{proof}
The angularity condition is clearly invariant under restriction to linear subspaces. Since restriction commutes with
globalization, and $\glob_0$ is injective on $\mathrm{Ang}^{U(n)}$, equations \eqref{rdelta} and \eqref{rdeltavanish}
follow from the invariance of $\mu_{kq}$ under $\C$-linear restrictions.

By \eqref{eq:BtoDelta} we have (abusing notation so that $r$ denotes the restriction map for both curvature measures and
valuations)
\begin{align*}
[r  (B_{kq})]_\lambda&=r \left( [B_{kq}]_\lambda\right)=\sum_{i\geq 0} \frac{\lambda^i(q+i)!}{\pi^i q!}r \left( [\Delta_{k+2i,q+i}]_\lambda\right)
\\
&=\sum_{i\geq 0} \frac{\lambda^i(q+i)!}{\pi^i q!}[\Delta_{k+2i,q+i}]_\lambda=[B_{kq}]_\lambda.      
\end{align*}
Hence $r ( B_{kq})-B_{kq}\in\ker\glob_\lambda$ for all $\lambda$. By \eqref{eq:ker intersect}, equation \eqref{rbeta}
follows.
\end{proof}

We deduce that
\[
 \Delta_{kq},\ 0\leq q\leq \frac k2; \quad N_{kq}, \  0\leq q < \frac k2
\]
define a basis of $\curv^{U(\infty)}$. We let $\ang^{U(\infty)}$ be the inverse limit of the spaces $\ang^{U(n)}$. Thus
$\ang^{U(\infty)}$ is the subspace of $\curv^{U(\infty)}$ spanned by the $\Delta_{kq}$. 

We distinguish some canonical valuations. 
\begin{definition}\label{def:tilde mu} For $\max\{0,k-n\}\leq q\leq \frac{k}{2}\leq n\leq\infty$ we set
\begin{align*}
\mu^\lambda_{kq} &:= [B_{kq}]_\lambda \in \mathcal{V}^n_\lambda,\quad k>2q,\\
\mu^\lambda_{2q,q}&:= \sum_{i=0}^\infty \left(\frac{\lambda}{\pi}\right)^i\frac{(q+i)!}{q!} [\Gamma_{2q+2i,
q+i}]_\lambda \in \mathcal{V}^n_\lambda,\end{align*}
which defines a basis of $\mathcal{V}^n_\lambda$. Another basis is
\begin{align*}\tau^\lambda_{kq}&:=\sum_{i=q}^{\lfloor\frac k2\rfloor}\binom{i}{q}\mu^\lambda_{ki} \in
\mathcal{V}^n_\lambda.
\end{align*}
\end{definition}

The valuations
$\mu_{kq}^0\in \mathcal{V}^n_0= \Val^{U(n)}$ coincide with the $\mu_{kq}$ studied in \cite{befu11}. For $\lambda \ne 0$
they coincide with the basis elements studied in \cite{ags} in the cases $k>2q$. The  rationale for the present 
definition of $\mu_{2q,q}^\lambda$ is equation \eqref{eq:BtoDelta}: with this definition, the corresponding relation
also holds for $k=2q$. As we will see, the resulting {bases $\mu_{kq}^\lambda, \tau_{kq}^\lambda$ have} many remarkable properties. 

The following point is obvious from the definitions, but will be crucial for Theorem \ref {thm:1st iso} below.

\begin{lemma}\label{lem:mus that vanish}
{ The kernel of the restriction map $\V^\infty_\lambda\to\V^n_\lambda$ is spanned by the valuations $\mu_{k,q}^\lambda$
with $q<k-n$ or $k>2n$.} In particular,  {if $k >n$ then $\mu_{k,0}^\lambda = 0$ in $\V^n_\lambda$.}
\end{lemma}

For future reference we restate equations \eqref{eq:n = b} and \eqref{eq:BtoDelta} in this new notation. 

\begin{lemma}\label{lem:def mu}
\begin{align*}
[\Delta_{kq}]_\lambda &=\mu^\lambda_{kq}-\lambda \frac{q+1}{\pi}\mu^\lambda_{k+2,q+1}\\
\mu^\lambda_{kq}&=\sum_{i\geq 0} \frac{\lambda^i(q+i)!}{\pi^i q!}[\Delta_{k+2i,q+i}]_\lambda.
\end{align*}
\end{lemma}

\proof
The first equation is an easy consequence of the second, which follows from \eqref{eq:BtoDelta} and the definitions. 
\endproof 

\subsection{Expansion of $t^k$ in terms of the $\tau_{kq}^\lambda$}\label{sect:t & tau}
In this section we show how to express the elements of the Lipschitz-Killing algebra of $\CP^n_\lambda$ as
linear combinations of the bases above. It turns out that this is readily accomplished in terms of exponential
generating functions. For general information on this subject we refer to \cite{gfology}. 

The geometric foundation of our calculation is the following well known fact.

\begin{lemma}\label{lem:kono curv} Let $e_i, i =1,\dots, n$ be a local hermitian-orthonormal frame for $\CP^n_\lambda$,
and put $e_{\bar i} := \sqrt{-1} e_i$. Let $\theta_1,\dots,\theta_{ n}, \theta_{\bar 1},\dots, \theta_{\bar n}$ be the
dual coframe. Then the curvature forms of $\CPn_\lambda$ at $x$ are given by
\begin{align*}
\Omega_{i, \bar i} &= 4 \lambda \theta_i\wedge \theta_{\bar i } + 2\lambda \sum_{j\ne i} \theta_j \wedge \theta_{\bar j}
\\
\Omega_{\alpha,\beta} &= \lambda(\theta_\alpha \wedge \theta_\beta + \theta_{\bar \alpha} \wedge \theta_{\bar \beta} ),
\quad \alpha , \beta \in \{1,\bar1,\dots, \bar n\}, \ \alpha \ne  \bar \beta
\end{align*} 
where $\bar{\bar i} :=i$.
\end{lemma}
\begin{proof} Cf. \cite{ko-no}, Chapter IX, Proposition 7.3.
\end{proof}

\begin{theorem} \label{egfs for t^k}
Define 
\begin{align}
g_i(\xi,\eta)&:= \xi^i \left(1- \xi \right)^{-i-\frac 1 2} \left(1-   \eta\right)^{-\frac 1 2}\\
h_i(\xi,\eta)&:= \xi^i \left(1-   \xi\right)^{-i-\frac 3 2} \left(1-   \eta\right)^{-\frac 1 2}.
\end{align}
Then
$$
\binom{2i} i \lambda^{-i}g_i\left(\frac{ \lambda x}\pi,\frac{\lambda  y} \pi\right), \quad \frac{2^{2i+1}}\pi
\lambda^{-i} h_i\left(\frac{\lambda  x}\pi,\frac {\lambda y} \pi\right)$$
 are the exponential generating functions of $t^{2i}, t^{2i+1}$ in terms of the $\tau_{kp}^\lambda$.
In other words,
\begin{align}
t^{2i}&=\binom{2i}{i}\lambda^{-i}\sum_{k,p=0}^\infty\left(\frac\lambda\pi\right)^{k+p}\left.\frac{\partial^{k+p}
g_i}{\partial^k\xi \partial^p \eta}\right|_{\xi=\eta=0} \tau_{2k+2p,p}^\lambda \\ 
t^{2i+1}&=\frac{2^{2i+1}}{\pi}\lambda^{-i}\sum_{k,p=0}^\infty\left(\frac\lambda\pi\right)^{k+p}\left.\frac{\partial^{k+p
} h_i}{\partial^k\xi \partial^p \eta}\right|_{\xi=\eta=0} \tau_{2k+2p+1,p}^\lambda  .
\end{align}
\end{theorem}
\begin{proof}By \eqref{eq:glob delta}, \eqref{eq:lamb=delt} and \eqref{eq:exp t}, 
\begin{equation}\label{eq:tj in terms of Lambda}
t^j= \frac{j!\omega_j}{\pi^j} [\Lambda_j]_\lambda.
\end{equation}
 By Propositions \ref{LKangular} and \ref{angulardelta}, each $\Lambda_j$ may be expressed uniquely as a linear
combination of the $\Delta_{kq}$. Thus our first step is to carry this out explicitly; globalizing, we will obtain the
expansion of $t^j$ in terms of the $[\Delta_{kq}]_\lambda$.

Let $E^{kq}=\mathbb R^{k-2q}\oplus\C^q\subset \C^n\simeq T_x\CP_\lambda^n$. 
Recalling the transfer map $\tau$ of Proposition \ref{prop_isom_curv}, for $k-j$ even we define $\pf^j_{kq}$ by
\[
\restrict{\tau(\Lambda_j)}{E^{kq}} ={\frac{1}{(2\pi)^{\frac{k-j}{2}}}}\pf_{kq}^j \restrict{\vol_k}{E^{kq}}.
\]
Thus
\begin{align}
\label{egfevenlambda} \tau(\Lambda_{2i}) &=(2\pi)^i\sum_{m,p\geq 0} \frac{1}{(2\pi)^{m+p}} \pf_{2m+2p,p}^{2i}
\Delta_{2m+2p,p}\\
\label{egfoddlambda} \tau(\Lambda_{2i+1})&=  (2\pi)^i \sum_{m,p\geq 0} \frac{1}{(2\pi)^{m+p}} \pf_{2m+2p+1,p}^{2i+1}
\Delta_{2m+2p+1,p} 
\end{align}
since $\restrict{\Delta_{ij}}{E^{kq}}= \delta_k^i\delta_q^j\restrict{\vol_k}{E^{kq}}$.

We study first the exponential generating functions $G_i,H_i$ for the constants $\pf_{kp}^j$, i.e.
 \begin{align}
G_i(x,y) & := \sum_{m,p=0}^\infty \pf^{2i}_{2m+2p,p} \frac{x^m y^p}{m! p!}\\
H_i(x,y) & := \sum_{m,p=0}^\infty \pf^{2i+1}_{2m+2p+1,p} \frac{x^m y^p}{m! p!}.
\end{align}

\begin{lemma} \label{GH formulas}
\begin{align}
\label{GH1}G_i(x,y)&= \frac{(x+y)^i}{i!} \left(1- 2\lambda (x+y)\right)^{-i-\frac 1 2} \left(1-2\lambda y\right)^{-\frac
3 2}\\
\label{GH2}H_i(x,y)&= \frac{(x+y)^i}{i!} \left(1- 2\lambda (x+y)\right)^{-i-\frac 3 2} \left(1-2\lambda y\right)^{-\frac
3 2}.
\end{align}
\end{lemma}

\begin{proof}
We claim that
\begin{align} 
\label{even to odd}  G_{i}&=(1-2\lambda x-2\lambda y) H_{i}\\
\label{odd to even} H_i&=(1-2\lambda x-2\lambda y) \frac{\partial G_{i+1}}{\partial x} - \lambda G_{i+1} .
\end{align}

To prove these claims, we establish first the relation
\begin{equation}\label{primitive relation}
\pf^l_{kp} = \pf^{l-1}_{k-1,p} + (k-2p-1)\lambda \pf^l_{k-2,p} + 2p\lambda \pf^l_{k-2,p-1}, \quad k>2p.
\end{equation}

By the remarks following Proposition \ref{prop_isom_curv_general}, and recalling Lemma \ref{lem:lk invariance}, in order
to find $\restrict{\tau(\Lambda_l)}{E^{k,p}}$ it is enough to calculate the density at $x$ of the $l$th
Lipschitz-Killing curvature measure of the Riemannian submanifold $\exp_x(E^{kp})$. 

To accomplish this,  we choose the frame $\{e_i,e_{\bar i}\}$ as in Lemma \ref{lem:kono curv} so that
$e_1,\dots,e_{k-2p}$ span the totally real subspace orthogonal to the complex $p$-plane  $P \subset E^{kp}$, and
$e_{k-2p+1}, e_{\overline{k-2p+1}},\dots,e_{{k-p}},e_{\overline{k-p}}$ span $P$ . 
Let $\{\theta_i\}_{i\in I}$ be the associated coframe, and let $\{\Omega_{ij}\}_{i,j\in I}$ be the curvature forms
where $$I=\{1,\dots,k-2p, k-2p{+1},\overline{k-2p{+1}}\dots,k-p,\overline{k-p}\}.$$ Then, by \eqref{LKinterior},
\begin{equation}\label{eq:pf1}
 \pf_{kp}^l d\vol_{E^{kp}} =\frac{1}{l!(\frac{k-l}2)!2^{\frac{k-l}2}}\sum_\epsilon \sgn(\epsilon)
\restrict{\left[\theta_{\epsilon_1} \wedge \cdots
\wedge \theta_{\epsilon_l} \wedge \Omega_{\epsilon_{l+1}\epsilon_{l+2}} \wedge \cdots \wedge \Omega_{\epsilon_{ k-1 }
\epsilon_k } \right ] } { E^ { k p}}
\end{equation} 
if $k-l$ is even, and $\pf_{kp}^l=0$ otherwise, where $\epsilon$ ranges over all bijections $\{1,\dots,k\}\to I$ and
the sign is determined by an appropriate identification of $\{1,\dots,k\}$ with $I$. By the Gauss equation,
the curvature forms at $x$ of our submanifold are simply the restrictions of the curvature forms of the ambient space
$\CPn_\lambda$.
Singling out the first coordinate, we partition the terms in \eqref{eq:pf1} into three groups:
\begin{itemize}
\item those including the factor $\theta_1$,
\item those including a factor $\Omega_{1j}, \ j \le k-2p$,
\item those including a factor $\Omega_{1j}$ or $\Omega_{i\bar j}$, $\ j>k-2p$.
\end{itemize}
Each group yields some multiple of the volume form of $E^{kp}$. 

Clearly the first group gives $\theta_1\wedge(\pf^{l-1}_{k-1,p} \, d\vol_{E^{k-1,p}})$. 

Among the second group, let us fix an index $j$ and consider the terms including  $\Omega_{1j}$. By Lemma \ref{lem:kono
curv}, the sum of these terms is $\lambda\theta_1 \wedge \theta_j\wedge(\pf^{l}_{k-2,p} \, d\vol_{E^{k-2,p}})=
\lambda\pf^{l}_{k-2,p} \, d\vol_{E^{kp}}$. Since there are $k-2p-1$ possible choices of the index $j$, this yields the
second term of \eqref{primitive relation}. 

The last term of \eqref{primitive relation} is  accounted for similarly by the last group.

Writing out the  {exponential generating functions} for the odd and even cases using \eqref{primitive relation}, we
obtain using the rules described in section 2.3 of \cite{gfology}
\begin{align}
H_{i} &= G_i + \lambda(2x+2y)H_i,\\
G_{i+1} &=  \int H_i \, dx + \lambda(2x+2y) G_{i+1} - \lambda\int G_{i+1} \, dx.
\end{align}
The first relation is \eqref{even to odd}, and differentiating the second relation we get \eqref{odd to even}.

A direct calculation shows that the given functions \eqref{GH1}, \eqref{GH2} satisfy the relations \eqref{even to odd},
\eqref{odd to even}. It remains to show that \eqref{GH1} satisfies the correct initial conditions, i.e. that if we
substitute $x=0$ in the given expression then
\begin{equation}\label{G relation}
 G_i(0,y) :=\sum_p \frac{\pf^{2i}_{2p,p}}{p!} y^p=  \frac{ y^i}{i!}(1-2\lambda y)^{-i-2} , \quad i =0,1,2,\dots.
 \end{equation}

Recall the relation
\begin{equation}
t^{2i}(\CP_\lambda^k)= \binom {2i}{i}\binom{k+1}{i+1} \lambda^{-i}
\end{equation}
of Corollary \ref{cor:t^k(cpn)}. Together with \eqref{egfevenlambda} this gives
\begin{equation}
\pf^{2i}_{2p,p} = \frac{p!2^{p-i}}{i!}\binom{p+1}{i+1}\lambda^{p-i}.
\end{equation}
On the other hand, the right hand side of \eqref{G relation} is
\begin{equation}\label{Gisum}\frac{y^i}{i!}\sum_n \binom{n+i+1}{n} (2\lambda y)^n 
= \frac 1 {i!}\sum_{p\ge i} 2^{p-i}\binom{p+1}{i+1}y^p  \lambda^{p-i}
\end{equation}
which concludes the proof of the lemma.
\end{proof}

Now we may finish the proof of Theorem \ref{egfs for t^k}. From Lemma \ref{GH formulas} and equations
\eqref{eq:tj in terms of Lambda}, \eqref{egfevenlambda}, \eqref{egfoddlambda} we deduce that
the exponential generating function for $t^{2i}$ in terms of the  $[\Delta_{2m+2p,p}]_\lambda$ is
$$\binom{2i} i \lambda^{-i}g_i\left(\lambda\frac{ x+y}\pi,\frac {\lambda y} \pi\right)\left(1-\frac {\lambda y}
\pi\right)^{-1}$$
and that the exponential generating function for $t^{2i+1}$ in terms of the  $[\Delta_{2m+2p+1,p}]_\lambda$ is
$$ \frac{2^{2i+1}}\pi \lambda^{-i} h_i\left(\lambda \frac{ x+y}\pi,\frac {\lambda y} \pi\right)\left(1-\frac {\lambda y}
\pi\right)^{-1}.$$
By Lemma \ref{lem:def mu}, the corresponding exponential generating functions in terms of the $\mu^\lambda_{kq}$
are obtained by multiplying these last functions by $\left(1-\frac{\lambda y}{\pi}\right)$. 

Finally, we claim that the exponential generating functions with respect to the $\tau^\lambda_{k,q}$ are obtained by
substituting $x$ for {$x+y$}  in these last formulas. Expanding each $(x+y)^ky^l$ and writing down
the corresponding linear combination of the $\mu^\lambda_{kq}$, this conclusion follows at once from Definition
\ref{def:tilde mu}.
\end{proof}

\subsection{Dictionary between $\mu_{kq}^\lambda$ and $t,s$}
So far we have expressed the elements $t^i$ of the Lipschitz-Killing algebra in terms of the bases $\mu_{kq}^\lambda,
\tau_{kq}^\lambda$. The following proposition (which we prove in section \ref{sect:mult s} below) allows us to express
any monomial $t^i s^j$ in these terms.

\begin{proposition}\label{prop:s tau}
\begin{align}
\label{eq:s mukq lambda}
  s\cdot  \mu_{kq}^\lambda  &=
\frac{(k-2q+2)(k-2q+1)}{2\pi(k+2)}\mu^\lambda_{k+2,q}+\frac{2(q+1)(k-q+1)}{\pi(k+2)}\mu_{k+2,q+1}^\lambda\\
\label{staucn}
s\cdot\tau_{kq}^\lambda&=\frac{(k-2q+1)(k-2q+2)}{2\pi(k+2)}\tau_{k+2,q}^\lambda+\frac{(q+1)(2q+1)}{\pi(k+2)}\tau_{k+2,
q+1}^\lambda.
\end{align}
\end{proposition}

Conversely, every invariant valuation can be expressed as a polynomial in $s,t$:

\begin{proposition}\label{prop:generators}
The algebra $\V^n_\lambda$ is generated by $s,t$. The kernel of $\pi:\R[s,t]\rightarrow \V^n_\lambda$ is
contained in $W_n$, the ideal of polynomials with all terms of $\deg \ge n+1$ (where $\deg s=2$ and $\deg t=1$).

In particular $\V^\infty_\lambda$ is isomorphic to the algebra of formal power series $\R[[s,t]]$.
\end{proposition}

\proof
The case $\lambda=0$ was proved in \cite{ale03}. For $\lambda\neq 0$, let us consider the composition 
\[
 \psi:\mathbb R[s,t]\stackrel{\pi}\longrightarrow \V^n_\lambda\stackrel{F_\lambda^{-1}}\longrightarrow \V^n_0,
\]
where $F_\lambda$ is the vector space isomorphism given by 
\begin{equation}\label{eq:def F lambda}
F_\lambda(\mu_{kq}):=\mu_{kq}^\lambda.
\end{equation}
Lemma \ref{lem:mus that vanish} ensures that $F_\lambda$ is well-defined.

Theorem \ref{egfs
for t^k} and Proposition \ref{prop:s tau} show that 
\[
\psi(s^i t^j)\equiv s^it^j \mod \bigoplus_{k>2i+j}\val_{k}^{U(n)}.
\]
By reverse induction on $2i+j$, one shows that $\psi$, and hence $\pi$, is
surjective. 

Let now $p(s,t)\in \ker \psi $, and write $p(s,t)=\sum_{i\ge d} p_i(s,t)$ with $p_i(s,t)$ homogeneous of degree $i$. By
the previous property, 
\[
 0=\psi (p(s,t))\equiv p_d(s,t) 
\]
modulo $\bigoplus_{i>d}\val_{i}^{U(n)}$. Hence, $p_d(s,t)$ vanishes as an element in $\val_d^{U(n)}$. By the $\lambda=0$
case,  $d\geq n+1$.
\endproof

\begin{proposition} \label{prop_mu2st}
Denoting 
\begin{align}
\label{eq:def v} v &:= t^2 {(1-\lambda s)}\\
\label{eq:def u}  u &:=4s-v
\end{align}
we have
\begin{align}\label{eventau}
\tau_{kq}^\lambda &=(1-\lambda s)  \frac{\pi^{k}}{\omega_{k}(k-2q)!(2q)!}v^{\frac{k}{2}-q}u^q\\
&=\frac{\pi^{k}}{\omega_{k}(k-2q)!(2q)!}(1-\lambda s)^{\frac  k2-q+1}t^{k-2q}u^q. \label{eq_tau_lambda_tu}
\end{align}
Equivalently,
\begin{equation}\label{mu2st}
  \mu_{kq}^\lambda=(1-\lambda s)\sum_{i=q}^{\lfloor \frac{k}{2}\rfloor} (-1)^{i+q} \binom{i}{q}
\frac{\pi^k}{\omega_k(k-2i)!(2i)!}v^{\frac{k}{2}-i} u^i.
\end{equation}
\end{proposition}

\begin{lemma}\label{lem:operator O}
Define the linear operators
$
\mathcal O,\mathcal P: \R[[\xi,\eta ]] \to\R[[v,u]]
$
by
\begin{align*}
\mathcal O\left(\sum c_{mp}\xi^m\eta^p \right)&:=\sum c_{mp} \frac{\binom{m+p}{p}}{\binom{2m}{m}\binom{2p}{p}} v^m
u^p,\\
\mathcal P\left( \sum c_{mp}\xi^m\eta^p \right)&:= \sum c_{mp}\frac{(2m+2p+1)
\binom{2m+2p}{m+p}\binom{m+p}{p}}{(2m+1)\binom{2m}{m}\binom{2p}{p}} v^m u^p.
\end{align*}
Then
\begin{align}
\label{eq:gi} \binom{2i}i\mathcal O(g_i) &= v^i\left(1- \frac{v+u}4\right)^{-i-1}\\
\label{eq:hi}\mathcal P(h_i) & =  v^i \left(1-v-u \right)^{-i-\frac 3 2 }.
\end{align}
\end{lemma} 
\begin{proof} Since 
$$
(1- t)^{-\frac 1 2} = \sum \binom {2k}k \left(\frac t 4\right)^k
$$
the relation \eqref{eq:gi} for $i=0$ follows at once. 

In view of the equation $g_i=\frac{2^i}{(2i-1)!!} \xi^i \frac{\partial^i}{\partial \xi^i}g_0$ and the obvious relation
\begin{equation}\label{eq:operator O}
\mathcal O \circ \left(\xi^i \frac{\partial^i}{\partial \xi^i}\right) = \left(v^i \frac{\partial^i}{\partial v^i}\right)
\circ \mathcal O
\end{equation}
the general case now follows.

Using the expansion
$$
(1- t)^{-\frac 3 2} = \sum (2k+1) \binom {2k}k \left(\frac t 4\right)^k
$$
the relation \eqref{eq:hi} follows similarly.
\end{proof}

\begin{proof}[Proof of Proposition \ref{prop_mu2st}]

Let $ \rho:{\V}^\infty_\lambda\rightarrow {\V}^\infty_\lambda$ be the linear map sending $\tau_{k,q}^\lambda$ to the
right hand side of \eqref{eventau}. We will prove first that $ \rho$ is the identity on the subspace spanned by powers
of $t$. The proof will be finished by showing that $ \rho(s \cdot \mu)=s\cdot  \rho(\mu)$ for any valuation $\mu\in
\mathcal{V}_\lambda^\infty$. 

By Theorem \ref{egfs for t^k} and Lemma \ref{lem:operator O}
\begin{align*}
 \rho(t^{2i}) &=\binom{2i} i \lambda^{-i} \sum_{m,p} \left(\frac\lambda\pi\right)^{m+p}\left. \frac{\partial^{m+p}
g_i}{\partial^m\xi\partial^p\eta}\right|_{0,0}  \rho(\tau_{2m+2p,p})\\
&= (1-\lambda s)\binom{2i} i \lambda^{-i} \sum_{m,p} \left. \frac{\partial^{m+p}
g_i}{\partial^m\xi\partial^p\eta}\right|_{0,0} \frac{(m+p)!}{(2m)!(2p)!} (\lambda v)^m (\lambda u)^p\\
&= (1-\lambda s)\binom{2i} i \lambda^{-i} \sum_{m,p} \left. \frac{\partial^{m+p}
g_i}{\partial^m\xi\partial^p\eta}\right|_{0,0} \frac{1}{m!p!} \frac{\binom{m+p}{m}}{\binom{2m}{m}\binom{2p}{p}}(\lambda
v)^m (\lambda u)^p\\
&= (1-\lambda s)\binom{2i} i \lambda^{-i}\mathcal O(g_i)(\lambda v,\lambda u)\\
&= (1-\lambda s)v^{i}\left(1- \lambda\left( \frac{v+u}{4}\right)\right)^{-i-1}\\
&= t^{2i}
\end{align*}
in view of the defining relations \eqref{eq:def v}, \eqref{eq:def u}. One may check by a similar procedure that $
\rho(t^{2i+1}) = t^{2i+1}$.

Finally, let us check that $ \rho(s\cdot \mu)={s\cdot  \rho(\mu)}$ for any valuation $\mu$. 
Using Proposition \ref{prop:s tau} we compute
\begin{align*}
  \rho( s\cdot\tau_{kq}^\lambda)&=\frac{(k-2q+1)(k-2q+2)}{2\pi(k+2)}
\rho(\tau_{k+2,q}^\lambda)+\frac{(q+1)(2q+1)}{\pi(k+2)} \rho(\tau_{k+2,q+1}^\lambda)\\
& =\frac{(k-2q+1)(k-2q+2)}{2\pi(k+2)}\frac{\pi^{k+2}}{\omega_{k+2}(k-2q+2)!(2q)!}(1-\lambda s)^{\frac
k2-q+2}t^{k-2q+2}u^q
\\
&
 +\frac{(q+1)(2q+1)}{\pi(k+2)}\frac{\pi^{k+2}}{\omega_{k+2}(k-2q)!(2q+2)!}(1-\lambda s)^{\frac
k2-q+1}t^{k-2q}u^{q+1}\\
&= \frac{\pi^{k+1}}{2(k+2)\omega_{k+2}(k-2q)!(2q)!}(1-\lambda s)^{\frac k2-q+1}t^{k-2q}u^q\underbrace{((1-\lambda
s)t^2+u)}_{4s}
\\
 &=s \cdot \rho(\tau_{kq}).
\end{align*}
\end{proof}

\subsection{$\mathcal{V}^n_\lambda\simeq \mathcal{V}^n_0$ as filtered algebras }
\label{subsec_algiso}

\begin{theorem}\label{thm:1st iso}
 There exists an algebra  isomorphism $I_\lambda:\mathcal{V}^n_0 \rightarrow \mathcal{V}^n_\lambda$ such that
\[
 I_\lambda(s)=s,\qquad I_\lambda(t)=t\sqrt{1-\lambda s}.
\] \end{theorem}
\begin{proof} Let $\tilde I_\lambda:\mathcal{V}_0^\infty\rightarrow \mathcal{V}^\infty_\lambda$ be the algebra
isomorphism defined by $\tilde I_\lambda(s)=s, \tilde I_\lambda(t)=t\sqrt{1-\lambda s}$. We must show the
existence of the algebra morphism $I_\lambda$ in the following diagram, where the vertical maps are
restrictions:
\begin{displaymath}
\xymatrix{\mathcal{V}^\infty_0\ar[d]\ar[r]^{\tilde I_\lambda}&\mathcal{V}^\infty_\lambda\ar[d]\\
\mathcal{V}^n_0\ar@{.>}[r]^{I_\lambda}&\mathcal{V}^n_\lambda
}
\end{displaymath}
Recall from \cite{fu06, befu11} that $\mathcal{V}^n_0\cong \mathcal{V}_0^\infty/(\mu_{n+1,0},\mu_{n+2,0})$, and by
Proposition \ref{prop_mu2st}
 \begin{align}
\tilde I_\lambda(\mu_{k0})&=\sum_{i=0}^{\lfloor \frac k2\rfloor} (-1)^{i}
\frac{\pi^k}{\omega_k(k-2i)!(2i)!}(1-\lambda s)^{\frac{k}2 -i}t^{k-2i} (4s-t^2(1-\lambda s))^i \nonumber\\
&=\frac{1}{1-\lambda s} \mu_{k0}^\lambda \label{eq:def F}
\end{align}
whose image in $\mathcal{V}^n_\lambda$ vanishes if $n<k$, by Lemma \ref{lem:mus that vanish}. Hence $I_\lambda$ is
well-defined. Since it is clearly surjective, by comparing dimensions, it follows that $I_\lambda$ is bijective.
\end{proof}

According to Theorem \ref{thm_ftaig}, the following diagram commutes
\begin{displaymath}
\xymatrix{\mathcal{V}^n_\lambda\ar[r]^{\hspace{-0.3cm}k_\lambda}\ar[d]_{\pd_\lambda}&\mathcal{V}^n_\lambda\otimes
\mathcal{V}^n_\lambda\ar[d]^{\pd_\lambda\otimes \pd_\lambda}\\
\mathcal{V}_\lambda^{n*}\ar[d]_{ I_\lambda^*}\ar[r]^{\hspace{-0.3cm}m^*}&\mathcal{V}_\lambda^{n*}\otimes
\mathcal{V}_\lambda^{n*}\ar[d]^{ I_\lambda^* \otimes I_\lambda^*}\\
\mathcal{V}^{n*}_0\ar[r]^{\hspace{-0.3cm}m^*}&\mathcal{V}^{n*}_0\otimes \mathcal{V}^{n*}_0\\
\mathcal{V}^n_0\ar[r]^{\hspace{-0.3cm}k}\ar[u]^{\pd}&\mathcal{V}^n_0\otimes \mathcal{V}^n_0\ar[u]_{\pd\otimes \pd}}
\end{displaymath}

Hence the map $J_\lambda:=\pd_\lambda^{-1}\circ (I_\lambda^{-1})^* \circ \pd$ is a co-algebra isomorphism from
$\mathcal{V}^n_0$ to $\mathcal{V}^n_\lambda$, i.e.
\begin{equation}\label{eq:J coalg iso}
k_\lambda \circ J_\lambda =(J_\lambda \otimes J_\lambda) \circ k.
\end{equation}

{\begin{proposition}\label{prop_miracle}
$
 J_\lambda=(1-\lambda s)^2 I_\lambda.
$
\end{proposition}}

\proof 
{Note that $\mathcal{V}_\lambda^{n*}$ is a $\mathcal{V}_\lambda^{n}$-module in the usual way. We clearly have 
\begin{displaymath}
 \pd_\lambda(\phi_1 \cdot \phi_2)=\phi_1 \pd_\lambda(\phi_2), \quad \phi_1,\phi_2 \in \mathcal{V}_0^{n}
\end{displaymath}
and 
\begin{displaymath}
 (I_\lambda^{-1})^*(\phi \cdot \psi)=I_\lambda(\phi) (I_\lambda^{-1})^*(\psi), \quad \phi \in \mathcal{V}_0^{n}, \psi
\in \mathcal{V}_0^{n*}. 
\end{displaymath}

From these equations we obtain that  
\begin{equation}\label{mult}
 J_\lambda(\phi_1 \cdot \phi_2)=I_\lambda(\phi_1)\cdot J_\lambda(\phi_2),\qquad \phi_1,\phi_2 \in \mathcal{V}^n_0.
\end{equation}}

Since $I_\lambda(\chi) = \chi$, it is enough to show that 
\begin{equation}\label{eq:j lambda chi}
 J_\lambda(\chi)=(1-\lambda s)^2.
\end{equation}

Take $\varphi \in \mathcal{V}^n_\lambda$. On the one hand, 
\begin{displaymath}
 \langle \pd_\lambda ( \varphi),J_\lambda(\chi)\rangle=\langle \pd(\chi),I_\lambda^{-1}
\varphi\rangle=\langle\vol^*,I_\lambda^{-1} \varphi\rangle.
\end{displaymath}
On the other hand,
\begin{displaymath}
\langle \pd_\lambda ( \varphi),(1-\lambda s)^2\rangle=\langle\vol^*, \varphi (1-\lambda
s)^2\rangle.
\end{displaymath}
To prove \eqref{eq:j lambda chi}, it suffices to check that the right hand sides of these last two relations agree in
the case $\varphi:=t^{2i}s^j \in \mathcal{V}_\lambda^n$. 
For $\lambda>0$ we have, from Corollary \ref{cor:t^k(cpn)} and equation \eqref{eq:sk_total},
\begin{equation}\label{vol*monomial}
\langle \vol^*,t^{2i} s^j\rangle=\frac{\lambda^{n-i-j}}{\omega_{2n}}\binom{2i}{i}\binom{n-j+1}{i+1}.
\end{equation}

By analytic continuation, the same holds for all $\lambda {\in \R}$. Indeed, by Theorem \ref{egfs
for t^k} and Proposition \ref{prop:s tau}, the left hand side above depends analytically on $\lambda$.

{Let us first assume that $i>0$.} Using  \eqref{vol*monomial} we get
\begin{displaymath}
\langle\vol^*,I_\lambda^{-1} \varphi\rangle=\left\langle\vol^*,\frac{t^{2i} s^j}{(1-\lambda s)^{i}}\right\rangle=\sum _k
\binom{k+i -1}{k}\lambda^k\langle\vol^*,t^{2i}
s^{k+j}\rangle=\binom{2i}{i}\binom{n-j-1}{i-1}\frac{\lambda^{n-i-j}}{\omega_{2n}},
\end{displaymath}
and 
\begin{displaymath}
\langle\vol^*,\varphi (1-\lambda s)^2\rangle=\frac{\lambda^{n-i-j}}{\omega_{2n}
}\binom{2i}{i}\binom{n-j-1}{i-1}. 
\end{displaymath}
{The case $i=0$ can be treated in a similar way.}
\endproof

\subsection{The principal kinematic formula in complex space forms}
We may now deduce that the principal kinematic formulas in the $\CPn_\lambda$ are formally independent of $\lambda$ when
expressed in terms of the $\mu_{kq}^\lambda$ or $\tau_{kq}^\lambda$. Since the $\lambda = 0$ case was given in
\cite{befu11}, the parallel statement holds for all values of $\lambda$.

From Propositions \ref{prop_mu2st} and \ref{prop_miracle} it follows that
\begin{equation}\label{jmu}
J_\lambda(\mu_{kq})=(1-\lambda s)\mu_{kq}^\lambda.
\end{equation}
Thus the linear isomorphism of \eqref{eq:def F lambda} is given by
$F_\lambda = (1-\lambda s)^{-1}J_\lambda$. 
Clearly
$ F_\lambda(\tau_{kq} ) =
\tau^\lambda_{kq}.
$

\begin{theorem}\label{thm:pkf lambda} The principal kinematic formula in $\CP^n_\lambda$ is given by
$$
k_\lambda(\chi) = (F_\lambda\otimes F_\lambda)\circ k(\chi).
$$
\end{theorem}

\begin{proof}
\begin{align*}
 k_\lambda(\chi)&=((1-\lambda s)^{-1}\otimes(1-\lambda s)^{-1})k_\lambda((1-\lambda s)^2)\\
&=((1-\lambda s)^{-1}\otimes(1-\lambda s)^{-1})k_\lambda(J_\lambda(\chi))\\
  &=((1-\lambda s)^{-1}\otimes(1-\lambda s)^{-1})(J_\lambda\otimes J_\lambda)(k(\chi))
\end{align*}
which with \eqref{jmu}  gives the desired relation.
\end{proof}
One can show similarly that
\begin{equation}
 \label{kinform_mu}
k_\lambda(\mu_{kq}^\lambda)=(F_\lambda\otimes F_\lambda)\circ k((1-\lambda s)\mu_{kq}).
\end{equation}

\subsection{More isomorphisms} Although we will not make use of them here, there are other natural isomorphisms of
algebras $\mathcal{V}^n_\lambda \simeq \mathcal{V}^n_0$.

\begin{theorem}  \label{thm_other_isom}
Let $p \in \R[[t,u]]$. The algebra homomorphism $R:\R[[t,u]] \to \R[[t,u]]$ with 
\begin{align*}
t & \mapsto t p(t,u),\\
u & \mapsto u p(t,u)^2  
\end{align*}
induces an algebra homomorphism $\mathcal{V}^n_0 \to  \mathcal{V}^n_\lambda$. This map is an isomorphism if and only if
$p$ is a unit.
\end{theorem}

\proof
We compute  
\begin{align*}
 R(\mu_{k0}) & = \sum_{i=0}^{\lfloor \frac k2\rfloor} (-1)^i \frac{\pi^k}{\omega_k(k-2i)!(2i)!}R(t^{k-2i}u^i)\\
& = p(t,u)^k \sum_{i=0}^{\lfloor \frac k2\rfloor} (-1)^i \frac{\pi^k}{\omega_k(k-2i)!(2i)!}t^{k-2i} u^i.
\end{align*}

Using \eqref{eq_tau_lambda_tu} and Definition \ref{def:tilde mu} this gives us 
\begin{align*}
R(\mu_{k0}) & = p(t,u)^k \sum_{i=0}^{\lfloor \frac k2\rfloor} (-1)^i (1-\lambda s)^{-k/2+i-1}
\tau_{ki}^\lambda\\
& =\frac{p(t,u)^k}{(1-\lambda s)^{k/2+1}}  \sum_{i=0}^{\lfloor \frac k2\rfloor} (\lambda s-1)^i
\sum_{q=i}^{\lfloor \frac k2\rfloor} \binom{q}{i}\mu_{kq}^\lambda\\
& =\frac{p(t,u)^k}{(1-\lambda s)^{k/2+1}} \sum_{q=0}^{\lfloor \frac k2\rfloor} \sum_{i=0}^q \binom{q}{i}
(\lambda s-1)^i\mu_{kq}^\lambda\\
& =\frac{p(t,u)^k}{(1-\lambda s)^{k/2+1}} \sum_{q=0}^{\lfloor \frac k2\rfloor} (\lambda s)^q
\mu_{kq}^\lambda.
\end{align*}
We claim that if $k>n$, then $s^q \cdot\mu_{kq}^\lambda =0$ in $\mathcal{V}^n_\lambda$ for all $q$. Indeed, since
multiplication by $s$ in terms of the $\mu_{kq}^\lambda$ is independent of $\lambda$ (compare \eqref{eq:s mukq
lambda}), it is enough to prove this for $\lambda=0$. In this case 
\begin{align*}
 s^q \mu_{kq}(A) & = c \int_{{\overline{\Gr}^\C_{n-q}}} \mu_{k,q}(A \cap \bar E) d\bar E.
\end{align*}
Since $q<k-(n-q)$, the restriction of $\mu_{kq}$ to the $(n-q)$-dimensional complex affine subspace $\bar E$ vanishes
{by Lemma \ref{lem:mus that vanish}}. This proves the claim.

Thus $R$ maps the kernel of the projection $\V^\infty_0\to \V^n_0$ to the kernel of $\V^\infty_\lambda\to \V^n_\lambda$,
implying the first conclusion. It is clear that $R$ is invertible if and only if $p$ is a unit. 
\endproof

\begin{corollary} \label{cor_other_isoms}  
\begin{enumerate}
 \item[(i)] The algebra isomorphism $\R[[t,s]] \to \R[[t,s]]$ determined by
\begin{align*}
t & \mapsto \frac{t}{\sqrt{1+\frac{\lambda t^2}{4}}},\\
s & \mapsto s  
\end{align*}
induces an algebra isomorphism $R:\mathcal{V}^n_0 \to  \mathcal{V}^n_\lambda$. 
\item[(ii)] The algebra isomorphism $\R[[t,u]] \to \R[[t,u]]$ determined by 
\begin{align*}
t & \mapsto t,\\
u & \mapsto u  
\end{align*}
induces an algebra isomorphism $R:\mathcal{V}^n_0 \to  \mathcal{V}^n_\lambda$. 
\end{enumerate}
\end{corollary}

\proof
Choose $p(t,u):=\frac{1}{\sqrt{1+\frac{\lambda t^2}{4}}}$ for (i) and $p(t,u):=1$ for (ii). 
\endproof


\section{Tube formulas}\label{section: tube formulas} We apply Theorem \ref{thm:pkf lambda} to derive  tube formulas for
subsets of complex space forms. Our first formula, given in Theorem \ref{thm:general tube}, is {\it global} in the sense
that it gives the volume of the entire tubular neighborhood of a set or submanifold; as opposed to the {\it local} tube
formulas, which essentially compute the pullbacks under the exponential map of the volume form of the ambient manifold
to the normal bundles of submanifolds. Thus the global formulas involve valuations, while the local formulas involve
curvature measures. 

Strictly speaking, the global formulas are simply the reductions of the local formulas under the globalization map from
curvature measures to valuations. Moreover, the local tube formulas in complex space forms were computed in \cite{abbena
et al} by computing explicitly with the exponential map. However, it turns out that the global formula of Theorem
\ref{thm:general tube} simplifies in a striking way: it is a direct transcription of the tube formula in $\Cn$, i.e. in
an even dimensional euclidean space. This formulation also yields a simple global tube formula for totally real
submanifolds in Theorem \ref{thm:real tube}.
We also show in Section \ref{sect:local tube} below how to compute the local formulas in our framework, and check that
they correspond to the formulas of \cite{abbena et al}.

Recall the expressions for the volume of a metric ball in $\CP^n_\lambda$ (\cite{gray}, Lemma 6.18):
\begin{equation}\label{eq:vol of ball}
\vol_{2n}(B_r )=\frac{\pi^n}{n!} (\sn_\lambda r)^{2n}.
\end{equation}
If $\lambda >0$ then we assume that $r< \frac{\pi}{2\sqrt\lambda}$, the injectivity radius of the complex
projective space $\CP^n_\lambda$.

\begin{lemma} \label{lem:mukq of br}
$$\mu^\lambda_{kq}(B_r) =
c_{nkq} {2^{k-2q}\pi^n} 
\sn_\lambda^{k}(r)\cs_\lambda^{2n-k}(r).$$
\end{lemma}
\begin{proof} 
Let $\nu:\partial B_r\to S\CPn_\lambda$ denote the outward pointing unit normal field, and  select a local orthonormal
frame $e_{\bar 1}: = \sqrt{-1} \nu, e_2,e_{\bar2} := 
\sqrt{-1}e_2,\dots, e_n,e_{\bar n} := \sqrt{-1}e_n$ for $T\partial { B_r}$  with dual coframe $\eps_i:= e_i^*$.
By the calculations of Section 6.4 of \cite{gray} (the discrepancy in the signs being due to the opposite convention for
the orientation)
\begin{align*}
\nu^* \beta &= \eps_{\bar 1},\\
\nu^*\gamma &= 2 \ct_\lambda(2r) \ \eps_{\bar 1}=(\ct_\lambda-\lambda \tn_\lambda)(r)\ \eps_{\bar 1}, \\
\nu^* \theta_0 & = \ct_\lambda^2(r) \sum_{i\ge 2} \eps_i\wedge\eps_{\bar i} = \ct_\lambda^2 (r)
\kappa  \\
\nu^* \theta_1 & = 2\ct_\lambda(r)\ \kappa\\
\nu^* \theta_2 & = \kappa
\end{align*}
where $\kappa = \sum_{i\ge 2} \eps_i\wedge \eps_{\bar i}$ is the restriction of the K\"ahler form to
$T_x\partial{ B_r}$. Therefore for $k<2n$, the values of the basic curvature measures on the ball $B_r$ are 
\begin{align}
B_{kq}(B_r, \cdot)& = c_{nkq} 2^{k-2q-1} (n-1)!  \ct_\lambda^{2n-k-1}(r)\,
d\vol_{\partial B_r},\label{B_on_balls}\\
\Gamma_{kq}(B_r,\cdot)&=  {c_{nkq} }  2^{k-2q-1} (n-1)! (\ct_\lambda
-\lambda \tn_\lambda)(r)\ct_\lambda^{2n-k-2} (r)\, d\vol_{\partial B_r}.\label{Gamma_on_balls}
\end{align}
Since
$$
\vol_{2n-1}(\partial B_r) = \frac{2\pi^n}{(n-1)!} \sn_\lambda^{2n-1}(r) \cs_\lambda
(r)
$$
we obtain
\begin{align}
 [B_{kq}]_\lambda(B_r)&= c_{nkq} \pi^n 2^{k-2q} (\cs_\lambda^{2n-k} \sn_\lambda^k)(r), \quad
k>2q,\\
\label{eq:gamma2qq of br}[\Gamma_{kq}]_\lambda(B_r)&=c_{nkq} {\pi^n 2^{k-2q}}(\cs_\lambda^{2n-k} \sn_\lambda^k
- \lambda \cs_\lambda^{2n-k-2}\sn_\lambda^{k+2})(r) 
\end{align}
for $k<2n$. This gives the stated formula in the case $k>2q$, and the remaining cases $k=2q$ follow from a direct
calculation after substituting the values $c_{nkq}$ in the relations \eqref{eq:gamma2qq of br}.
\end{proof}

We define
$$
\mu_k^\lambda:=\tau_{k0}^\lambda= { \sum_{q=\max\{0,k-n\}}^{\lfloor \frac k2\rfloor}} \mu_{kq}^\lambda .
$$
In particular, $\mu_k^0 = \mu_k$, the $k$th intrinsic volume.

\begin{corollary} 
$$
\mu_k^\lambda(B_r) =\binom {2n} k  \frac{\omega_{2n}}{\omega_{2n-k}} 
\sn_\lambda^k(r) \cs_\lambda^{2n-k}(r).$$
\end{corollary}

{\proof
Lemma \ref{lem:mukq of br} implies that  $\mu_k^\lambda(B_r)= c \sn_\lambda^k(r)\cs_\lambda^{n-k}(r)$, for some constant
$c$ independent of $\lambda$. Taking $\lambda=0$ we may evaluate $c$ using the relation
$\mu_k(B_r)=\binom{2n}{k}\frac{\omega_{2n}}{\omega_{2n-k}}r^k$.
\endproof}

\begin{theorem}  \label{thm:general tube}
For any sufficiently regular compact $A \subset \CP_\lambda^n$
$$
\int_{\CP_\lambda^n} \chi(A\cap B(p,r))\, dp= 
\sum_{k=0}^{2n} \omega_{2n-k} \mu^\lambda_k(A)\sn_\lambda^{2n-k}(r)
\cs_\lambda^{k}(r).
$$
In particular this formula gives the volume of the tube $A_r$ if $r \le \operatorname {reach}(A)$.
\end{theorem}

As stated in the second Remark following Definition \ref {def:curv & val}, the precise regularity condition {needed
here} is unclear, although it is enough that $A$ be a smooth submanifold, or a set with positive reach, or a subanalytic
set.

\begin{proof}
The tube formula in $\Cn =\R^{2n}$ is
\begin{equation}
{ \sum_{k=0}^{2n} \omega_{2n-k}r^{2n-k}  \mu_k  = k_0(\chi) (B_r, \cdot).}
\end{equation}
In other words, if
$$
k_0(\chi) = \sum_{kpq} { a_{k,q,p}} \, \mu_{kq} \otimes \mu_{2n-k,p}
$$
then for each fixed $k$
$$
 \sum_{p,q} a_{kqp} \,  \mu_{2n-k,p}(B_r) \mu_{kq} = \omega_{ 2n-k} r^{2n-k}\mu_k.
$$
Since the $\mu_{kq}$ are linearly independent, it follows that for fixed $k,q$
$$
\sum_{p} a_{kqp} \,  \mu_{2n-k,p}(B_r) = \omega_{ 2n-k} r^{2n-k}
$$
and we deduce from Lemma \ref{lem:mukq of br} that
$$
\sum_{p} a_{kqp} \,  \mu^\lambda_{2n-k,p}(B_r) = \omega_{ 2n-k}
\sn_\lambda^{2n-k}(r)\cs_\lambda^k(r).
$$

But by Theorem \ref{thm:pkf lambda}, the tube formula in $M_\lambda$ is given by
\begin{align*}
k_\lambda(\chi) (\cdot, B_r) &= (F_\lambda\otimes F_\lambda)\circ k_0(\chi)(\cdot, B_r)\\
&= \sum_{k,p,q} a_{kqp} \,  \mu^\lambda_{2n-k,p}(B_r) \mu^\lambda_{kq}\\
&= \sum_{k} \omega_{ 2n-k} \sn_\lambda^{2n-k}(r)\cs_\lambda^k(r)
\,\mu^\lambda_{k}
\end{align*}
as claimed.
\end{proof}

\subsection{Tube formulas for totally real submanifolds of $\CP^n_\lambda$} Following \cite{gray}, we say that
a submanifold $N \subset \CP^n_\lambda$ is {\bf totally real} if $T_xN \perp \sqrt{-1}T_xN $ for each $x \in
N$. In other words, $N$ is totally real if and only if it is isotropic with respect to the K\"ahler form $\kappa$ of
$\CP^n_\lambda$.
Theorem \ref{thm:general tube} specializes nicely for such submanifolds. 

\begin{lemma}\label{lem:u vanish} Let $N\subset \CP^n_\lambda$ be a smooth totally real submanifold. Then
$u|_N = 0$. Equivalently,
\begin{equation}\label{eq:st equiv}
\left.s\right|_N = \left.\frac{t^2}{4+ {\lambda t^2}} \right|_N
\end{equation}
\end{lemma}

\begin{proof} It is clear that for any differentiable polyhedron $P \subset N$, the invariant form
$\theta_2{=\pi^*\kappa}$ vanishes when restricted to the normal bundle $N(P)$. Therefore the curvature measures
$B_{kq}(P,\cdot)= \Gamma_{kq}(P,\cdot) = 0$ if $q >0$, so $ \mu^\lambda_{kq}(P) = 0$ if $q >0$. It follows that  $
\tau^\lambda_{kq}(P) = 0$ if $q >0$. In other words, $\restrict{\tau^\lambda_{kq}}N=0$ for $q>0$.

Thus by Proposition \ref{prop_mu2st}, 
$$
\left. \left(s(4+ {\lambda t^2}) - t^2\right) \right|_N =\restrict {u}N = {\frac{2}{\pi}}
\restrict{[(1-\lambda s)^{-1}\tau_{21}^\lambda]}{N}= 0.
$$
Since the restriction map on valuations is a homomorphism of algebras, the lemma follows.
\end{proof}

{\bf Remark.} In fact $u$ generates the ideal of all elements of $\mathcal{V}^n_\lambda$ that vanish on all
totally real submanifolds of $\CP^n_\lambda$.

\begin{theorem}\label{thm:real tube}
Let  $N \subset \CPn_\lambda $ be a smooth totally real submanifold. Then
\begin{equation}\label{eq:isotropic intrinsic}
\left.\mu_k^\lambda\right|_{N} = \frac{\pi^k}{k!\omega_k}\left. \frac{t^k}{\left(1+ \frac{\lambda t^2} 4\right)^{\frac k
2 +1}}\right|_N.
\end{equation}
\end{theorem}
{Thus
$$
\restrict{\sum_k \omega_k \mu_k^\lambda }N=\restrict{\frac{1}{\sqrt{1 + \frac{\lambda t^2}{4}}}\exp\left(\frac{\pi t }{\sqrt{1 + \frac {\lambda t^2}{4}}}\right) }N
$$
}
\begin{proof} By  Lemma \ref{lem:u vanish} and Proposition \ref{prop_mu2st}, modulo the kernel of the restriction
$\mathcal{V}^n_\lambda \to \mathcal{V}(N)$,
\begin{align*}
\mu_k^\lambda &= \tau_{k0}^\lambda\\
&=\frac{\pi^{k}}{k!\omega_{k}}(1-\lambda s)^{\frac  k2+1}t^{k}\\
&\equiv\frac{\pi^{k}}{k!\omega_{k}}\frac{t^k}{\left(1+ \frac{\lambda t^2} 4\right)^{\frac k 2 +1}}
\end{align*}
after substituting $\frac{t^2}{4+ \lambda t^2}$ for $s$.
\end{proof}

Thus the global tube formula for totally real subsets is {\em intrinsic}.

\subsection{Kinematic and tube formulas for complex subvarieties} 
When restricted to complex analytic submanifolds of $\CP^n_\lambda$, the array of invariant curvature measures
and valuations simplifies greatly. 

 By a complex analytic subvariety of $\CP^n_\lambda$ we mean a subspace $X \subset \CP^n_\lambda$ with the
following property: there exists an open cover $\{U_\alpha\}$ of $X$ such that for each $\alpha$ there are
complex analytic functions $f_1,\dots,f_N$ on $U_\alpha$ with $X\cap U_\alpha = \bigcap_{i=1}^N f_i^{-1}(0)$.
Thought of as a subset of $\bigcup_\alpha U_\alpha$, where the $U_\alpha$ cover $X$ as above, a complex
analytic subvariety $X\subset \CP^n_\lambda$ admits a normal cycle.

Put $\cmeas^n_\lambda$ for the vector space spanned by the restrictions of elements of $\curvun$ to Borel
subsets of complex analytic subvarieties of $\CP^n_\lambda$, and
 $$R_\lambda:\curvun\to \cmeas^n_\lambda$$ 
 for the restriction map. 
For $\lambda_0>0$ we define also $e_{\lambda_0}: \curvun\to \R^{n+1}$ by
$$
e_{\lambda_0}(\Phi):= ([\Phi]_{\lambda_0}(\CP^0_{\lambda_0}),
[\Phi]_{\lambda_0}(\CP^1_{\lambda_0}),\dots,[\Phi]_{\lambda_0}(\CP^n_{\lambda_0})).
$$

Note that
\begin{equation} \label{eq_gamma_2kk_cpj}
 [\Gamma_{2k,k}]_{\lambda_0}(\CP^j_{\lambda_0})=\delta_j^k\vol_{2j}(\CP^j_{\lambda_0})=\delta_j^k 
\frac{\pi^j}{\lambda_0^j j!}.
\end{equation}

\begin{proposition}\label{prop:C curv meas}  
Let $\lambda\in \R$ and $\lambda_0 >0$. Then 
\begin{equation}
\ker R_\lambda = \ker (\glob_{ \lambda_0} \circ R_{ \lambda_0}) = \ker e_{\lambda_0} = \spann \{B_{kq}\}_{k,q} 
\oplus \spann\{ \Gamma_{kq}: k>2q\}.
\end{equation}
\end{proposition}
Here we abuse notation by putting $\glob_{\lambda_0}$ for the obvious map from $\cmeas^n_{\lambda_0}$ to the space
of finitely additive set functions defined on algebraic subvarieties of $\CPn_{\lambda_0}$.

\begin{proof}  The normal cycle $N(X)$ of a complex analytic subvariety $X\subset \CP^n_\lambda$ is an integer
linear combination of the
normal cycles of its strata \cite{fu94}. Thus the restriction of the distinguished 1-form $\beta$, referred to at the
beginning of Section \ref{subsect:invariant} above and defined in  \cite{ags,befu11}, vanishes
identically on $N(X)$. It follows that $\theta_1 = d\beta$ also vanishes. Since the defining differential
forms of the $B_{kq}$ are multiples of $\beta$, and those of the $\Gamma_{kq}$ are multiples of $\theta_1$ if
$k>2q$, it follows that  the indicated kernels include the span of these elements.

On the other hand, if $X\subset \CP^n_\lambda$ is a $q$-dimensional subvariety then $\Gamma_{2q,q}(X,\cdot)$ is the
volume measure of $X$, and $\Gamma(X,\cdot) =0$ for $p>q$. It follows that the $R_\lambda(\Gamma_{2q,q})$ (and their
globalizations)
are linearly independent. Using \eqref{eq_gamma_2kk_cpj} it follows
that all kernels are as indicated.
\end{proof}

Thus we arrive at the following simple description of the integral geometry of complex analytic submanifolds of the
complex space forms:
\begin{itemize}
\item The transfer principle holds in the following sense. Let $S_\lambda:\cmeas^n_\lambda \to
\curvun$ be the unique right inverse of $R_\lambda$ with $S_\lambda\circ
R_\lambda(\Gamma_{2q,q})=\Gamma_{2q,q}$. The
local integral geometry of subvarieties of $\CP^n_\lambda$ is then encapsulated by the restricted kinematic
operator
$$
K_{\C,\lambda}  : \cmeas^n_\lambda \to \cmeas^n_\lambda\otimes \cmeas^n_\lambda
$$
given by
$$
K_{\C,\lambda} = (R_\lambda\otimes R_\lambda) \circ K \circ S_\lambda.
$$
In other words
if we express the local kinematic operator $K$ in terms of the basis $\{B_{kq},\Gamma_{kq}\}$, then
$K_{\C,\lambda}$ is obtained by restricting to $\spann\{ \Gamma_{2q,q}\}_q$ and projecting to $\spann\{
\Gamma_{2q,q}\otimes  \Gamma_{2r,r}\}_{q,r}$.
\item  If $\lambda >0$ then the restriction to $\spann\{\Gamma_{2q,q}\}_q$ of the globalization map yields a bijection
to the space of  valuations restricted to complex subvarieties. Hence the local kinematic formulas can be read off from
the global kinematic formulas.
\item There are convenient template objects in the compact spaces of the family that permit the evaluation of the global
kinematic operator by elementary means.
\end{itemize}
 
 \begin{theorem}\label{thm:kc gamma}
 $$
 K_{\C,\lambda}(R_\lambda (\Gamma_{2q,q}) )= \frac{1}{n! q!} \sum_{i+j = n+q}
 i!j! R_\lambda(\Gamma_{2i,i}) \otimes R_\lambda(\Gamma_{2j,j}).
 $$
 \end{theorem}

\begin{proof}
By the remarks above, we already know that a formula of this type exists. To compute the
coefficients, we use as templates $\CP^i_\lambda, \CP^j_\lambda$, taking into account \eqref{eq_gamma_2kk_cpj}.
\end{proof}

The span of the curvature measures $\Gamma_{2q,q}$ admits a different basis that is more closely adapted to the language
of topology and algebraic geometry, namely the {\it Chern curvature measures} defined below. In fact it is natural to
define these curvature measures on a general K\"ahler manifold.
 Note that the sphere bundle $SM$ of any K\"ahler manifold $M$ admits canonical 1-forms $\alpha,
\beta,\gamma$, where $\alpha$ is the usual contact form, $\beta_{\xi}(v) := \langle \xi, \sqrt{-1}
\pi_{M*}v\rangle$ and $\gamma$ is the vertical (with respect to the Levi-Civita connection) 1-form whose
restriction to each fiber $S_xM$ is the unit covector field parallel to the Hopf fibration. If $M=
\CP^n_\lambda$ then these forms agree with the like-named forms defined in \cite{befu11} and \cite{ags}. In
the general case, the form $\gamma$ described here corresponds to a multiple of the form called $\beta$ in
\cite{fu94}.

\begin{theorem}[\cite{fu94}]\label{thm:chern classes} Let $M$ be a K\"ahler manifold of complex dimension $n$. Then
there exist canonically defined differential forms $\tilde\gamma_k \in \Omega^{2n-2k-1}(SM)$ such that for any
subvariety $X\subset M$, possibly with singularities, each current $\pi_{M*}(N(X) \lefthalfcup \tilde \gamma_k )$ is
closed and represents the Chern-MacPherson homology class $c_k(X)$.
The $\tilde \gamma_k$ are polynomials in $\gamma, d\gamma$ and the Chern forms of $M$. \end{theorem}
\begin{proof} In \cite{fu94}, a similar construction was given by contracting the projectivized conormal cycle $\mathbb
PN^*(X)$  against certain differential forms $\gamma_k$ within the projectivized cotangent bundle $\mathbb P^*M$. This
may be restated in terms of the full conormal cycle $N^*(X)$, within the cosphere bundle $S^*M$, by pulling back the
$\gamma_k$ via the Hopf fibration $S^*M \to \mathbb P^* M$ and multiplying by an appropriate multiple of the 1-form
$\beta$ defined in Section 1.4, op. cit. The proof of the first assertion is concluded by pulling back to $SM$ via the
identification $SM \simeq S^*M$ induced by the K\"ahler metric.

The second assertion follows directly from the construction of \cite{fu94}.
\end{proof}

\begin{definition} Let $M, \tilde \gamma_k$ be as above, and let $\kappa$ denote the K\"ahler form of $M$. We define the
{\bf  Chern curvature measure} $C_k$ to be the curvature measure on $M$ corresponding to the differential form $\tilde
\gamma_k \wedge \kappa^{k} \in \Omega^{2n-1}(SM)$, and the {\bf Chern valuation} $c_k$ to be the corresponding valuation.
\end{definition}

\begin{proposition} If $f:M\to M$ is a Hermitian isometry of $M$ and $\tilde f:SM\to SM$ the induced map, then $\tilde
f^*(\tilde \gamma_k \wedge \kappa^{k}) = \tilde \gamma_k \wedge \kappa^{k}$. In particular, $C_k, c_k$ are invariant
under $f$.
\end{proposition}
\begin{proof} This follows at once from the canonical nature of the construction of \cite{fu94}.
\end{proof}

In particular, if $M=\CP^n_\lambda $ then the Chern curvature measures and Chern valuations are invariant.
As elements of $\curvun$ the former vary analytically with $\lambda$, so we denote them by
$C^\lambda_k, c^\lambda_k$ respectively.

\begin{proposition}\label{prop:chern cm=gamma}
\begin{align}
C^\lambda_k &=\sum_{q \ge k} q!\binom{q+1}{k+1} \left(\frac \lambda \pi\right)^{q-k} \Gamma_{2q,q}
\label{eq:chern cm=gamma} \\
\Gamma_{2q,q} & = \frac{1}{q!}\sum_{k \ge q}  \left(-\frac{\lambda}{\pi}\right)^{k-q}
\binom{k+1}{q+1} C_k^\lambda \label{eq_gamma_Chern}\\
\mu^\lambda_{2q,q} & = \frac{1}{q!}\sum_{k \ge q}  \binom{k}{q}\left(-\frac{\lambda}{\pi}\right)^{k-q}
c_{k}^\lambda \label{eq_mu_chern}\\
c_k ^\lambda& = \sum_{q \ge k} q! \left(\frac\lambda \pi\right)^{q-k} \binom{q}{k}\mu_{2q,q}^\lambda.
\label{eq_chern_mu}
\end{align}
\end{proposition}
\begin{proof} From the last assertion of Theorem \ref{thm:chern classes} and Lemma 2.4 of \cite{ags}, we
deduce that the differential forms determining the $C_k^\lambda$ are polynomials in $\gamma, d\gamma =
2\theta_0 -2(\alpha\wedge \beta +\theta_2)$ and the Chern forms of $\CP^n_\lambda$, which are themselves
multiples of powers of the K\"ahler form $\theta_2$ of $\CP^n_\lambda$. Thus
$C^\lambda_k$
is a linear combination of the $\Gamma_{2q,q}$. It remains to show that the coefficients are as given.
By analytic continuation, it is enough to do this for $\lambda >0$.

By the proof of Corollary \ref{cor:t^k(cpn)}, if $j\ge k$ then
\begin{equation}
 c_k^\lambda (\CP_\lambda^j)=\int_{\CP_\lambda^j} \ch_{j-k}(\CP_\lambda^j)\wedge\kappa^k=\left(\frac
\pi\lambda\right)^{k}\binom{j+1}{k+1} \label{eq_value_chern}
\end{equation}
and $c_k^\lambda(\CP^j) = 0$ if $j<k$. Using \eqref{eq_gamma_2kk_cpj}  this is sufficient to determine that
the coefficients in \eqref{eq:chern
cm=gamma} are as stated.

Equation \eqref{eq_chern_mu} follows by globalizing \eqref{eq:chern cm=gamma}, taking into account Lemma
\ref{lem:def mu}. Next, \eqref{eq_mu_chern} is an easy consequence of \eqref{eq_chern_mu}, and finally
\eqref{eq_gamma_Chern} follows from Lemma \ref{lem:def mu} and \eqref{eq_mu_chern}.  
\end{proof}

Using Theorem \ref{thm:kc gamma} we may now deduce the  complex kinematic formula in terms of the Chern
curvature measures. Put $\Ch^\lambda_q:= R_\lambda(C^\lambda_q)$.
\begin{corollary}\label{cor:shifrin k}
 \begin{equation}\label{eq:local c kinematic}
 K_{\C,\lambda}(\Ch^\lambda_q) = \frac{1}{n!} \sum_{k+l \ge n+q}
\left(-\frac \lambda \pi\right)^{k+l-n-q}\binom{k+l-q}{n} \Ch^\lambda_k \otimes  \Ch^\lambda_l.
 \end{equation}
In particular, if $\lambda >0$ and $X,Y \subset \CP^n_\lambda$ are (possibly singular) subvarieties in general position
then
\begin{equation}\label{eq:global c kinematic}
 c_q(X\cap Y)  = (-1)^n\sum_{k+l \ge n+q}
\left(-\frac \lambda \pi\right)^{k+l-q}\binom{k+l-q}{n}c_k(X) c_l(Y).
\end{equation}
\end{corollary}
The projective ($\lambda >0$) cases of these formulas were given by Shifrin in \cite{shifrin81, shifrin84}.

\begin{proof} 
By Theorem \ref{thm:kc gamma} and Proposition \ref{prop:chern cm=gamma},
\begin{align*}
K_{\C,\lambda}(\Ch^\lambda_q) =& \frac 1 {n!} \sum_{r\ge q} \binom{r+1}{q+1} \left(\frac\lambda
\pi\right)^{r-q} \times\\
&\times \sum_{i+j=n+r}\left(\sum_{k\ge i}\left(-\frac\lambda \pi\right)^{k-i} \binom{k+1}{i+1}\Ch^\lambda_k
\right)\otimes
\left(\sum_{l\ge j}\left(-\frac\lambda \pi\right)^{l-j} \binom{l+1}{j+1}\Ch^\lambda_l \right)\\
=& \frac{(-1)^{n+q}}{n!}\sum_{k+l\ge n+q}  \left(-\frac\lambda \pi\right)^{k+l-n-q}\Ch^\lambda_k\otimes\Ch^\lambda_l
\sum_{i+j\ge n+q} (-1)^{i+j}\binom{i+j-n+1}{q+1}\binom{k+1}{i+1}\binom{l+1}{j+1}
\end{align*}
so the relation \eqref{eq:local c kinematic} follows from the numerical identity
\begin{equation}\label{eq:numerical identity}
\sum_{i+j\ge n+q} (-1)^{i+j}\binom{i+j-n+1}{q+1}\binom{k+1}{i+1}\binom{l+1}{j+1}= (-1)^{n+q} \binom{k+l-q}{n}.
\end{equation}
To prove this we write the left hand side as
\begin{align*}
\sum_{r\ge n+q}(-1)^r\binom{r-n+1}{q+1}\sum_{i+j=r }\binom{k+1}{i+1}\binom{l+1}{j+1}
&= \sum_{r\ge n+q}(-1)^r\binom{r-n+1}{q+1} \binom{k+l+2}{r+2}.
\end{align*}
In view of standard identities (cf. e.g. \cite{gfology}, Section 2.5), the last expression is $(-1)^{n+q}$ times the
coefficient of $x^{-(n+q+2)}$ in the Laurent series 
$$
(1+ x)^{-(q+2)}\left(1+\frac 1 x\right)^{k+l+2} = x^{-(k+l+2)}(1+x)^{k+l-q}
$$
which yields \eqref{eq:numerical identity}.

To prove the intersection formula \eqref{eq:global c kinematic} for projective varieties, we recall  the well-known fact
that the set  of positions $g \in G_\lambda$ such that $X, gY$ fail to meet transversely has real codimension $\ge 2$ in
$G_\lambda$. Thus the complementary  set $S \subset G_\lambda$ is connected, and  the map $g\mapsto N(X\cap gY)$ is
continuous on $S$. Since $f(g):=c_q(X\cap gY) $ is given as the integral of a fixed differential form on the sphere
bundle of $\CP^n_\lambda$, it follows that $f $ is continuous at each $g\notin S$. On the other hand, $f(g)$ is a linear
combination of characteristic numbers of the variety $X\cap gY$, with coefficients given as powers of $\frac \pi
\lambda$. The set of  values of $f$ is discrete, so $f$ must be constant, with value given by the kinematic integral
divided by the volume of $G_\lambda$.
\end{proof}

{ Using Theorem \ref{thm:general tube}} we may also deduce the local tube formulas for smooth complex analytic
submanifolds of $\CP^n_\lambda$, which are equivalent to the local tube formulas of Gray (\cite{gray}, Thm.~7.20;
\cite{gray85}, Thm.~1.1).

\begin{theorem}[Gray] \label{thm_local_complex_tube}
Let $X \subset \CP^n_\lambda$ be a smooth complex analytic submanifold, and $U \subset X$ a relatively compact
open subset. Then for all sufficiently small $r>0$
\begin{align*}
\vol(\exp((N(X) \cap \piinv U)\times [0,r])) & = \\
\sum_k \omega_{2n-2k}\sn_\lambda^{2n-2k}(r)\cs_\lambda^{2k}(r)
& \frac{1}{k!} \sum_{j\ge k} \left(-\frac{\lambda}{\pi}\right)^{j-k} \binom{j}{k} C_j^\lambda(X ,U).
\end{align*}
\end{theorem}

\begin{proof} For $0<r < \frac\pi{\sqrt\lambda}$, consider the invariant differential form 
$\psi^\lambda_r \in \Omega^{2n-1}(S\CP^n_\lambda)$ given by
$$
\psi_r^\lambda:= \int_0^r i_{\frac{\partial}{\partial t}}(\exp^*(d\vol_M \, dt))
$$
where $\exp: S\CP^n_\lambda \times \R\to \CP^n_\lambda$ is the exponential map and $\frac{\partial
}{\partial t}$ is the
tangent vector field to the $\R$ factor. The form $\psi_r^\lambda$ is clearly invariant under the action of
the isometry
group, so the associated curvature measure $\Psi^\lambda_r: =\Int(\psi_r^\lambda,0)$ belongs to $\curvun$. From the
definition we have 
$$
\vol(\exp((N(A) \cap \piinv U)\times [0,r])) = \Psi^\lambda_r(A,U)
$$
for any smooth submanifold $A \subset \CP^n_\lambda$, any relatively compact open set $U \subset X$ and
$r>0$ sufficiently small.

By Propositions \ref{prop:C curv meas} and \ref{prop_kernels_glob} we have $\ker R_\lambda\supset\ker\glob_\lambda$.
Hence there exists a map $r_\lambda:\V_\lambda^n\to \cmeas^n_\lambda$ such that $R_\lambda=r_\lambda\circ
\glob_\lambda$. By \eqref{eq_mu_chern}, this map is given by $r_\lambda(\mu_{kq}^\lambda)=0$ for $k\neq 2q$ and 
\begin{displaymath}
r_\lambda(\mu_{2q,q}^\lambda)=\frac{1}{q!}\sum_{k\geq q}\binom{k}{q}\left(-\frac \lambda \pi\right)^{k-q}
\Ch_k^\lambda.
\end{displaymath}
Using Theorem \ref{thm:general tube} we have
\begin{align*}
 R_\lambda(\Psi_r^\lambda)&=r_\lambda(\glob_\lambda\Psi_r^\lambda)\\
&=\sum_{k=0}^{2n} \omega_{2n-k} \sn_\lambda^{2n-k}(r)
\cs_\lambda^{k}(r) r_\lambda(\mu^\lambda_k)\\
&=\sum_{q=0}^n \omega_{2n-2q}  \sn_\lambda^{2n-2q}(r)
\cs_\lambda^{2q}(r) r_\lambda(\mu^\lambda_{2q,q}),
\end{align*}
and the result follows.
\end{proof}

\begin{corollary}[Gray, \cite{gray85}]
 The volume of the $r$-tube around $\CP^m_\lambda \subset \CP^n_\lambda, \lambda>0$, is given for {$r <
\frac{\pi}{2\sqrt \lambda}$} by
\begin{displaymath}
 \vol((\CP^m_\lambda)_r)=\frac{\pi^n}{n!}  \sum_{k=0}^m{\frac{1}{\lambda^k}} \binom{n}{k}
\sn_\lambda^{2n-2k}(r)\cs_\lambda^{2k}(r).
\end{displaymath}
\end{corollary}

\proof
Globalizing Theorem \ref{thm_local_complex_tube} and replacing \eqref{eq_value_chern} leaves us with a double
sum which simplifies to the given formula.
\endproof

\subsection{General local tube formulas}\label{sect:local tube}
Let $M^n$ be a Riemannian manifold with injectivity radius $r_0 >0$. For $0<r<r_0$, the volume of the tube
of radius $r$ defines a smooth valuation $\tau_r \in \mathcal{V}(M)$, namely
$$
\tau_r := \int_M \chi(\cdot \cap B(p,r)) \, d\vol_M(p).
$$
On the other hand, the exponential map determines a curvature measure $T_r$ with $\tau_r= [T_r]$, as follows. Let
$\exp:SM\times \R \to M$ be the exponential map of $M$. We then define 
$T_r:= \Int( \psi_r,d\vol_M) \in \curvinfty(M)$,  where
$$
\psi_r:= \int_0^r i_{\frac\partial{\partial t}}(\exp^*(d\vol_M))\, dt \in \Omega^{n-1}(SM).
$$

Next we show how to recover the curvature measures $T_r$ from the valuations $\tau_r$ by
valuation-theoretic operations. In particular this permits us to derive the local form of the tube formula in
complex space forms from the global formula given in Theorem \ref{thm:general tube}. Recall from \cite{befu11}
the first variation operator 
$
\delta: \mathcal{V}(M) \to \curvinfty(M).
$

\begin{theorem}\label{thm:local tube} In a general Riemannian manifold, 
$$
T_r = \vol + \int_0^r \delta \tau_s \, ds.
$$
\end{theorem}

This theorem is an immediate consequence of the following proposition.
 
\begin{proposition} Let $A$ be a compact domain with smooth boundary and $\reach A>r$. Then, for
every smooth function $f\in C^\infty(M)$,
 \[
 \frac{d}{dr} T_r(A,f)=\delta \tau_r(A,f).
 \]
\end{proposition}
\proof
Let  $\nu$ be the unit outer normal field along $\partial A$, and let $X$ be a smooth vector field { on $M$} with
$X(x)=f(x) \nu(x)$ for every $x\in\partial A$. Consider the flow $F_t:M\rightarrow M$ defined by $X$, and the
induced flow of contact transformations $\tilde F_t:SM\rightarrow SM$. Then
\[
\delta \tau_r(A,f)=\left.\frac{d}{dt}\right|_{t=0}\tau_r(F_tA)= \left.\frac{d}{dt}\right|_{t=0}
\vol((F_tA)_r).
\]
Let  $\exp:TM\rightarrow M$ be the exponential map, and $\exp_r:SM\rightarrow M$ be given by
$\exp_r(x, v)=\exp(x,r v)$. Then $\partial (F_tA)_r=\exp_r\circ\tilde F_t(N(A))$. Denoting
$G_t=\exp_r\circ\tilde F_t$, we look for the derivative at $t=0$ of the volume enclosed by $G_t(N(A))$. Hence,
the first variation formula for volume (cf. e.g. \cite{spivak}, p. 193, eq (II)) yields
\[
 \left.\frac{d}{dt}\right|_{t=0} \vol((F_tA)_r)=\int_{\partial A_r} \left\langle
\left.\frac{d}{dt}\right|_{t=0}G_t(x, \nu(x)),  \nu_r(y)\right\rangle d_y\vol_{n-1}
\]
where $x=x(y)$ is the point in $A$ closest to $y$, $ \nu_r$ is the outer unit normal vector to $\partial A_r$, and
$d\vol_{n-1}$ denotes
the surface area of $\partial A_r$. 

By the lemma below (taking $(p(t),v(t))=\tilde F_t(x, \nu(x))$), the integrand in the last integral equals
\[
\left \langle \left.\frac{d}{dt}\right|_{t=0}G_t(x, \nu(x)),\left.\frac{d}{du}\right|_{u=r}\exp(x,u
 \nu_r(x))\right\rangle= \left\langle  \left.\frac{d}{dt}\right|_{t=0} F_t(x), \nu(x)\right\rangle= f(x),
\]
and the result follows.
\endproof

\begin{lemma}
Let $(p(t),v(t))$ be a curve in $SM$. Then
\[
 \left\langle \left.\frac{d}{dt}\right|_{t=0} \exp(p(t),rv(t)),\left.\frac{d}{ds}\right|_{s=r} \exp(p(0),sv(0))
\right\rangle= \left\langle \left.\frac{d}{dt}\right|_{t=0} p(t),v(0) \right\rangle.
\]
\end{lemma}

\proof
This is an easy exercise in Riemannian geometry. 
\endproof

Theorem \ref{thm:local tube} may be applied to the $\CP^n_\lambda$ via

\begin{proposition}
 \begin{equation}\label{varmulambda}
  \delta \mu_k^\lambda=\frac{\omega_{2n-k-1}}{\omega_{2n-k}}\left(2\pi \Delta_{k-1}-\lambda\sum_{q=0}^{\lfloor \frac
k2\rfloor}(k-2q+1)B_{k+1,q}\right).
 \end{equation}
\end{proposition}
\proof
This follows from Proposition 3.7 in \cite{ags}.
\endproof

\begin{corollary} \label{cor:gray local} In $\CPn_\lambda$,
 \begin{equation}
T_r=\vol +\sum_{k=0}^{2n} \omega_{2n-k-1}\Big(2\pi
\Delta_{k-1}-\lambda\sum_{q=0}^{\lfloor \frac k2\rfloor}(k-2q+1)B_{k+1,q}\Big)  \int_0^r \sn_\lambda^{2n-k} (s
)\cs_\lambda^{k}(s) ds \label{agv} 
\end{equation}
\end{corollary}

Formula \eqref{agv} was obtained in \cite{abbena et al} by classical differential geometric methods.


\section{$\curv^{U(n)}$ as a module over $\Val^{U(n)}$}\label{sect:module_cn}

By Corollary \ref{coro_module_versus_kin}, the semi-local kinematic formulas can be obtained from the
module structure. Since $\valun$ is generated by $t$ and $s$, it is enough to determine the action of these
valuations on $\curvun$. This will be achieved in this section. 

The basic outline is as follows. Multiplication by $s$ is determined by two properties, which will be shown
below: it is independent of the curvature $\lambda$, and it maps the span of the $B_{mp}$'s to itself. 

Multiplication by $t$ is determined by the Angularity Theorem \ref{thm_angular_flat} and the
obvious fact that it commutes with multiplication by $s$.

\subsection{Characterization of  the curvature measures $B_{kq}$}\label{caracterization}

Put 
$\mathrm{Beta}\subset \curvuinf$ for the space of all curvature measures $\Phi$ with the following property. Let
$k$ be  a nonnegative integer, $A \subset \Cn$  a differentiable polyhedron, $F^{2k} \subset A$ a face of
dimension $2k$ and $x \in F$. Suppose that $T_xF$ is a complex $k$-plane. Then the density at $x$ of 
$\Phi(A,\cdot)|_F$ with respect to $2k$ dimensional volume is zero. 
Let $\mathrm{Beta}_m=\mathrm{Beta}\cap \curv_m^{U(\infty)}$. 

\begin{proposition}\label{prop:b=b}
 $\mathrm{Beta}_m =\spann\{B_{mp}\}$.
\end{proposition}

\begin{proof} 
Recall that each $B_{mp}$ is represented by a multiple of $\beta$. Let $F$ be a face of a smooth polyhedron,
and assume that $T_xF$ is a complex space for some $x\in F$. Then $\beta$ vanishes at every $\xi\in N(F)$
lying over $x$. Hence $\spann\{B_{mp}\}\subset \mathrm{Beta}_m$. 

We claim that $\mathrm{Beta}_m\cap { \ang}=\{0\}$. In order to prove the claim, suppose that $\Phi=\sum d_j
\Delta_{m,j}$ belongs to $\mathrm{Beta}_m$. We will show that $\Phi=0$. 

Let $z_j= x_j + i y_j, \ j = 0,\dots k$ be coordinates for $\C^{k+1}$, and consider the domain
$$
D_k:= \left\{\vec z: y_{0}\le \frac 1 2 (y_1^2 +\dots+ y_k^2), \ x_{0} \le 0\right\}{\subset \C^{k+1}.}
$$  Thus the boundary of $D_k$ consists of two smooth pieces meeting along the $2k$-dimensional submanifold
$$V_k:= \left\{x_{0} = 0, y_{0}= \frac 1 2 (y_1^2 +\dots+ y_k^2) \right\}.
$$
Now, for any $p\leq \frac m2$, put $k:= m-2p$ and consider the domain
$$
E_{kp}:= D_k \times \C^{p} \subset \C^{ k+1+p}.
$$
The boundary again has two smooth pieces, meeting along the singular locus  
$$
W_{kp} := V_k \times \C^p,
$$
which has tangent plane $T_0W_{kp} = 0 \times \C^{k+p}$.
We show that the density of $\Delta_{mj}({E_{m-2p,p}},\cdot)|_{W_{kp}}$ at 0 is 
$$
\Theta^{2k}(\Delta_{mj}({E_{m-2p,p}},\cdot){|_{W_{kp}}},0)=\frac{2(j-p+1)}{m-2p+2}\Theta^{2k}(\Gamma_{mj}({
E_{m-2p,p}},\cdot){|_{W_{kp}}},0)=\sigma_p\delta_j^p
$$
for a constant $\sigma_p\neq 0$. 

Let $z_j=x_j+iy_j$ and $w_l = u_l +i v_l$ (with $j=0,\ldots,k$ and $l=1,\ldots,p$) be coordinates in
$\C^{k+p+1}$. Let $\xi_j, \eta_j $ and $\mu_l,\nu_l$ be the corresponding coordinates for the fiber of
$T\C^{k+p+1} \to \C^{k+p+1}$. The fiber of $N(E_{kp})$ at  $ (\frac i 2 (y_1^2+\dots y_k^2),z_1, \dots,
z_k,w_1,\dots,w_p)\in W_{kp}$ is the quarter circle
\[
 \left\{\cos\psi(1,0,\ldots,0)+i\frac{\sin\psi}{\sqrt{1+y_1^2+\ldots+y_k^2}}(1,-y_1,\ldots,-y_k,0,\ldots,
0)\colon 0\leq \psi\leq \frac\pi{ 2}\right\}.
\]
This defines a  parametrization $F(\psi,\vec z, \vec w)$ of the $\left. N(E_{kp})\right|_{W_{kp}}$.
Evaluating at $(\psi, 0,0)$,
\begin{align*}
F^*dx_0&= F^* dy_{0} = 0, \\
F^*d\xi_0={-\sin\psi} d\psi,&\quad F^*d\eta_0={\cos\psi} d\psi,\\
F^*d\xi_j= 0,&\quad F^* d\eta_j = -\sin\psi d y_j, \quad j =1,\dots k,\\
F^*d\mu_l&= F^* d\nu_{l} = 0.
\end{align*}   
Since
\begin{align*}
 \theta_0 &=\sum_j d\xi_j \wedge d\eta_j+\sum_l d\mu_l \wedge d\nu_l,\\
\theta_1 &= \sum_j (dx_j \wedge d\eta_j + d\xi_j \wedge dy_j)+ \sum_l(du_l \wedge d\nu_l + d\mu_l
\wedge dv_l),\\
\theta_2 &= \sum_j dx_j \wedge dy_j + \sum_ldu_l \wedge dv_l,\\
\end{align*}
one computes
\begin{align*}
F^*\theta_0&=0\\
F^*\theta_1 &= - \sin\psi \sum_{j\ne 0} dx_j \wedge dy_j  \\
F^*\theta_2 &=  \sum_{j\ne 0} dx_j \wedge dy_j + \sum_ldu_l \wedge dv_l 
\end{align*}
at $(\psi,0,0)$. Hence
\[F^*(\gamma_{mj})=\frac{c_{nkq}}{2}F^*(\gamma \wedge \theta_0^{j-p} \wedge \theta_1^{m-2j}
\wedge \theta_2^j)=0\]
unless $j=p$. In this case
\begin{equation}
\left.F^*(\gamma_{mp})=F^*(\gamma \wedge \theta_1^k \wedge \theta_2^p) \right|_0= d\psi \wedge (-\sin\psi)^k
k! p! \bigwedge_j dx_j \wedge dy_j \wedge \bigwedge_l du_l \wedge dv_l.
\end{equation}
So the density at $0$ is nonzero, as claimed. This shows that $\mathrm{Beta}_m\cap \ang=\{0\}$.

We deduce that
\[
 \curv_m^{U(n)}=\spann\{B_{mp}\}\oplus\ang_m^{U(n)}\subset\mathrm{Beta}_m \oplus\ang_m^{U(n)}\subset \curvun_m,
\]
and the result follows.    
\end{proof}

\subsection{Multiplication by $s$}\label{sect:mult s}
From this point we identify $\mathcal{C}^{G_\lambda}(\CP_\lambda^n)$ and $\curvun$ by Proposition
\ref{prop_isom_curv_general}. Our next goal is to determine the action of $s\in\V^n_\lambda$ on $\curvun$. By
Proposition \ref{prop:module_analytic} and  Theorem 1 of \cite{ags} this action depends analytically on
$\lambda$.
\begin{proposition} \label{prop:s indep}
The action of $s$ on $\curvun$ is independent of $\lambda$.
\end{proposition}

\proof 
By analytic continuation, it is enough to show this for $\lambda >0$. Recall that $\vol(G_\lambda)=\vol
\CP^n_\lambda=\frac{\pi^n}{\lambda^n n!}$ and that 
\begin{displaymath}
 s= \frac{\lambda^{n-1} n!}{\pi^{n}} \int_{G_\lambda} \chi( gP_\lambda \cap \cdot)\, dg,
\end{displaymath}
where $P_\lambda$ is a fixed complex hyperplane.

Thus for $\phi \in \mathcal{V}^n_\lambda$ we have by  Proposition \ref{prop:pd phi},
\begin{equation*}
 \langle \pd_\lambda \phi,s\rangle 
  =\frac{\lambda^{n-1} n!}{\pi^{n}} \phi(P_\lambda).
\end{equation*}
Now 
\begin{align*}
 [B_{kq}]_\lambda(P_\lambda) & = 0,\\
 [\Gamma_{kq}]_\lambda(P_\lambda) & =0, \quad (k,q) \neq (2n-2,n-1)\\
[\Gamma_{2n-2,n-1}]_\lambda(P_\lambda) & =
\vol(P_\lambda)=\vol(\CP^{n-1}_\lambda)=\frac{\pi^{n-1}}{\lambda^{n-1}(n-1)!}.
\end{align*}
Hence, {for $\Phi\in \curvun$},
\begin{displaymath}
 \langle \pd_\lambda \circ \glob_\lambda \Phi,s\rangle=\frac n{ \pi} \langle \Phi,\Gamma_{2n-2,n-1}^*\rangle,
\end{displaymath}
where $\Gamma_{2n-2,n-1}^*$ is the indicated element of the basis dual to the $B_{k,q},\Gamma_{k,q}$. 

By Corollary \ref{coro_module_versus_kin} it follows that for $\Phi \in \curvun$ 
\begin{align*}
 s \cdot \Phi & = \left\langle (id \otimes( \pd_\lambda \circ \glob_\lambda))\circ K(\Phi),s\right\rangle \\
& ={\frac{n}{\pi}} \left\langle K(\Phi),\Gamma_{2n-2,n-1}^*\right\rangle.
\end{align*}

By the transfer principle Theorem \ref{thm_transfer}, the last expression is independent of $\lambda$.
\endproof

\begin{proposition}\label{prop:sb sub b}
$$
s\mathrm{Beta}_k \subset \mathrm{Beta}_{k+2}.
$$
\end{proposition}
\begin{proof} 
Suppose $\Phi\in\mathrm{Beta}_{k}$. Let $x\in F\subset A$ be as in the beginning of Subsection
\ref{caracterization}.  {By Corollary \ref{cor_product_special}}, for any  Borel set
$U\subset F$
\[
 s\cdot \Phi(A,U)=\int_{{\AGr}_\C(n-1)}\Phi(H\cap A,H\cap U) dH
\]
\[
 =\int_{{\AGr}_\C(n-1)}\left(\int_{H\cap U}g_{H}(y)\, d_y\vol_{H\cap F}\right) dH,
\]
where $g_{H}\in C^\infty(H\cap F)$  is the density of $\Phi(H\cap A, \cdot )|_{H\cap F}$ with respect to $d\vol_{H\cap
F}$. The coarea formula applied to the map $p\colon F\times \Gr_\C(n-1)\rightarrow \AGr_\C(n-1)$ given by $p(y,\vec
H)=y+\vec H$ yields
\[
 s\cdot\Phi(A,U)=\int_{U}\left(\int_{{\Gr}_\C(n-1)}g_{y+ \vec H}(y)\cdot  \mathrm{jac}\, p(y,\vec H)\ d \vec H\right)
d_y\vol_{F},
\]
where $ \mathrm{jac}\, p$ is the Jacobian of $p$. The latter integral between brackets  is the density function of
$s\cdot\Phi(A)$ with respect to  $d\vol_F$. It vanishes for $y=x$, since $g_{H}(x)$ vanishes for every $H$. Hence
$s\cdot \Phi\in\mathrm{Beta}_{k+2}$.
\end{proof}

\begin{lemma}\label{lem:Phi k lambda} Put
\begin{align}
\Phi^\lambda_k & := \sum_{j\ge k} \frac{j!\lambda^{j-k}}{\pi^j}\left((j-k+1) \Gamma_{2j,j}+\sum_{q\le j-1}
\frac{1}{4^{j-q}} \binom{2j-2q}{j-q} B_{2j,q}\right) \in \curvuinf. \label{eq_ags_sk}
\end{align}
Then
\begin{align}\label{eq:s_to_mu}
s^k= [\Phi^\lambda_k]_\lambda &= \sum_{j\ge k}\frac{j!\lambda^{j-k}}{\pi^j}\sum_{q\leq j} \frac{1}{4^{j-q}} \binom{2j-2q}{j-q}
\mu_{2j,q}^\lambda
\in \V^\infty_\lambda.
\end{align}
\end{lemma}
\begin{proof} That $s^k= [\Phi^\lambda_k]_\lambda $ was shown in \cite{ags}. To prove the second equality, it is enough
to check
\begin{align*}
\sum_{q\ge k} \frac{q!\lambda^{q-k}}{\pi^q} \mu_{2q,q}^\lambda &=\sum_{q\geq k}\frac{q!\lambda^{q-k}}{\pi^q}\sum_{i\geq
0}\frac{\lambda^i(q+i)!}{\pi^i q!}[\Gamma_{2q+2i,q+i}]_\lambda \\
&=\sum_{q\geq k}\sum_{j\geq q} \frac{\lambda^{j-k} j!}{\pi^j}[\Gamma_{2j,j}]_\lambda\\
&= \sum_{j\ge k}\sum_{q=k}^j \frac{j!\lambda^{j-k}}{\pi^j}[ \Gamma_{2j,j}]_\lambda = \sum_{j\ge
k}\frac{j!\lambda^{j-k}}{\pi^j}(j-k+1)[ \Gamma_{2j,j}]_\lambda.
\end{align*}
\end{proof}

\begin{lemma}\label{lem:sigma Gamma} For $\lambda \in \R$, define the linear map $\tilde\sigma:
\spann \left\{\Gamma_{2q,q}: q\ge 0\right\} \to \V^\infty_\lambda$ by
$$
\tilde\sigma \Gamma_{2q,q} = \left[\frac{q+1}{\pi} \Gamma_{2q+2,q+1} + \frac{1}{2\pi(q+1)} B_{2q+2,q} +
\frac{1}{2\pi(q+2)} N_{2q+2,q}\right]_\lambda
$$
Then 
$$
\tilde\sigma\left( \sum_{k\ge 0} \left(\frac {\lambda}{ \pi}\right)^k\frac{(q+k)!}{q!} \Gamma_{2q+2k,q+k}\right) = 
\frac{1}{2\pi(q+1)}\mu^\lambda_{2q+2,q}+\frac{q+1}{\pi}\mu_{2q+2,q+1}^\lambda.$$
\end{lemma}
\begin{proof} We
compute using \eqref{eq:n = b}
\begin{align*}
\tilde\sigma\left( \sum_{k\ge 0} \left(\frac {\lambda}{ \pi}\right)^k{(q+k)!} \Gamma_{2q+2k,q+k}\right)  &=  \frac 1
\pi\sum_{k\ge 0} \left(\frac\lambda \pi\right)^k (q+k+1)! [\Gamma_{2q+2k+2,q+k+1}]_\lambda \\
&+ \frac 1 {2\pi}\sum_{k\ge 0} \left(\frac\lambda \pi\right)^k \frac{(q+k)!}{q+k+1} [B_{2q+2k+2,q+k}]_\lambda \\
&- \frac{\lambda}{2\pi^2}\sum_{k\ge 0} \left(\frac\lambda \pi\right)^k (q+k)!
\frac{q+k+1}{q+k+2}[B_{2q+2k+4,q+k+1}]_\lambda.
\end{align*}
The first of these sums is $\frac{(q+1)!}{\pi} \mu_{2q+2,q+1}^\lambda$, while the second and third collapse to give
$\frac{q!}{2\pi(q+1)} \mu^\lambda_{2q+2,q}$, which establishes the result.
\end{proof}

\begin{lemma} \label{prop_char_mults}
 There is at most one linear operator $\sigma:\curvuinf \to \curvuinf$ with the following properties:
\begin{enumerate}
\item[(i)] $\sigma$ is of degree $2$;
 \item[(ii)] $\glob_0(\sigma \Phi)=s \glob_0(\Phi)$ for all $\Phi \in \curvuinf$;
\item[(iii)] $\sigma(\mathrm{Beta})\subset \mathrm{Beta}$; 
\item[(iv)] $\sigma(\ker \glob_\lambda)\subset\ker \glob_\lambda$ for all $\lambda$, i.e. $\sigma$ induces an
endomorphism of $\V^\infty_\lambda$;
\item[(v)] for every $\lambda$, if $s^k=\glob_\lambda(\Phi)$ for some  $\Phi\in \curvuinf$, then
$s^{k+1}=\glob_\lambda(\sigma \Phi)$. 
\end{enumerate}
\end{lemma}

\proof
Suppose $\sigma_1,\sigma_2$ are two operators with these properties and set $\sigma:=\sigma_1-\sigma_2$. We want to show
that $\sigma=0$. By (ii) we have 
\begin{equation} \label{eq_glob0}
 \glob_0(\sigma \Phi)=0 \quad \forall \Phi \in \curvuinf.
\end{equation}

By (iii) we have 
\begin{displaymath}
 0=\glob_0(\sigma B_{k,q})=\glob_0\left(\sum c_p B_{k+2,p}\right)=\sum_p c_p \mu_{k+2,p}.
\end{displaymath}
Since the valuations $\mu_{k+2,p}$ are linearly independent in $\curvuinf$, we must have $c_p=0$ for all $p$. Hence
$\sigma$ vanishes on $\mathrm{Beta}$.

Let $\max\{0,k-n+1\}\leq q<\frac{k}{2}$. By Lemma \ref{lem:gamma beta translation}
\begin{displaymath}
 \Gamma_{kq}-B_{kq}+\frac{(2n-k)(q+1)\lambda}{2(n-k+q)\pi}B_{k+2,q+1} \in \ker \glob_\lambda.
\end{displaymath}
By (iv) and using what we have shown we obtain that $\glob_\lambda(\sigma \Gamma_{kq})=0$. Proposition
\ref{prop_kernels_glob} implies that $\sigma \Gamma_{kq}=0$. Note that this argument breaks down if $k=2q$.
 
Let us finally prove by reverse induction on $q$ that $\sigma \Gamma_{2q,q}=0$ {in $\curvun$} for all
$q=0,\ldots,n$. For $q=n$ this is trivial. Suppose that  $\sigma \Gamma_{2q,q}=0$ for all $q>k$. 
By (v) and Lemma \ref{lem:Phi k lambda} we have $\sigma \Phi^\lambda_k \in \ker \glob_\lambda$. On the other hand, by
the induction hypothesis and by what we have already shown,
\begin{displaymath}
 \Phi_k^\lambda \in \frac{k!}{\pi^k} \Gamma_{2k,k}+\ker \sigma
\end{displaymath}
It follows that $\sigma \Gamma_{2k,k} \in \ker \glob_\lambda$. Using  Proposition \ref{prop_kernels_glob} we obtain that
$\sigma \Gamma_{2k,k}=0$.
\endproof

\begin{proposition} \label{cor_mults}
The action of $s$ on $\curv^{U(\infty)}$ is given by
\begin{align}
\label{eq:sb} s\cdot B_{kq} & = \frac{(k-2q+2)(k-2q+1)}{2\pi(k+2)}B_{k+2,q}+\frac{2(q+1)(k-q+1)}{\pi(k+2)}B_{k+2,q+1}\\
\notag s \cdot \Delta_{kq} &
=\frac{(k-2q+2)(k-2q+1)}{2\pi(k+2)}\Delta_{k+2,q}+\frac{2(q+1)(k-q+1)}{\pi(k+2)}\Delta_{k+2,q+1}\\
\label{eq:s delta}& \quad
-\frac{(k-2q+2)(k-2q+1)}{\pi(k+2)(k+4)}N_{k+2,q}-\frac{2(q+1)(k-2q)}{\pi(k+2)(k+4)}N_{k+2,q+1}.
\end{align}
In particular,
\begin{align}\label{sN}
   s\cdot N_{kq} & =\frac{(k-2q+2)(k-2q+1)}{2\pi(k+4)}N_{k+2,q}+\frac{2(q+1)(k-q+2)}{\pi(k+4)}N_{k+2,q+1}
\end{align}
and
\begin{align}\label{eq:s gamma}
s\cdot \Gamma_{2q,q} &= \frac{q+1}{\pi} \Gamma_{2q+2,q+1} + \frac{1}{2\pi(q+1)} B_{2q+2,q} +
\frac{1}{2\pi(q+2)} N_{2q+2,q}.
\end{align}
\end{proposition}

\proof By Proposition \ref{prop:s indep}, multiplication by $s$ is an operator on $\curvuinf$ which is independent of
the curvature $\lambda$, and which clearly enjoys properties (i,ii,iv,v) of Lemma \ref{prop_char_mults}. Proposition
\ref{prop:sb sub b} states that property (iii) is satisfied as well.

Now let $\sigma$ be the linear operator on $\curvuinf$ defined by the formulas above. We claim that $\sigma$ satisfies
the properties in Lemma \ref{prop_char_mults}. By the uniqueness statement of Lemma \ref{prop_char_mults}, it follows
that $s=\sigma$.

Properties (i), (iii) are trivial. Since $\glob_0(B_{kq})=\glob_0(\Delta_{kq})=\mu_{kq}$, property (ii) follows from
the formula 
\begin{equation}\label{eq:s mukq}
  s \cdot\mu_{kq}  =
\frac{(k-2q+2)(k-2q+1)}{2\pi(k+2)}\mu_{k+2,q}+\frac{2(q+1)(k-q+1)}{\pi(k+2)}\mu_{k+2,q+1},
\end{equation}
which follows from  equation (37) of \cite{befu11}.
The kernel of $\glob_\lambda$ is spanned by the measures $N_{kq} + \lambda\frac{q+1}{\pi}B_{k+2,q+1}$.
Using
this, one checks directly that $\sigma(\ker \glob_\lambda)\subset \ker \glob_\lambda$, which is property (iv).

To establish property (v), observe that, by \eqref{eq:sb} and Lemma \ref {lem:sigma Gamma}, the induced maps
$\bar \sigma_\lambda : \V^\infty_\lambda\to \V^\infty_\lambda$ are intertwined by the linear isomorphisms
$\V^\infty_\lambda\to \V^\infty_{\lambda'}$ determined by $\mu^\lambda_{kq}\mapsto \mu^{\lambda'}_{kq}$. On
the other hand, we know {from \eqref{eq:s mukq}} that {$\bar \sigma:=\bar \sigma_0$} agrees with
multiplication by $s$, i.e.
\begin{align*}
\bar \sigma \left(\frac{k!}{\pi^k }\sum_q 4^{q-k} \binom{2k-2q}{k-q}\mu_{2k,q}\right)&= \bar \sigma(s^k) =
s^{k+1} \\
&=\frac{(k+1)!}{\pi^{k+1} }\sum_q 4^{q-k-1} \binom{2k+2-2q}{k+1-q}\mu_{2k+2,q}.
\end{align*}
Therefore, taking $\lambda'=0$,
$$
\bar \sigma_\lambda\left(\frac{k!}{\pi^k }\sum_q 4^{q-k} \binom{2k-2q}{k-q}\mu^\lambda_{2k,q}\right)
=\frac{(k+1)!}{\pi^{k+1} }\sum_q 4^{q-k-1} \binom{2k+2-2q}{k+1-q}\mu^\lambda_{2k+2,q}.
$$
In view of Lemma \ref{lem:Phi k lambda}, this implies that  $\bar \sigma_\lambda(s^k) = s^{k+1}$ for all $\lambda \in
\R$.
\endproof 

\begin{proof}[Proof of Proposition \ref{prop:s tau}] In view of the validity of the Proposition for $\lambda = 0$,
this follows from the intertwining property in the last paragraph.
\end{proof}

\subsection{Multiplication by $t$}
We complete the description of the $\val^{U(\infty)}$-module structure of $\curv^{U(\infty)}$ by determining
the action of $t \in \val^{U(\infty)}$ on $\curv^{U(\infty)}$. Notice that by Theorem
\ref{thm_angular_flat} and the results in \cite{befu11} we already know that
\[
t \Delta_{kq}  = \frac{\omega_{k+1}}{\pi \omega_k} \left( (k-2q+1)\Delta_{k+1,q}+2(q+1)\Delta_{k+1,q+1}\right).
\]
It remains to find $t N_{kq}$. By using that $ts\Delta_{00}=st\Delta_{00}$ one gets easily
\begin{equation}\label{tn20}
 t N_{20}=\frac{16}{5\pi}N_{30}+\frac{16}{15\pi} N_{31}.
\end{equation}

\begin{lemma} \label{lemma_tn10}
\begin{displaymath}
  tN_{10}=\frac34 N_{20}.
\end{displaymath} 
\end{lemma}

\proof
We know that $tN_{10}=cN_{20}$ for some constant $c$.  Let us work in dimension $n=3$. The principal kinematic formula
on $\C^3$ is given by (cf.{\cite{befu11, pa02}})
\begin{align}
 k(\chi) & = \mu_{00} \otimes \mu_{63} + \mu_{63} \otimes \mu_{00} +\frac{16}{15 \pi}\mu_{10} \otimes \mu_{52} +
\frac{16}{15 \pi}\mu_{52} \otimes \mu_{10}\label{pkfC3} \\
& \quad + \frac{5}{24} \mu_{20} \otimes \mu_{41}+\frac{5}{24} \mu_{41} \otimes \mu_{20}+\frac16 \mu_{20} \otimes
\mu_{42}+\frac16 \mu_{42} \otimes \mu_{20}\notag\\
& \quad +\frac16 \mu_{21} \otimes \mu_{41} +\frac16 \mu_{41} \otimes \mu_{21}+\frac13 \mu_{21} \otimes \mu_{42}
+\frac13 \mu_{42} \otimes \mu_{21}\notag\\
& \quad +\frac{2}{3\pi} \mu_{30} \otimes \mu_{30}+\frac{4}{9\pi} \mu_{30} \otimes \mu_{31} +\frac{4}{9\pi} \mu_{31}
\otimes \mu_{30}
+\frac{16}{27\pi} \mu_{3 1} \otimes \mu_{31}.\notag 
\end{align}
We multiply it by $\chi\otimes N_{10}$ and obtain, according to the first equality in  \eqref{eq:bark mult},  
\begin{align*}
 \bar k(N_{10}) =\mu_{63} \otimes N_{10} + \frac{8 c}{15}\mu_{52} \otimes N_{20} +
\left(\frac{2}{15}c-\frac{1}{10}\right) \mu_{41} \otimes N_{31}+\frac25 \mu_{42} \otimes N_{31}
\end{align*}
where we have used \eqref{mu2st} to express $\mu_{kq}$ in terms of $s,t$, as well as equations \eqref{sN} and
\eqref{tn20}.  
Since the kernel of the globalization map is spanned by $N_{10}, N_{20}, N_{31}$ we deduce, by the cocommutativity of
$K$,
\begin{align*}
 K(N_{10}) & = N_{10} \otimes \Delta_{63} +\Delta_{63} \otimes N_{10} + \frac{8 c}{15}( N_{20} \otimes
\Delta_{52}+\Delta_{52} \otimes N_{20})  \\
& \quad + \left(\frac{2}{15}c-\frac{1}{10}\right) (N_{31} \otimes \Delta_{41}+\Delta_{41} \otimes N_{31})+\frac25
(N_{31} \otimes \Delta_{42}+\Delta_{42} \otimes N_{31}).
\end{align*}
Globalizing in $M^3_\lambda$ yields the kinematic formula $k_\lambda(\mu_{31}^\lambda)$. In that fomula,
$\mu_{4,1}^\lambda\otimes\mu_{5,2}^\lambda$ appears with coefficient $\frac{4c-1}5$. On the other hand,  from
\eqref{kinform_mu}
\[
 k_\lambda(\mu_{31}^\lambda)=(F_\lambda\otimes F_\lambda)\circ (\mu_{31}\otimes (1-\lambda s))\cdot k_0(\chi).
\]
By \eqref{pkfC3}, using 
\[
 \mu_{31}\cdot \mu_{10}=\frac{3\pi}{8}\mu_{41}+\frac{3\pi}4\mu_{42}
\]
 we find that coefficient to be $\frac 25$. Therefore $c=\frac 34$.
\endproof

\begin{proposition} \label{prop_char_multt}
 There is at most one linear operator $\theta:\curv^{U(\infty)} \to \curv^{U(\infty)}$ with the following properties:
\begin{enumerate}
\item[(i)] $\theta$ is of degree $1$;
 \item[(ii)] $\glob_0(\theta \Phi)=t \glob_0(\Phi)$ for all $\Phi \in \curv^{U(\infty)}$;
\item[(iii)] $\theta(\ang^{U( \infty)})\subset\ang^{U( \infty)}$;
\item[(iv)] $\theta$ commutes with multiplication by $s$. 
\item[(v)] $\theta N_{1,0}=\frac{3}{4} N_{2,0}$.  
\end{enumerate}
\end{proposition}

\proof
Suppose that $\theta_1,\theta_2$ are such operators. We use induction on $k$ to show that $\theta:=\theta_1-\theta_2$
vanishes on $\curv^{U(\infty)}_k$. 

By (ii), we have $\glob_0(\theta \Phi)=0$. The globalization map $\glob_0:\curv^{U(\infty)} \to \Val^{U(\infty)}$ is
injective on $\ang^{U(\infty)}$. Hence $\theta$ vanishes on $\ang^{U(\infty)}$ by (iii). 
Using (v) it follows that $\theta$ vanishes in degrees $0$ and $1$. 

Let us now assume that $k \geq 2$. By (iv) and the induction hypothesis, $\theta$ vanishes on the image of
$s:\curv^{U(\infty)}_{k-2} \to \curv^{U(\infty)}_k$. 
The proof will be finished by showing that the image of  $s:\curv^{U(\infty)}_{k-2} \to \curv^{U(\infty)}_k$ and
$\ang^{U(\infty)}_k$ span $\curv^{U(\infty)}_k$.

From Proposition \ref{cor_mults} we obtain that for $q<\frac{k-2}{2}$  
\begin{align*}
sN_{k-2,q} & \equiv a_{11} N_{k,q}+a_{12} N_{k,q+1} \mod \ang^{U(\infty)}_k\\
s\Delta_{k-2,q} & \equiv a_{21} N_{k,q}+a_{22} N_{k,q+1} \mod \ang^{U(\infty)}_k
\end{align*}
with some non-degenerate $2 \times 2$-matrix $a$. 
\endproof

\begin{proposition} \label{prop_module_mult_t}
  \begin{align*}
 t\cdot N_{kq} & =\frac{\omega_{k+1}}{\pi \omega_k} \frac{k+2}{k+3} \left((k-2q+1)
N_{k+1,q}+\frac{2(q+1)}{k-2q}(k-2q-1) N_{k+1,q+1}\right), \quad 0 \leq q<\frac{k}{2}  \\
t\cdot \Delta_{kq} & = \frac{\omega_{k+1}}{\pi \omega_k} \left( (k-2q+1)\Delta_{k+1,q}+2(q+1)\Delta_{k+1,q+1}\right).
\end{align*}
\end{proposition}

\begin{proof}
Let $\theta$ be the linear operator on $\curv^{U(\infty)}$ which is defined by these formulas. We claim that $\theta$
satisfies the properties in Proposition \ref{prop_char_multt}. Indeed, (ii) follows from \cite{befu11}, (iv) is a
computation and the other properties are trivial. 

Since multiplication by $t$ has the same properties (see Lemma \ref{lemma_tn10} and Theorem \ref{thm_angular_flat}),
Proposition \ref{prop_char_multt} implies that they agree. 
\end{proof}

\subsection{The $\el$- and $\en$-maps}

 \begin{definition}
Define two natural maps from $\Val^{U(\infty)}$ to $\curv^{U(\infty)}$ by
\begin{align*}
 \el: \Val^{U(\infty)} & \to \curv^{U(\infty)}\\
\phi & \mapsto \phi \Delta_{0,0}
\end{align*}
and  
\begin{align*}
 \en: \Val^{U(\infty)} & \to \curv^{U(\infty)}\\
\phi & \mapsto \phi N_{1,0}.
\end{align*}
\end{definition}

\begin{lemma} \label{lemma_l_uk} Let $u=4s-t^2\in\Val^{U(\infty)}$. Then
 \begin{align*}
  \el(u^k) & =\frac{(2k)!}{\pi^kk!} \left(\Delta_{2k,k}-\frac{1}{k+1}N_{2k,k-1}\right)\\
  \en(u^k) & = \frac{4^k k!}{\pi^k} N_{2k+1,k}\\
\el(t^k) & = \frac{\omega_k k!}{\pi^k} \sum_{i=0}^{\left\lfloor \frac{k}{2}\right\rfloor} \Delta_{ki}\\
\en(t^k) & = \frac{3}{4}k!\frac{\omega_{k+3}}{\pi^{k+1}} \sum_{i=0}^{\left\lfloor \frac{k}{2}\right\rfloor}
(k-2i+1)N_{k+1,i}. 
 \end{align*}
\end{lemma}
\proof
Induction on $k$, using the relations
\begin{align}
 u\Delta_{kq}=&\frac{2}{\pi(k+2)}\Big(2(q+1)(2q+1)\Delta_{k+2,q+1}-4(q+1)(q+2)\Delta_{k+2,q+2}\\
&-\frac{2(k-2q+2)(k-2q+1)}{k+4}N_{k+2,q}-\frac{4(q+1)(k-2q)}{k+4}N_{k+2,q+1}\Big)\notag\\
uN_{kq}=&\frac{4(q+1)}{\pi(k+4)}\Big((2q+5)N_{k+2,q+1}-\frac{2(q+2)(k-2q-2)}{k-2q}N_{k+2,q+2}\Big),
\end{align}
which follow directly from Propositions \ref{cor_mults} and \ref{prop_module_mult_t}.
\endproof

\subsection{The modules $\curv^{U(\infty)}$ and $\curv^{U(n)}$}

\begin{theorem} \label{thm_free_module}
 The $\Val^{U(\infty)}$-module $\curv^{U(\infty)}$ is freely generated by $\Delta_{00}$ and $N_{10}$, i.e.
\begin{equation}\label{eq:free_decomposition}
 \curv^{U(\infty)}= \Val^{U(\infty)}\Delta_{00} \oplus \Val^{U(\infty)}N_{10} = \image \el \oplus \image \en.
\end{equation}
\end{theorem}

\proof
We use induction on $k$ to show that the map
\begin{align*}
 \psi_k: \Val_k^{U(\infty)} \oplus \Val_{k-1}^{U(\infty)} & \mapsto \curv_k^{U(\infty)}\\
(\phi_1,\phi_2) & \mapsto \phi_1 \Delta_{00}+\phi_2 N_{10}
\end{align*}
is onto. 
For $k=0$ this is clear. For $k=1$ we have $t\Delta_{00}=\frac{2}{\pi}\Delta_{10}$.

Let us now suppose that $k \geq 2$ and that we know the result for $k-1$ and $k-2$. 

Let us recall that
\begin{align}
  t N_{k-1,q} & =\frac{\omega_{k}}{\pi \omega_{k-1}} \frac{k+1}{k+2} \left(
(k-2q)N_{kq}+\frac{2(k-2q-2)(q+1)}{k-2q-1} N_{k,q+1}\right), 0 \leq q<\frac{k-1}{2} \label{eq_tN}\\
sN_{k-2,q} & = \frac{(k-2q)(k-2q-1)}{2\pi(k+2)} N_{kq}+\frac{2(q+1)(k-2q)}{\pi(k+2)}N_{k,q+1}, \label{eq_tT} 
0 \leq q<\frac{k-2}{2}.
\end{align}

The corresponding $2 \times 2$-matrix is non-singular. It follows that $N_{kq}$ lies in the image of
$\psi_k$ for all $0 \leq q < \frac{k}{2}$. Hence   $\nul_k^{U(\infty)}$  is  a subset of the image of $\psi_k$. 

If $\Phi \in \curv_k^{U(\infty)}$ has globalization $\phi$, then $\Phi-\phi \Delta_{00} \in \nul_k^{U(\infty)}$ and it
follows that $\Phi$ is included in the image of $\psi_k$.

This shows that $\psi_k$ is onto. Now we compare dimensions: the dimension on the left hand side is 
$\dim  \Val_k^{U(\infty)}+\dim \Val_{k-1}^{U(\infty)}=\lfloor \frac{k}{2}\rfloor+1+\lfloor \frac{k-1}{2}\rfloor +1=k+1$,
which is the dimension of the right hand side. The map being onto, it must be an isomorphism of vector spaces.
\endproof

The following proposition allows us to find explicitly the decomposition \eqref{eq:free_decomposition} of a given element
of $\curv^{U(\infty)}$.

\begin{proposition} \label{prop_n_inv}
The $\en$-map is a linear isomorphism between $\Val^{U(\infty)}$ and $\nul^{U(\infty)}$. Its inverse is given by 
\begin{align*}
 \en^{-1}(N_{kq})= \frac{4(k+2)\pi^{k-1}}{(k-2q)q!\omega_k} \sum_{r=0}^{\lfloor \frac{k-2q-1}{2} \rfloor}
\frac{(-1)^r(q+r+1)!}{(k-2q-2r-1)!(2q+2r+3)!r!} t^{k-2q-1-2r}u^{q+r},
\end{align*}
where $u=4s-t^2$.
 \end{proposition}

\proof
By Theorem \ref{thm_free_module} each $\Phi \in \curv^{U(\infty)}$ can be uniquely written as
$\Phi=\phi_1 \Delta_{00}+\phi_2 N_{10}$ with $\phi_1,\phi_2 \in \Val^{U(\infty)}$. It follows that $\en$ is
injective. If $\Phi \in \ker \glob_0$, then $\phi_1=0$ and hence $\Phi=\en(\phi_2)$. Therefore $\en$ is onto. 

To prove the displayed formula, it is enough to check that it is compatible with multiplications by $t$ and $u$, which
is  straightforward.
\endproof

It remains to describe the restriction map $\curv^{U(\infty)}\to\curv^{U(n)}$ in terms of the decomposition
\eqref{eq:free_decomposition}. This follows from Lemma \ref{lem:mus that vanish}, together with  the
following proposition.

\begin{proposition} \label{prop_kernel_n_map}
Let $p \in \mathbb{C}[t,u]$. Then $\en(p)=0$ in $\curv^{U(n)}$ if and only if $p$ belongs to the ideal
$\langle g_{n-1},g_n\rangle$, where $g_n$ is the degree $n$ part in 
\begin{displaymath}
 e^t \frac{\sin \sqrt{u}-\sqrt{u} \cos \sqrt{u}}{2\sqrt{u}^3}.
\end{displaymath}
Explicitly
\begin{displaymath}
 g_n=\sum_{r=0}^{\lfloor \frac{n}{2}\rfloor} \frac{(-1)^r(r+1)}{(n-2r)!(2r+3)!} t^{n-2r}u^r.
\end{displaymath}
\end{proposition}

\proof
We first claim that the kernel of the restriction map 
\begin{displaymath}
\nul^{U(\infty)} \to \nul^{U(n)}
\end{displaymath}
is the submodule $S$ generated by $N_{n0}$ and $N_{n+1,0}$. Clearly, this kernel is spanned by all
$N_{kq}$ with $q \leq \min\{k-n,\frac{k-1}{2}\}$. In degree $n$, only $N_{n0}$ belongs to $S$, while in
degree $n+1$ the measures $N_{n+1,0}, N_{n+1,1}$ belong to $S$. We may write $N_{n+1,1}$ as a combination of $tN_{n0}$
and $N_{n+1,0}$. 

Now suppose inductively that for some  $l \geq n+2$ each $N_{kq}$ with $k<l, q \leq k-n$
belongs to $S$. By \eqref{eq_tN}, \eqref{eq_tT} and the induction hypothesis, $N_{lq} \in S$ for all $q \leq
\min\{l-n-1,\frac{l-1}{2}\}$. If $l \geq 2n-1$, then these measures span the degree $l$-part of the kernel. If $l \leq
2n-2$ then these measures, together with the measure $N_{l,l-n}$ span the degree $l$
part of the kernel. By \eqref{eq_tN}, $N_{l,l-n}$ equals some non-zero multiple of
$tN_{l-1,l-n-1}$ modulo $N_{l,l-n-1}$. In both cases, we obtain that the degree $l$ part of the kernel belongs to $S$.
By induction it follows that the kernel equals $S$.  

It now follows that for $p \in \mathbb{C}[t,u]$ we have $\en(p)=0$ in $\curv^{U(n)}$ if and only if $\en(p)=p_1
N_{n,0}+p_2 N_{n+1,0}$ for some polynomials $p_1,p_2$. Equivalently $p=p_1\en^{-1}(N_{n0})+p_2
\en^{-1}(N_{n+1,0}).$
Hence, the kernel of the restriction map above is the ideal generated by $g_{n-1}:=c_{n-1}^{-1}
\en^{-1}(N_{n0})$ and ${g_n:=c_n^{-1} \en^{-1}(N_{n+1,0})}$ (where $c_n:=\frac{4(n+2)\pi^{n-1}}{n
\omega_n}$). 

The explicit formula for $g_n$ follows by Proposition \ref{prop_n_inv}. From this the generating function is easily
obtained. 
\endproof


\section{Local kinematic formulas and module structure}\label{sect:module cpn}

In this section, we will consider the curvature $\lambda$ as a parameter. 
In order to distinguish between the valuations $t$ living on the different spaces $\CP^n_\lambda$, we will write
$t_\lambda$ for the valuation $t \in \mathcal{V}^n_\lambda$. The same applies to $u$ and $v$. In view of Proposition
\ref{prop:s indep}, we will not make this distinction for $s$.

\subsection{Local kinematic formulas}
Here we will find the local kinematic formulas $K(\Delta_{0,0})$ and $K(N_{1,0})$. By Theorem
\ref{thm_free_module} and  Theorem \ref{thm_ftaig_curv}, this determines the local kinematic
operator $K$ completely.

Recall from Subsection \ref{subsec_algiso} the map $J_\lambda=\pd_\lambda^{-1} \circ (I_\lambda^{-1})^* \circ
\pd:\mathcal{V}_0 \to \mathcal{V}_\lambda$ and the commuting diagram
\begin{displaymath}
 \xymatrix{\curvuinf \ar[r]^-K \ar[d]_{\glob_\lambda} & \Sym^2 \curvuinf \ar[d]^{\glob_\lambda \otimes
\glob_\lambda}\\ \mathcal{V}_\lambda^\infty \ar[r]^-{k_\lambda} \ar[d]_{J_\lambda^{-1}} & \Sym^2
\mathcal{V}_\lambda^\infty
\ar[d]^{J_\lambda^{-1} \otimes J_\lambda^{-1}}\\ \mathcal{V}^\infty_0 \ar[r]^-k& \Sym^2 \mathcal{V}_0^\infty}
\end{displaymath}
Define $H_\lambda:=J_\lambda^{-1} \circ \glob_\lambda: \curvuinf \to \mathcal{V}_0^\infty$. In particular
$H_0=\glob_0$. We will denote \[H_0'=\left.\frac d{d\lambda}\right|_{\lambda=0}H_\lambda,\quad \mbox{and}\quad
H_0''=\left.\frac {d^2}{d\lambda^2}\right|_{\lambda=0}H_\lambda.\] 
Let $p\in \R[[t,s]]=\Val^{U(\infty)}$. Define differential operators $D_1,D_2$ by 
\begin{align*}
 D_1p & :=\frac{t^2-2s}{2}p-\frac{tu}{4}\frac{\partial p}{\partial t}\\
 D_2p & :=-\frac{3\pi u t}{8}p+\frac{\pi u^2}{8}\frac{\partial p}{\partial t}.
\end{align*}
where $u=4s-t^2$. Equivalently, in terms of $t$ and $u$, we have
\begin{align*}
 D_1p & =\frac{t^2-u}{4}p+\frac{t^2u}{2}
\frac{\partial p}{\partial u}-\frac{ut}{4}\frac{\partial p}{\partial t}\\
 D_2p & = -\frac{3\pi u t}{8}p-\frac{\pi u^2 t}{4}
\frac{\partial p}{\partial u}+\frac{\pi u^2}{8} \frac{\partial p}{\partial t} .
\end{align*}

\begin{lemma} \label{lemma:h_lambda_series}
For all $p \in \mathbb{C}[t,s]$ we have 
\begin{align}
\label{eq_Hlambda_ell} H_\lambda(\el(p)) & =p+\frac{\lambda}{1-\lambda s} D_1(p)=p+\lambda D_1(p)+\lambda^2 s
D_1(p)+O(\lambda^3)\\
\label{eq_Hlambda_en} H_\lambda(\en(p)) & =\frac{\lambda}{1-\lambda s} D_2(p)=\lambda D_2(p)+\lambda^2 s
D_2(p)+O(\lambda^3).
\end{align}
In particular, 
 \begin{displaymath}
  H_0''=2s H_0'.
 \end{displaymath}
Moreover, $D_1$ and $D_2$ are well-defined operators on $\Val^{U(n)}$. 
\end{lemma}

\proof
Both sides of these equations are multiplicative with respect to $s$, hence it suffices to prove the
statements for powers of $t$. 

Now we compute 
\begin{align*}
 H_\lambda \left(\el(t^k)\right) & =\frac{\omega_k k!}{\pi^k}  H_\lambda \left(\sum_i \Delta_{ki} \right)&& \text{ by Lemma
\ref{lemma_l_uk}}\\  
& = \frac{\omega_k k!}{\pi^k} J_\lambda^{-1}\left(  \sum_i[\Delta_{ki}]_\lambda\right)\\
& = \frac{\omega_k k!}{\pi^k} J_\lambda^{-1}\left(  \sum_i \left(\mu^\lambda_{ki}-\lambda
\frac{i+1}{\pi}\mu_{k+2,i+1}^\lambda\right)\right) &&\text{ by Lemma \ref{lem:def mu}}\\
& = \frac{\omega_k k!}{\pi^k} J_\lambda^{-1} \left(\tau_{k0}^\lambda \right)-\lambda \frac{\omega_k k!}{\pi^{k+1}} J_\lambda^{-1}\left(
\tau_{k+2,1}^\lambda\right) &&\text{ by definition of } \tau_{k,q}^\lambda \\
& = { J_\lambda^{-1}\left((1-\lambda s)^{\frac{k}{2}+1} t_\lambda^k \left(1-\frac{\lambda
u_\lambda}{4}(k+2)\right)\right)} &&\text{ by Proposition \ref{prop_mu2st}} \\
& = \frac{1}{1-\lambda s} t^k\left(1-\frac{k+2}{4}\lambda u\right) &&\text{ by Proposition \ref{prop_miracle}}\\
& = t^k+\frac{\lambda}{1-\lambda s} D_1 t^k.
\end{align*}

Similarly,
\begin{align*}
H_\lambda \left(\en(t^k)\right) & = \frac{3}{4}k!\frac{\omega_{k+3}}{\pi^{k+1}} H_\lambda \left(\sum_{i=0}^{\left\lfloor
\frac{k}{2}\right\rfloor} (k-2i+1)N_{k+1,i} \right)&&\text{ by Lemma \ref{lemma_l_uk}}\\
& = -\lambda \frac{3}{4}k!\frac{\omega_{k+3}}{\pi^{k+2}} J_\lambda^{-1}\left( \sum_{i=0}^{\left\lfloor
\frac{k}{2}\right\rfloor} (k-2i+1) (i+1)\mu_{k+3,i+1}^\lambda\right)\\
& = -\lambda \frac{3}{4}k!\frac{\omega_{k+3}}{\pi^{k+2}}
J_\lambda^{-1}((k+1)\tau_{k+3,1}^\lambda-4\tau_{k+3,2}^\lambda) &&\text{ by definition of }
\tau_{k,q}^\lambda\\
& = { J_\lambda^{-1} \left(\lambda \frac{\pi}{8} (1-\lambda s)^\frac{k+1}{2} t_\lambda^{k-1} u_\lambda
\left(ku_\lambda-3(1-\lambda s)t_\lambda^2\right)\right)} &&\text{ by Proposition \ref{prop_mu2st}}\\
& = \frac{\lambda}{1-\lambda s} \frac{\pi}{8}t^{k-1}u(ku-3t^2)&&\text{ by Proposition \ref{prop_miracle}}\\
& = \frac{\lambda}{1-\lambda s} D_2 t^k.
\end{align*}

If $p=0$ in $\Val^{U(n)}$, then $\el(p)=0, \en(p)=0$ in $\curv^{U(n)}$ and hence $\frac{\lambda}{1-\lambda s}D_1p=0$ and
$\frac{\lambda}{1-\lambda s} D_2(p)=0$ in $\Val^{U(n)}$ for all $\lambda$. This implies $D_1p=0, D_2p=0$ and hence
$D_1,D_2$ are well-defined on $\Val^{U(n)}$.
\endproof

\begin{proposition} \label{prop_h0h0}
\begin{displaymath} 
(H_0' \otimes H_0') \circ K=0.
\end{displaymath}
\end{proposition}

\proof
We take first and second derivatives at $\lambda=0$ of the equation $k
\circ H_\lambda=(H_\lambda \otimes H_\lambda) \circ K$:
\begin{align*}
 k \circ H_0' & = (H_0' \otimes H_0 + H_0 \otimes H_0') \circ K\\
 k \circ H_0'' & = (H_0'' \otimes H_0 + 2 H_0' \otimes H_0'+H_0 \otimes H_0'') \circ K.
\end{align*}

Since $H_0''=2s H_0'$, we obtain by using Theorem \ref{thm_ftaig}
\begin{equation} \label{eq:second_derivatives}
 (2s H_0' \otimes H_0 + 2 H_0' \otimes H_0'+H_0 \otimes 2s H_0') \circ K = (2s H_0' \otimes H_0
+ 2s H_0 \otimes H_0') \circ K.
\end{equation}
Since $s$ commutes with $H_0'$ by Proposition \ref{prop:s
indep}, we have 
\begin{displaymath}
 (H_0 \otimes 2s H_0')\circ K = (H_0 \otimes H_0')
\circ (id \otimes 2s) \circ K=(H_0 \otimes H_0')
\circ K \circ 2s=(H_0 \otimes H_0')
\circ (2s \otimes id) \circ K=(2s H_0 \otimes H_0')\circ K,
\end{displaymath}
where we used Theorem \ref{thm_ftaig_curv}. The statement now follows from \eqref{eq:second_derivatives}.
\endproof

We may write 
\begin{align*}
K(\Delta_{00}) & =(\el\otimes\el)k(\chi)+A_1\\
K(N_{10}) & =(\en\otimes \el+\el\otimes \en) k(\chi) +A_2, 
\end{align*}
with $A_1,A_2 \in \nul\otimes\nul$. Indeed, by \eqref{eq:bark mult bis}
we have $A_1, A_2 \in  \curv \otimes \nul$, and by symmetry they must lie in $\nul\otimes\nul$. 

Proposition \ref{prop_h0h0} shows that  in order to find
$K(\Delta_{00}), K(N_{10})$  we need to solve for $A_1,A_2\in\nul\otimes\nul$ in
\begin{align} 
 (H_0'\otimes H_0') A_1 & =-(D_1 \otimes D_1) k(\chi) \label{solve_A1}\\
 (H_0'\otimes H_0') A_2 & =-(D_1\otimes D_2+D_2\otimes D_1)k(\chi). \label{solve_A2}
\end{align}
Hence, the following lemma will be useful.

\begin{lemma}
 The map $\Sigma:\Val^{U(n)} \to \nul^{U(n)}$ defined by 
\begin{align*}
\Sigma(\mu_{kq})=\begin{cases} -\frac{\pi}{q}N_{k-2,q-1} & 0<q<\frac{k}{2}\\0 & q=0,\frac{k}{2}\end{cases}
\end{align*}
is left inverse to the map $H_0'|_{\nul}:\nul^{U(n)} \to \Val^{U(n)}$. In particular, the latter
map is injective and we have 
\begin{displaymath}
 \Sigma \circ D_2=\en.
\end{displaymath}
\end{lemma}

\proof
We have 
\begin{align*}
 H_0' N_{kq} & =
\left.\frac{d}{d\lambda}\right|_{\lambda=0}J^{-1}_\lambda [N_{kq}]_\lambda \\
 & = - \left.\frac{d}{d\lambda}\right|_{\lambda=0}J^{-1}_\lambda \frac{\lambda (q+1)}{\pi}\mu_{k+2,q+1}^\lambda \\
& = - \frac{(q+1)}{\pi}\mu_{k+2,q+1},
\end{align*}
from which the statement follows.
\endproof

Recall from \cite{befu11} that 
\begin{equation}\label{global_kinematic}
 k(\chi)=\sum_{k,r} a_{nkr}\pi_{kr}\otimes\pi_{2n-k,r}
\end{equation}
where
\begin{displaymath}
\pi_{kr} =\frac{(-1)^r(2n-4r+1)!!\pi^k}{\omega_k} \sum_{i=0}^r (-1)^{i}
\frac{(2r-2i-1)!!}{(2r-2i)!(2i)!(2n-2r-2i+1)!!} t^{k-2i}u^i, 
\end{displaymath}
and
\[
 a_{nkr}=\frac{\omega_k\omega_{2n-k}}{\pi^n}\frac{(n-r)!}{8^r(2n-4r)!}\frac{(2n-2r+1)!!}{(2n-4r+1)!!}\binom{n}{2r}^{-1}.
\]

\begin{lemma} \label{lemma_sigma_d1_pi}    \mbox{ }
Define 
\begin{align}
\rho_{kr} & := \frac{2(-1)^r (2n-4r+1)!! \pi^{k-1}}{\omega_k} \Bigg(\frac{(2r-1)!!(k+1)!}{(2n-2r+1)!!(2r)!}
\sum_{i=0}^{\lfloor\frac{k-1}{2}\rfloor} \frac{(-1)^{i+1}}{(2i+3)!(k-2i-1)!}t^{k-2i-1}u^i \nonumber \\
& \quad +\sum_{i=0}^{r-1} \frac{(-1)^i(2r-2i-3)!!}{(2n-2r-2i-1)!!(2r-2i-2)!(2i+2)!}t^{k-2i-1}u^i\Bigg).
\label{eq_def_rho}
\end{align}
Then 
\begin{displaymath}
 \en(\rho_{kr})=\Sigma \circ D_1 \pi_{kr}.
\end{displaymath}
\end{lemma}

\proof
It is easily checked that for $h \in \R[t,u]$ we have $D_1(uh)=-\frac{2}{\pi}D_2(th)$, hence  
\begin{displaymath}
 \Sigma \circ D_1(uh)=-\frac{2}{\pi}\en(th).
\end{displaymath}

Let us define  
\begin{displaymath}
 \tilde \rho_k:=\frac{2(k+1)!}{\pi} \sum_{i=0}^{\lfloor\frac{k-1}{2}\rfloor}
\frac{(-1)^{i+1}}{(2i+3)!(k-2i-1)!}t^{k-2i-1}u^i \in
\R[t,u]
\end{displaymath}

From (\cite{befu11}, Proposition 3.7. and Corollary 3.8) and a straightforward computation one gets 
\begin{displaymath}
 D_2\tilde \rho_k=D_1(t^k)-\frac{\omega_{k+2}(k+2)!}{4\pi^{k+2}}\mu_{k+2,0}.
\end{displaymath}
We apply $\Sigma$ to this equation and obtain that
\begin{displaymath}
 \en(\tilde \rho_k)=\Sigma \circ D_1(t^k). 
\end{displaymath}

With 
\begin{align*}
 c & := \frac{(-1)^r(2n-4r+1)!!\pi^k}{\omega_k} \frac{(2r-1)!!}{(2r)!(2n-2r+1)!!}\\
h & := \frac{(-1)^r(2n-4r+1)!!\pi^k}{\omega_k} \sum_{i=0}^{r-1}
\frac{(-1)^{i+1}(2r-2i-3)!!}{(2n-2r-2i-1)!!(2r-2i-2)!(2i+2)!}t^{k-2i-2}u^i
\end{align*}
we compute 
\begin{align*}
 \Sigma \circ D_1\pi_{k,r} & =\Sigma \circ D_1(ct^k+uh)\\
& =  c\, \en(\tilde \rho_k)-\frac{2}{\pi}\en(th)\\
& = \en(\rho_{k,r}).
\end{align*}
\endproof

We are now ready to state one of our main theorems, the local kinematic formulas on complex space forms. 

\begin{theorem}\label{thm:local kin}
 \begin{align} \label{eq_plkf}
   K(\Delta_{00}) & =\sum a_{nkr} \left[\el(\pi_{kr}) \otimes \el(\pi_{2n-k,r})-
\en(\rho_{kr}) \otimes \en(\rho_{2n-k,r})\right].\\
   K(N_{10}) & =\sum a_{nkr} \big[\en(\pi_{kr}) \otimes \el(\pi_{2n-k,r})+\el(\pi_{kr}) \otimes \en(\pi_{2n-k,r})
\nonumber \\
& \quad -\en(\pi_{kr}) \otimes \en(\rho_{2n-k,r}) - \en(\rho_{kr}) \otimes \en(\pi_{2n-k,r})\big]  \label{eq_plkf2}.
\end{align}
\end{theorem}

\proof
Apply $\Sigma \otimes \Sigma$ to equations \eqref{solve_A1}, \eqref{solve_A2} and use \eqref{global_kinematic} and Lemma
\ref{lemma_sigma_d1_pi}.
\endproof

\subsection{$\curv^{U(n)}$ as a module over $\V^n_\lambda$}
Next we determine the $\V^n_\lambda$-module structure of $\curv^{U(n)}$. By Propositions \ref{prop:s
indep} and \ref{cor_mults}, it remains only to compute the product with $t_\lambda$.
\begin{lemma} \label{lemm_symmetries}
Let $\phi \in \Val^{U(n)}$. Then 
 \begin{align*}
 \sum a_{nkr} \phi \pi_{kr} \otimes \pi_{2n-k,r} &=
\sum a_{nkr} \pi_{kr} \otimes \phi \pi_{2n-k,r}\\
\sum a_{nkr} \en\left(\phi \rho_{kr}\right) \otimes D_2 \rho_{2n-k,r} &=
\sum a_{nkr} \en(\rho_{kr}) \otimes D_2 \left(\phi \rho_{2n-k,r}\right)\\
\sum a_{nkr}  [\en(\phi \pi_{kr}) \otimes D_2 \rho_{2n-k,r}&+ \en(\phi \rho_{kr}) \otimes D_2 \pi_{2n-k,r}]= \\
&= \sum a_{nkr}  \left[\en(\pi_{kr}) \otimes D_2 \phi \rho_{2n-k,r}+ \en(\rho_{kr}) \otimes D_2 \phi
\pi_{2n-k,r}\right].
\end{align*}
\end{lemma}

\proof
The first equation comes from the fact that $(id \otimes \phi) k(\chi)=(\phi \otimes id) k(\chi)$.  Similarly, since
$(\phi \otimes id)K(\Delta_{00})=(id \otimes \phi)K(\Delta_{00})$, we obtain 
\begin{multline*}
 \sum a_{nkr} \left[ \el\left(\phi \pi_{kr}\right) \otimes \el(\pi_{2n-k,r})-\en\left(\phi \rho_{kr}\right) \otimes
\en(\rho_{2n-k,r})\right]\\
=\sum a_{nkr} \left[ \el(\pi_{kr}) \otimes \el\left(\phi \pi_{2n-k,r}\right)-\en(\rho_{kr}) \otimes \en\left(\phi
\rho_{2n-k,r}\right)\right].
\end{multline*}

We apply $\el \otimes \el$ to the first equation and substitute this into the equation above to get 
\begin{align}
\sum a_{nkr} \en\left(\phi \rho_{kr}\right) \otimes \en(\rho_{2n-k,r}) & =\sum a_{nkr} \en(\rho_{kr}) \otimes
\en\left(\phi \rho_{2n-k,r}\right), \label{eq_symmetry_n_part}
\end{align}
from which we deduce that 
\begin{displaymath}
 \sum a_{nkr} \en\left(\phi \rho_{k,r}\right) \otimes \rho_{2n-k,r}  \equiv \sum a_{nkr} \en(\rho_{kr}) \otimes
\phi \rho_{2n-k,r} \mod \curv \otimes \ker(\en).
\end{displaymath}

Applying $id \otimes D_2$, taking into account $\ker(\en) \subset \ker D_2$, yields the second equation.  

The third equation follows in a similar way from $(\phi \otimes id)K(N_{10})=(id \otimes \phi)K(N_{10})$.
\endproof

By Theorem \ref{thm_free_module} and Proposition \ref{prop:module commute}, the $\V^n_\lambda$-module structure of
$\curv^{U(n)}$ is determined by the following equations.

\begin{theorem} \label{thm_module_cpn}
 We have 
\begin{align} 
t_\lambda \Delta_{00} & =\frac{t(1-\frac{\lambda t^2}{4})}{(1-\lambda s)^\frac32}\Delta_{00}+\frac{\lambda t^2}{2\pi
(1-\lambda s)^\frac32}N_{10} \label{mult_t_lambda_t00} \\
t_\lambda N_{10} & =-\frac18 \pi \lambda \frac{(t^2-4s)^2}{(1-\lambda s)^\frac32}\Delta_{00}+ \frac{t-2\lambda
ts+\lambda \frac{t^3}{4}}{(1-\lambda s)^\frac32}N_{10}.  \label{mult_t_lambda_n10}
\end{align}
\end{theorem}

\proof
Let $\overline m_\lambda: \curv \to \curv \otimes \mathcal{V}_\lambda^*$ be the module structure. By Corollary
\ref{coro_module_versus_kin}, $\overline m_\lambda$ is the composition of the maps 
\begin{displaymath}
 \curv \stackrel{K}{\longrightarrow} \curv \otimes \curv \stackrel{id \otimes \glob_\lambda}{\longrightarrow} \curv
\otimes \mathcal{V}_\lambda \stackrel{id \otimes \pd_\lambda}{\longrightarrow} \curv \otimes \mathcal{V}_\lambda^*.
\end{displaymath}

Since $J_\lambda=\pd_\lambda^{-1} \circ (I_\lambda^{-1})^* \circ \pd$ we have 
\begin{displaymath}
\overline m_\lambda=(id \otimes ((I_\lambda^{-1})^* \circ \pd \circ H_\lambda)) K.
\end{displaymath}

By Theorem \ref{thm:local kin} and Lemma \ref{lemma:h_lambda_series},
\begin{align*}
 t_\lambda\sqrt{1-\lambda s} \Delta_{00} & = \langle (id \otimes (I_\lambda^{-1})^* \circ \pd \circ
H_\lambda) K(\Delta_{00}),t_\lambda\sqrt{1-\lambda s}\rangle\\
 & = \langle (id \otimes \pd \circ H_\lambda) K(\Delta_{00}),I_\lambda^{-1}(t_\lambda\sqrt{1-\lambda
s})\rangle\\
& = \sum a_{nkr} \Big[\el(\pi_{kr}) \left\langle \pd \circ H_\lambda \circ
\el(\pi_{2n-k,r}),t\right\rangle\\
& \quad - \en(\rho_{kr}) \left\langle \pd \circ H_\lambda \circ \en(\rho_{2n-k,r}),t\right\rangle\Big]\\
& = \sum a_{nkr} \Big[\el(\pi_{k,r}) \left\langle \pd \circ \left(\pi_{2n-k,r}+\frac{\lambda}{1-\lambda
s}D_1 \pi_{2n-k,r}\right),t\right\rangle\\
& \quad - \en(\rho_{kr}) \left\langle \pd \frac{\lambda}{1-\lambda s} D_2 \rho_{2n-k,r},t\right\rangle\Big].
\end{align*}

Now we use Lemma \ref{lemm_symmetries} with $\phi:=\frac{1}{{1-\lambda s}}$ and obtain, recalling that
$D_1,D_2$ commute with $s$,
\begin{align*}
t_\lambda {\sqrt{1-\lambda s}}  \Delta_{00} & = \sum a_{n,k,r} \Big[\el\left(\pi_{k,r}\right)
\left\langle \pd \pi_{2n-k,r},t\right\rangle\\
& \quad +\el\left(\frac{\lambda}{1-\lambda s}\pi_{k,r}\right) \left\langle \pd \circ D_1
\pi_{2n-k,r},t\right\rangle \\
& \quad - \en\left(\frac{\lambda}{1-\lambda s}\rho_{k,r}\right) \left\langle \pd \circ D_2
\rho_{2n-k,r},t\right\rangle \Big].
\end{align*}

For degree reasons, the first summand vanishes except for $k=1,r=0$, whereas the second and third summands vanish except
for $k=3,r=0,1$. We deduce that 
\begin{displaymath}
 t_\lambda \Delta_{00}=\frac{p_1}{(1-\lambda s)^\frac32} \Delta_{00}+\frac{p_2}{(1-\lambda s)^\frac32}
N_{1,0},
\end{displaymath}
where $p_1, p_2$ are polynomials of degree $3$ and $2$ respectively.  

Similarly, 
\begin{align*}
 t_\lambda \sqrt{1-\lambda s} N_{10} & = \langle id \otimes \pd \circ H_\lambda
K(N_{10}),I_\lambda^{-1}(t_\lambda\sqrt{1-\lambda s})\rangle\\
& = \sum a_{nkr} \Big[\en(\pi_{kr}) \left\langle \pd \circ H_\lambda \circ
\el(\pi_{2n-k,r}),t\right\rangle\\
& \quad + \el(\pi_{kr}) \left\langle \pd \circ H_\lambda \circ \en(\pi_{2n-k,r}),t\right\rangle\\
& \quad - \en(\pi_{kr}) \left\langle \pd \circ H_\lambda \circ \en(\rho_{2n-k,r}),t\right\rangle\\
& \quad - \en(\rho_{kr}) \left\langle \pd \circ H_\lambda \circ \en(\pi_{2n-k,r}),t\right\rangle\Big]\\
& = \sum a_{nkr} \Big[\en(\pi_{kr}) \left\langle \pd \circ (\pi_{2n-k,r}+\frac{\lambda}{1-\lambda
s}D_1\pi_{2n-k,r}),t\right\rangle\\
& \quad + \el(\pi_{kr}) \left\langle \pd \circ \frac{\lambda}{1-\lambda s} D_2 \pi_{2n-k,r},
t\right\rangle\\
& \quad - \en(\pi_{kr}) \left\langle \pd \circ \frac{\lambda}{1-\lambda s} D_2\rho_{2n-k,r},t\right\rangle\\
& \quad - \en(\rho_{kr}) \left\langle \pd \circ \frac{\lambda}{1-\lambda s}
D_2\pi_{2n-k,r},t\right\rangle\Big].
\end{align*}

By Lemma \ref{lemm_symmetries}, with $\phi:=\frac{1}{{1-\lambda s}}$, this may be rewritten as
\begin{align*}
t_\lambda \sqrt{1-\lambda s}N_{10} & = \sum a_{nkr} \Big[\en\left({\pi_{kr}}\right) \left\langle \pd
\circ \pi_{2n-k,r},t\right\rangle\\
&\quad +\en\left(\frac{\lambda }{1-\lambda s }\pi_{kr}\right) \left\langle \pd \circ
D_1\pi_{2n-k,r},t\right\rangle\\
& \quad + \el\left(\frac{\lambda}{1-\lambda s } \pi_{kr}\right) \left\langle \pd \circ D_2 \pi_{2n-k,r},
t\right\rangle\\
& \quad -\en\left(\frac{\lambda}{1-\lambda s}\pi_{kr}\right) \left\langle \pd \circ
D_2\rho_{2n-k,r},t\right\rangle\\
& \quad - \en\left(\frac{\lambda }{1-\lambda s }\rho_{kr}\right) \left\langle \pd \circ
D_2\pi_{2n-k,r},t\right\rangle\Big].
\end{align*}

All terms vanish for $k>4$, and we obtain that 
\begin{displaymath}
 t_\lambda N_{10}=\frac{p_3}{(1-\lambda s)^\frac32} \Delta_{00}+\frac{p_4}{(1-\lambda s)^\frac32} N_{10},
\end{displaymath}
where $p_3,p_4$ are polynomials of maximal degree $4$. 

It remains to find the polynomials $p_1,p_2,p_3,p_4$. Since they are independent of $n$, we may compute them explicitly
by choosing a small value of $n$. 
Indeed, by Proposition \ref{prop_kernel_n_map} it suffices to take $n:=5$. This is a tedious, but straightforward
computation. 
\endproof

\subsection{Angularity theorem}

As an application we prove a result similar to Theorem \ref{thm_angular_flat} in the non-flat complex
space forms. 

\begin{proposition} \label{prop_decomp_angular_measures} \mbox{}
\begin{enumerate}
 \item[(i)] Define a map 
\begin{align*} 
 A:\R[[t,u]] & \to \curv^{U(\infty)}\\
 g & \mapsto \el\left(g+2u \frac{\partial g}{\partial u}\right)+\en \left(\frac{4t}{\pi}
\frac{\partial g}{\partial u}\right).
\end{align*}
Then $A$ is injective and its image is the space $\ang^{U(\infty)}$ of invariant angular measures. 
\item[(ii)] Let $p_1,p_2 \in \Val^{U(\infty)}$. 
 The measure 
\begin{displaymath}
 \Phi:=\el(p_1)+\en(p_2) \in \curv^{U(\infty)}
\end{displaymath}
is angular if and only if
\begin{equation} \label{eq_angular}
 \frac{t}{\pi} \frac{\partial p_1}{\partial s}-\frac{4s-t^2}{2} \frac{\partial p_2}{\partial s}=3p_2.
\end{equation}
\end{enumerate}
\end{proposition}

\proof
\begin{enumerate}
\item[(i)] Injectivity is easy. Let us prove that the image of $A$ is inside $\ang$. 
Clearly $A$ is compatible with multiplication by $t$. Hence, by Theorem \ref{thm_angular_flat}, it is
enough to check that $A(u^k) \in \ang$. Using Lemma \ref{lemma_l_uk} and Proposition \ref{prop_module_mult_t}
\begin{align*}
 A(u^k)
& = (2k+1)\el(u^k)+\frac{4tk}{\pi}\en(u^{k-1})\\
& = \frac{2^k (2k+1)!!}{\pi^k} \left(\Delta_{2k,k}-\frac{1}{k+1}N_{2k,k-1}\right)+\frac{4^ktk!}{\pi^k}N_{2k-1,k-1}\\
& = \frac{2^k (2k+1)!!}{\pi^k} \Delta_{2k,k}.
\end{align*}

To prove surjectivity, let $\Phi \in \ang$. Then there exists some $g \in \R[[t,u]]$ with 
\begin{displaymath} 
g+2u\frac{\partial g}{\partial u}=  2\sqrt{u} \frac{\partial}{\partial u}\left(\sqrt{u}g\right)=\glob_0(\Phi).
\end{displaymath}
Then $A(g)-\Phi \in \ang \cap \ker \glob_0 = \{0\}$, hence $\Phi=A(g)$.

\item[(ii)] In the $ts$-basis, we have
\begin{displaymath}
 A(g)=\el(p_1)+\en(p_2)
\end{displaymath}
with 
\begin{align*}
 p_1 & = g+\frac12 (4s-t^2)\frac{\partial g}{\partial s}\\
 p_2 & = \frac{t}{\pi} \frac{\partial g}{\partial s}.
\end{align*}
It is easy to check that $p_1,p_2$ satisfy \eqref{eq_angular}. 

Conversely, suppose that \eqref{eq_angular} holds. Then we have $\Phi=A(g)$ with  
\begin{displaymath} 
 g:=p_1-\frac12(4s-t^2)\frac{\pi}{t} p_2.
\end{displaymath}
\end{enumerate}
\endproof

\begin{theorem}[Angularity theorem, second version] \label{thm_angularity_thm2}
 Let $\Phi \in \ang^{U(n)}$ be an invariant angular curvature measure. Then $t_\lambda \Phi$ is angular for each
$\lambda$. 
\end{theorem}

\proof
Let $\Phi=\el(p_1)+\en(p_2)$
be angular, i.e. $p_1,p_2$ satisfy \eqref{eq_angular}. By Theorem
\ref{thm_module_cpn}, we have $t_\lambda \Phi=\el(\tilde p_1)+\en(\tilde p_2)$ with 
\begin{align*}
 \tilde p_1 & = \frac{t(1-\frac{\lambda t^2}{4})}{(1-\lambda s)^\frac32} p_1 - \frac18 \pi \lambda
\frac{(t^2-4s)^2}{(1-\lambda s)^\frac32} p_2\\
 \tilde p_2 & = \frac{\lambda t^2}{2\pi (1-\lambda s)^\frac32} p_1+\frac{t-2\lambda ts+\lambda \frac{t^3}{4}}{(1-\lambda
s)^\frac32}p_2.
\end{align*}

It is now straightforward (but tedious) to check that $\tilde p_1,\tilde p_2$ satisfy \eqref{eq_angular}. Hence
$t_\lambda \Phi$ is angular.
\endproof


\section{Concluding remarks, questions and conjectures}\label{sect:questions}
\subsection{Numerical identities and relations in $\V^n_\lambda$}
It {came as} a surprise that the algebras $\mathcal{V}^n_\lambda$ for fixed $n$ and varying $\lambda$ are all isomorphic as
filtered algebras. Before the discovery of Theorem \ref{thm:1st iso} and Theorem \ref{thm_other_isom}, the first
conjecture about the structure of $\mathcal{V}^n_\lambda, \lambda \ne 0$ was the following.

\begin{conjecture} \label{cx space form conj} Define the formal series $\bar f_k(s,t,\lambda)$ by
\begin{align}
\notag\log (1 +& sx^2 + tx + \lambda x^{-2} + 3\lambda^2x^{-4} + 13 \lambda^3x^{-6} +...) \\
&= \log (1 + s x^2 + tx 
\label{eq:crazy expansion}+ \sum \left[\binom{4n+1}{n+1} - 9 \binom{4n+1}{n-1}\right] \lambda^nx^{-2n})\\
&= \sum_k \bar f_k(s,t,\lambda) x^k.
\end{align}
Then 
\begin{equation}
\mathcal{V}^n_\lambda\simeq \R[s,t]/(\bar f_{n+1},\bar f_{n+2}, t^{2n+1}, st^{2n-1},\dots, s^nt).
\end{equation}
\end{conjecture}

This statement can be checked numerically in any given dimension $n$ by explicit calculation--- the most efficient way
seems to be via Corollary \ref{cor_other_isoms}, by which means we have confirmed it through $n=70$. This case involves
the first $34$ terms of the $\lambda$ series.

A different approach, less efficient but more elementary, is the following. First, it is sufficient to show that 
\begin{equation}\label{eq:fbar = 0}
\bar f_{n+1}(s,t,\lambda)=0
\end{equation}
as elements of $\mathcal{V}^n_\lambda$.  We know that there are no relations between $s,t\in \mathcal{V}^n_\lambda$ of
weighted degree $\le n$. By Alesker Poincar\'e duality, it follows that for $\lambda >0$ the relation $\bar f_{n+1}(s,t)
= 0$ holds in $\mathcal{V}^n_\lambda$ iff 
\begin{equation}
(\bar f_{n+1}(s,t) \cdot s^j t^{k})(\CP^n_\lambda) = 0, \quad 2j + k\le n.
\end{equation}
For given $n$, this can be checked directly using Corollary \ref{cor:t^k(cpn)}. The relation \eqref{eq:fbar = 0} may
then be extended to $\lambda \le 0$ by analytic continuation.

For example, modulo filtration $n+3$ the conjecture is equivalent to the family of identities
\begin{equation}\label{special pfaff-saalschutz}
\sum_{i=0}^{\lfloor\frac {n+1} 2\rfloor} \frac{(-1)^i }{n+1-i}\binom{n+1-i}{i} \binom{2n-2k-2i}{n-k-i} = 0, \quad k \le
\frac n 2.
\end{equation}
As kindly pointed out to us by I. Gessel, these identities are well known--- in fact the sum is equal to $\frac
{(-1)^{n-k}} {n+1} \binom k{n-k}$ for general $k$. At this level, the statement is independent of the curvature
$\lambda$, and yields a very efficient proof of the structure theorem for $\mathcal{V}^n_0$ given in  \cite{fu06}. 
The validity of the conjecture modulo filtration $n+5$ is equivalent to the family of identities
\begin{align*}
\sum_{j=0}^{\lfloor \frac{n+3} 2\rfloor} (-1)^{j+1}& \frac 1{n-j+4}\binom{n-j+4}{1,j,n-2j+3}\binom{2n-2k-2j}{n-k-j}\\
+ &\sum_{i=0}^{\lfloor \frac{n+1} 2\rfloor} (-1)^i \frac{n-k-i+1}{n-i+1}\binom{n-i+1}{i}\binom{2n-2k-2i-2}{n-k-i-1} =
0,\\
& k=0,\dots \left\lfloor \frac{n-3}{2}\right\rfloor.
\end{align*}
Gessel  has given a proof of these identities as well. In principle one could continue in this way, but the numerical
identities that arise become unreasonably complicated.

As shown by F. Chapoton \cite{chapoton}, the series appearing in \eqref{eq:crazy expansion} also arises as the
generating function for a certain counting problem. Chapoton also proved that this series  defines an algebraic function
$g(\lambda)$ satisfying
 $$
 g= f(1-f-f^2), \quad f= \lambda(1+f)^4.
 $$
 This fact was independently observed by Gessel. 

 It can also be amusing to write down the numerical identities that are equivalent in this way to Theorem \ref{thm:1st
iso} and Corollary \ref{cor_other_isoms}. 
This is a good illustration of the power (and clumsiness) of the template method.

\subsection{A Riemannian geometry question} We expect that the results above will shed light on the integral geometry of
general Riemannian manifolds, in particular the geometric meaning of the Lipschitz-Killing valuations and curvature
measures. From this perspective Theorem \ref{thm_angularity_thm2} is particularly striking and suggests the following.

\begin{conjecture}\label{conj:angularity conjecture} Let $M$ be a smooth Riemannian manifold,
with Lipschitz-Killing subalgebra $LK(M) \subset \mathcal{V}(M)$. Let
$\mathcal A(M)$ denote the vector space of angular curvature measures on $M$. Then
$$LK(M)\cdot \mathcal A(M) \subset  \mathcal A(M) .$$
\end{conjecture}
 Note that even for $M=\CP^n_\lambda$ this conjecture is stronger than Theorem \ref{thm_angularity_thm2}.

Let us call a smooth valuation on $M$ {\it angular} if it stabilizes $\mathcal A(M)$. Clearly, this is a subalgebra of
$\mathcal{V}(M)$. The above conjecture states that this algebra contains $LK(M)$. A stronger form of this conjecture is

\begin{conjecture}\label{conj:last}
 The algebra of angular valuations on $M$ equals   $LK(M)$. 
\end{conjecture}

Using Proposition \ref{prop_decomp_angular_measures} it is easy to check that the only translation-invariant and
$U(n)$-invariant angular valuations on $\C^n$ are indeed polynomials in $t$. 

\subsection{An algebra question} We may think of each $\V^n_\lambda$ as a commutative algebra of linear operators on
$\curvun$. By the proof of Proposition \ref{prop:module commute}, all of these operators commute with each other.
What is the structure of the resulting commutative filtered algebra, generated by $\bigcup_{\lambda \in
\R}\mathcal{V}^n_\lambda$?
   

\begin{appendix}
\section{A product formula for valuations in isotropic spaces}
This appendix is devoted to establishing the following. 
\begin{theorem} \label{thm:main appendix}Let $(M,G)$ be an isotropic space,  $Q\in \mathcal P(M)$  a  simple compact
differentiable polyhedron, and $\rho \in C^\infty(G)$. Then
$$
\psi(P):= \int_G \chi(gP \cap Q) \, \rho(g) \, dg 
$$
defines a smooth valuation on $M$.
If $\phi \in \V^\infty(M)$ then the Alesker product of $\phi$ and $\psi$ is given by
\begin{equation}\label{eq:G asymmetric product}
(\phi\cdot \psi) (P)= \int_G \phi (gP \cap Q) \, \rho(g)\, dg.
\end{equation}
\end{theorem}

Since the  map $g \mapsto g^{-1}$ is smooth and preserves the Haar measure on $G$, this implies Proposition \ref{general_Hadwiger_thm}.
The rest of the appendix is devoted to the proof. We will need some facts arising from the slicing procedure in the proof of Theorem \ref{thm:basic K}.

\begin{lemma}\label{lem:slice} Let $(M,G)$ be an isotropic space with $\dim M =n$, and $P,Q \in \mathcal P(M)$.
\begin{enumerate} 
\item Given $\beta \in \Omega^{n-1}(SM)$, the function
$
g\mapsto  \int_{N(gP\cap Q)} \beta
$
is integrable.
\item There is a constant $C<\infty$, depending only on $(M,G)$ and independent of $P,Q$, such that
\begin{equation*}
\int_G \mass(N(gP\cap Q))\,dg \le C \left(\mass N(P) \mass N(Q) + \mass N(P) \vol Q + \vol P \mass N(Q) \right).
\end{equation*}
\end{enumerate}
\end{lemma}
\begin{proof}
Both assertions follow at once from \eqref{eq:mass T}, \eqref{eq:slicing procedure} and Theorem 4.3.2 of \cite{gmt}.
\end{proof}

\begin{lemma}\label{lem:mathcal J} Let $Q \in \mathcal P(M)$.
 There is a map $\mathcal J=( \mathcal J_1,\mathcal J_2): \Omega^{n-1} (SM) \times C^\infty(G) \to \Omega^n(M)\times\Omega^{n-1}(SM) $, jointly continuous with respect to the $C^\infty$ topologies, such that for every $P \in \mathcal P(M)$
\begin{displaymath}
\int_G \rho(g) \left(\int_{N(gP\cap Q)} \beta \right) \, dg = \int_{P} \mathcal J_1(\beta,\rho)  \ + \int_{N(P)} \mathcal J_2 (\beta,\rho).
\end{displaymath}
Similarly, there is a map $\mathcal L: \Omega^{n} (M) \times C^\infty(G) \to \Omega^n(M)$, jointly continuous with respect to the $C^\infty$ topologies, such that for every $P \in \mathcal P(M)$
\begin{displaymath}
\int_G \rho(g) \left(\int_{gP\cap Q} \gamma \right) \, dg = \int_{P} \mathcal L(\gamma,\rho).
\end{displaymath}
\end{lemma}

\begin{proof} We recall the constructions and notation of Theorem \ref{thm:basic K}.
By  \eqref{eq:decomp T}, \eqref{eq:slicing procedure}, and Theorem 4.3.2 of \cite{gmt}, 
\begin{align}
\notag\int_G \rho(g) \left(\int_{N(gP\cap Q)} \beta \right) \, dg  &= \int_{T_0(P,Q)}  p^*(\pi_{SM}^* \beta \wedge \pi_G^*(\rho \, dg))  \\
\label{eq:slice decomp} &+   \int_{T_1(P,Q)} p_1^*(\pi_{SM}^* \beta \wedge \pi_G^*(\rho \, dg)) \\
\notag& +  \int_{T_2(P,Q)}p_2^*(\pi_{SM}^* \beta \wedge \pi_G^*(\rho \, dg)) 
\end{align}
Fixing $Q$, we show that the first and third integrals on the right may be written as $\int_{N(P)} \mathcal J_2 (\beta,\rho)$, while the second may be written as $\int_{P} \mathcal J_1(\beta,\rho) $.

The first of the integrals on the right may be written as
$$
\int_{N(P) \times N(Q)} \bar H (\beta, \rho) 
$$
where $\bar H (\beta,\rho) := \pi_{\CC*} (\iota \circ p)^*(\beta \wedge \rho\, dg) \in \Omega^*(SM\times SM) $, extending the map $H$ of \eqref{eq:def H} to the non-invariant case. Since the bundle map $E \to SM \times SM$ of \eqref{basic diagram} is a smooth fibration, it follows that $\bar H$ is continuous as a map $ \Omega^{n-1} (SM) \times C^\infty(G) \to \Omega^*(SM \times SM)$. Composing $\bar H$ with the fiber integral over $N(Q)$, it follows that 
$$
\mathcal J_2' := \pi_{N(Q)*} \circ \bar H
$$
is a $C^\infty$-continuous map $ \Omega^{n-1} (SM) \times C^\infty(G)  \to \Omega^{n-1}(SM)$, such that for any $P \in \mathcal P(M)$ the first term on the right of \eqref{eq:slice decomp} equals $\int_{N(P) }\mathcal J_2'(\beta,\rho)$.

The third term may be written in similar, but simpler, fashion, as $\int_{N(P) }\mathcal J_2''(\beta,\rho)$, where
$$
\mathcal J_2'' (\beta,\rho) := \pi_{Q*} \circ \pi_{G_o*}(p_2^*(\pi_{SM}^* \beta \wedge \pi_G^*(\rho \, dg)))
$$
where $\pi_{G_o*}$ is fiber integration over the fiber of the bundle $\Gamma_2 \to SM \times M$ of \eqref{eq:def gamma bundles}.

Now put $\mathcal J_2:= \mathcal J_2' +  \mathcal J''_2 $. The map $\mathcal J_1$ is defined similarly as
$$
\mathcal J_1(\beta,\rho) := \pi_{N(Q)*} \circ \pi_{G_o*}(p_1^*(\pi_{SM}^* \beta \wedge \pi_G^*(\rho \, dg)))
$$
with respect to the bundle $\Gamma_1$. Finally, the map $\mathcal L$ may be constructed by a similar procedure.
\end{proof}

\begin{corollary}\label{cor:integral is valuation} For every $\phi \in \V^\infty(M)$, the integral on the right hand side of \eqref{eq:G asymmetric product} defines a smooth valuation.
\end{corollary}
\begin{proof} This follows at once from Lemma \ref{lem:mathcal J}.
\end{proof}

\begin{proposition}\label{prop:cts prod} For fixed $\rho,Q$, the assignment
$$
\mathcal I:\phi \mapsto \int_G \phi (\cdot \cap gQ) \, \rho(g)\, dg
$$
gives a continuous map $\V^\infty(M) \to \V^\infty(M)$.
\end{proposition}
\begin{proof} Since the topology on $\V^\infty(M)$ is induced by the projection $\Glob:\Omega^{n-1}(SM) \times \Omega^n(M) \to \V^\infty(M)$, this follows at once from Lemma \ref{lem:mathcal J}.
\end{proof}

By \cite{ale05b}, Proposition 2.1.17, $Q$ admits a simple subdivision 
 $ Q = \bigcup_{i=1}^NQ_i$, where each $Q_i$ is contained in a coordinate neighborhood of $M$. Using the
inclusion-exclusion principle we may therefore assume that $Q$ itself has this property, and using a partition
of unity we may also assume that there is a coordinate neighborhood $U\subset M$ such that if $g\in G$ with
$\rho(g) \ne 0$ then $gQ \subset U$. Finally, since the assignment $U \mapsto \V^\infty(U)$ is a sheaf (\cite{ale05b},
Theorem 2.4.10), it is enough to show that
$\phi\cdot \psi (P)= \int_G \phi (P \cap gQ) \, \rho(g)\, dg$ for $P \in \mathcal P(U)$.

Let $F: U \to \Rn$ be the coordinate map. By \cite{alefu05} and \cite{ale05a}, Thm.~5.2.2, the topological algebra of
smooth valuations on $M$ supported in $U$ is isomorphic to the topological algebra $SV_{F(U)}$ of smooth convex valuations on
$\Rn$ supported in $F(U)$, where the map $\theta \mapsto \tilde \theta$ from $\V^\infty$ to $SV_{F(U)}$ is given by
\begin{equation}\label{eq:V to SV}
\tilde \theta(K) := \theta(F^{-1}(K)), \quad K \in \Ksm.
\end{equation}
The inverse map $\tilde \theta \mapsto \theta$ may be described as follows. It is shown in \cite{ale05a},
Thm.~5.2.2, that for each $\tilde \theta\in SV= SV_{\Rn}$ there exist smooth differential forms $\omega \in
\Omega^{n-1}(S\Rn ), \phi \in \Omega^n(\Rn)$ such that 
$$
\tilde \theta (K) = \int_{N(K)} \omega + \int_K \phi.
$$
The value at a smooth polyhedron $P \in \mathcal P(U)$ of the corresponding smooth valuation on $M$ is then
obtained by integrating these differential forms over $N(F(P)), F(P)$ respectively.

Moreover, it follows from Corollary 3.1.7 of  \cite{ale05a} that 
 the span of all convex valuations on $\R^n$ of the form 
\begin{equation} \label{eq_element_sv}
 \tilde \phi(K) = \left.\frac{\partial^k}{\partial \lambda_1 \cdots \partial \lambda_k}\right|_{\lambda=0}
\mu\left(K+\sum_{i=1}^k \lambda_i A_i\right),
\end{equation}
where $k$ is less than or equal to $n$, $A_1,\ldots,A_k\in \Ksm$ and $\mu$ is a smooth measure with compact support in  $U$, is dense in $SV_{F(U)}$.
Since, by Theorem 4.1.2 of \cite{ale05a}, the Alesker product is bilinear and continuous in the two factors, it therefore follows from Lemma \ref{lem:mathcal J} that we only need to prove
\eqref{eq:G asymmetric product} for all $\phi\in \V^\infty(M)$ that are supported in $U$ and whose image in $SV_{F(U)}$ has the form \eqref{eq_element_sv}.

\begin{lemma}\label{lem:sv product} Let $\tilde \phi ,\tilde \psi \in SV_{F(U)}$, with $\tilde \phi$ given by
\eqref{eq_element_sv}.  Then the Alesker product  $\tilde \phi\cdot \tilde \psi\in SV_{F(U)}$ is given by
\begin{displaymath}
 (\tilde \phi \cdot \tilde \psi)(K)
= \left.\frac{\partial^k}{\partial \lambda_1 \cdots \partial
\lambda_k}\right|_{\lambda=0} \int_{\R^n} \tilde \psi\left(K \cap (x-\sum_{i=1}^k \lambda_i
A_i)\right)d\mu(x), \quad K \in \Ksm.
\end{displaymath}
\end{lemma}

\begin{proof} This follows from relation (3) from \cite{ale05a}; cf. also the relation (68) from \cite{alefu05}.
\end{proof}

Recall that if $f:M_1\to M_2$ is a diffeomorphism between Riemannian manifolds $M_i$ then there is a natural lift
$\tilde f:SM_1 \to SM_2$ such that $\tilde f_* N(P) = N(f(P))$ for every $P \in \mathcal P(M_1)$, given as follows. Let
$g_i: SM_i \to S^*M_i, i = 1,2$, denote the diffeomorphism between the sphere and cosphere bundles of $M_i$, induced by
the Riemannian metric. Let $f^{-*}:S^*M_1 \to S^*M_2$ 
denote the diffeomorphism of the
cosphere bundles induced by pullback under $f^{-1}$. Then 
$$
\tilde f = g_2^{-1}\circ f^{-*}\circ g_1.
$$

\begin{lemma}\label{lem:small intersections} Given $A_1,\dots,A_k \in \Ksm(\Rn)$ and a simple differentiable polyhedron $P
\subset \Rn$, there are constants $\eps >0$ and $C<\infty$ such that 
 if $0\le \lambda_i <\eps, \ i = 1, \dots, k$ then
 \begin{itemize}
 \item $\mass N(P\cap (x- \sum \lambda_i A_i)) \le C$ for all $x \in \Rn$
 \item the map $x \mapsto N(P\cap (x- \sum \lambda_i A_i))$ is continuous at every point $x $ for which
 $P, (x- \sum \lambda_i A_i)$ meet transversely.
 \end{itemize}
\end{lemma}

\begin{proof} Since every such $P$ is locally diffeomorphic to some open subset of $\R_+^{ k}:=\{(x_1,\dots,x_{ k}): x_i
\ge 0, \ i = 1,\dots, { k}\}$, we may cover $P$ by finitely many open sets $U_1,\dots,U_N$ with the property that there exist open subsets $V_j \subset \R^n$ contained in some ball of radius $R<\infty$ and diffeomorphisms $f_j:U_j \to
V_j$ such that $f_j(U_j\cap P)$ is convex. 
We may assume also that the inverse maps to
the lifts $\tilde f_j:SU_j \to SV_j$ all have $C^1$ norms less than some given constant.

 Let $\eps $ be small enough that whenever $\lambda_1,\dots,\lambda_k < \eps$
\begin{itemize}
\item if $(x-\sum \lambda_iA_i) \cap P \ne \emptyset$ then $x-\sum \lambda_iA_i\subset U_j$ for some $j$, and
\item if $x-\sum \lambda_iA_i\subset U_j $ then $f_j(x-\sum \lambda_iA_i)$ is convex.
\end{itemize}
The second stipulation is possible in view of the bound on the $C^1$ norms of the $\tilde f_j$, since all principal curvatures of $\sum \lambda_i K_i$ tend to $\infty$ as the $\lambda_i \to 0$.
It follows that $f_j(P\cap (x-\sum \lambda_iA_i))= f_j(P\cap U_j ) \cap f_j(x-\sum \lambda_iA_i)$ is convex for such
$j$. Since this convex set is a subset of the ball of radius $R$, its normal cycle has mass $\le C' = C'(R)$. 

Furthermore, if  $P$ and $x_0- \sum \lambda_i A_i$ meet transversely then the same is true of their images under
$f_j$. It follows that $x \mapsto f_j(P )\cap f_j(x- \sum \lambda_i A_i)$ is continuous at $x_0$ as a map $\Rn\to \K(\Rn)$.
Since the normal cycle map $N:\K \to \mathbb I_{n-1}(S\Rn)$ is continuous, the map
$x \mapsto N(f_j(P )\cap f_j(x- \sum \lambda_i A_i))$ is continuous at $x_0$ as well.

Finally, the normal cycle of $P \cap (x-\sum \lambda_iA_i)$ is the image of $N(f_j(P\cap (x-\sum \lambda_iA_i)))$ under
the lift $\tilde f^{-1}_j$.  Since the $C^1$ norm of this map is bounded, the first conclusion follows. Since the
pushforward map  $(\tilde f^{-1}_j)_*:\mathbb I_{n-1}(SV_j) \to \mathbb{I}_{n-1}(SU_j)$ between the corresponding
groups of integral currents is continuous, the second conclusion follows also.
\end{proof}

\begin{lemma}\label{lem:pullback chi} If $\eta $ is a smooth measure on $\Rn$, $R\subset \Rn$ is a  compact smooth
polyhedron, and $B\in \Ksm$ then
\begin{equation}\label{eq:pushforward}
\int_{\Rn} \chi(R \cap (x-B))\, d\eta(x) = \eta(R)+ \int_{N(R)\times [0,1]} H_B^*d\eta
\end{equation}
where $H_B:S\Rn \times \R \to \Rn$ is given by 
$$
H_B(x,v;s):= x + s \nabla h_B(v).
$$
\end{lemma}
\begin{proof}  Since $N(R)$ is contained in a finite union of smooth $(n-1)$-dimensional  submanifolds of $S\Rn$, Sard's theorem implies that almost every point of $\Rn$ is a regular value of $\restrict{H_B}{N(R)\times \R}$. The area formula implies that the integral on the right-hand side of \eqref{eq:pushforward}  may be expressed as
$$
\int_{N(R)\times [0,1]} H_B^*d\eta = \int_{\Rn} \sum_{H_B(x,v;s) = y} \sgn \det D(\restrict{H_B}{N(R)\times \R}(x,v;s)) \, d \eta(y).
$$

Now let us assume, as we may, that $B$ contains the origin in its interior. Thus for each  $y\in \Rn $ we may define the function
$$
b_y(x):= \min\{t: x \in y-tB\},
$$
which is a smooth function away from $y$. Recalling that the smooth polyhedron $R$ has positive reach, 
let us assume for the moment that the critical points of $\restrict {b_y} R$ are all nondegenerate, and put $\iota_y(x,v)$ for the Morse index of this function at a critical point $(x,v) \in N(R)$, in the sense of \cite{fu89}. It follows that under these assumptions
$$
\chi(R\cap (y-B))= \sum_{(y,v) } (-1)^{\iota_{(x,v)}}{ +\mathbf{1}_R(y)}.
$$
where the sum ranges over all critical points $(y,v) \in N(R)$ of $b_x$ with $b_x(y) \le 1$.

We observe that  $(x,v) \in N(R), x\ne y,$ is a critical point of $\restrict{b_y}R$ in the sense of \cite{fu89} iff $y = x + s\nabla h_B(v), \ s = b_y(x)$. Furthermore, $(x,v)$ is a nondegenerate critical point of $b_y$ iff the Jacobian of $\restrict{H_B}{N(R)\times \R}$ at $(x,v;s)$  is itself nondegenerate. In particular, the nondegeneracy assumption of the last paragraph holds for a.e. $y \in \Rn$.  Finally, the sign of the Jacobian determinant of $\restrict{H_B}{N(R)\times \R}$ at $(x,v;s)$ is precisely
$(-1)^{\iota_y(x,v)}$. The lemma follows.
\end{proof}

Now we can complete the proof of the theorem. We may assume that $\phi\in \V^\infty(M)$ corresponds to $\tilde
\phi$ as in \eqref{eq_element_sv}  with $k>0$ (the case $k=0$ can be easily treated separately). Consider for $P
\in \mathcal P(U)$ the expression
\begin{align}\label{eq:def theta}
\theta(P)&:= \int_{\R^n} \psi\left(P \cap  \left(x-\sum_{i=1}^k \lambda_i A_i\right)\right)d{ \mu}(x) \\
\notag&= \int_{\R^n}\int_G \chi\left(P \cap  \left(x-\sum_{i=1}^k \lambda_i A_i\right)\cap gQ\right) \rho(g)\,dg\, d{
\mu}(x)\\
\notag&=\int_{\R^n}\int_G \left(\int_{N(P \cap  (x-\sum_{i=1}^k \lambda_i A_i)\cap gQ)} \kappa_0\right) \rho(g) \,
dg\, d{ \mu}(x).
\end{align}
Here $\kappa_0\in \Omega^{n-1}(S\Rn)$ is the pullback to $S\Rn \simeq \Rn \times S^{n-1}$ of the invariant volume
form on the fiber $S^{n-1}$ with integral $1$, so that integration of $\kappa_0$ over the normal cycle yields the Euler
characteristic.

We wish to show that this is a well-defined finite quantity, and that we may interchange the outer two integrals. By
Tonelli's theorem, it is enough to show that the innermost integral defines a measurable function of $(x,g)$ and that
the iterated integral of this function is absolutely convergent. 
That it is measurable follows from Lemma \ref{lem:slice} (1).
That it is absolutely integrable follows from  Lemma \ref{lem:slice} (2) and Lemma \ref{lem:small intersections}.

Thus $\theta(P)$ is indeed well defined, and
\begin{align*}
\theta(P)&=\int_G \int_{\R^n} \chi\left(P\cap gQ \cap  \left(x-\sum_{i=1}^k \lambda_i A_i\right)\right) \, d{ \mu}(x)
\rho(g)\, dg\\
&= \int_G  \left({ \mu(P\cap gQ)+}\int_{N(P\cap gQ)\times [0,1]} H_{\mathbf\lambda}^*d{ \mu} \right)\rho(g)\, dg
\end{align*}
by Lemma \ref{lem:pullback chi}, where
 $$H_{\mathbf\lambda}(x,v;s):= x + s\sum \lambda_i \nabla h_{A_i}(v).
 $$
 Using  Lemma \ref{lem:slice} (2), dominated
convergence implies that
\begin{equation*}
\delta(P):=\manyderivs\theta(P)
= \int_G  \left(\int_{N(P\cap gQ)\times [0,1]}\left( \manyderivs H_{\mathbf\lambda}^*d{ \mu}\right) \right)\rho(g)\, dg.
\end{equation*}

In particular, by Lemma \ref{lem:mathcal J}, the functional $\delta$ is
a compactly supported smooth valuation on  $U$.  Using the expression \eqref{eq:def theta}, Lemma \ref{lem:sv
product} implies that the image  $\tilde \delta \in SV_U$ agrees with $\tilde \phi \cdot \tilde \psi$.  On the other
hand, 
\[
 \phi =\int_{N(\cdot)\times[0,1]}\left( \manyderivs H_{\mathbf\lambda}^*d\mu\right).
\]
Indeed, both sides are smooth valuations and their images in $SV$ coincide, by Lemma \ref{lem:pullback chi} and dominated convergence. Since the map $\V^\infty(U) \to SV_{F(U)}$ of \eqref{eq:V to SV} is an isomorphism of algebras, we
deduce that $\delta=\phi\cdot\psi$. This completes the proof of Theorem \ref{thm:main appendix}.
\end{appendix}


\end{document}